\documentclass[11pt]{amsart}
\usepackage{amsmath,amsthm,amscd,amsfonts,amssymb,epic,eepic,bbm, graphicx, youngtab,txfonts, ytableau}
\usepackage{tikz,color,graphicx,tikz-cd,xcolor,tkz-euclide}
\usepackage{hyperref}
\numberwithin{equation}{section}
\allowdisplaybreaks

\catcode`\@=11
\@namedef{subjclassname@2020}{%
  \textup{2020} Mathematics Subject Classification}
\catcode`\@=12

\newtheorem{thm}{Theorem}[section]
\newtheorem{cor}[thm]{Corollary}

\newtheorem{lem}[thm]{Lemma}

\newtheorem{prop}[thm]{Proposition}

\theoremstyle{definition}
\newtheorem{defn}[thm]{Definition}

\newtheorem{problem}[thm]{Problem}

\theoremstyle{remark}
\newtheorem{rem}[thm]{Remark}
\newtheorem{rems}[thm]{Remarks}
\newtheorem{example}[thm]{Example}

\input pathlate.sty

\newcommand\UD{\operatorname{UD}}
\newcommand\GD{\operatorname{UDP}}

\newcommand\DP{\operatorname{DP}}
\newcommand\Mot{\operatorname{Mot}}
\newcommand\VT{\operatorname{VT}}
\newcommand\Tr{\operatorname{tr}}

\newcommand\Par{\operatorname{Par}}

\newcommand\NNest{\operatorname{NN}}

\newcommand\CRST{\operatorname{CRST}}
\newcommand\CSSYT{\operatorname{CSSYT}}

\newcommand{\la}{\lambda}
\newcommand{\ka}{\kappa}

\newcommand{\SYT}{\operatorname{SYT}}
\newcommand{\CSYT}{\operatorname{CSYT}}

\newcommand{\NCNN}{\operatorname{NCNN}}
\newcommand{\NC}{\operatorname{NC}}

\newcommand{\RR}{\mathbb{R}}

\renewcommand\vec[1]{\mathbf{#1}}
\newcommand\vx{\mathbf{x}}

\DeclareMathOperator*{\Pf}{Pf}

\newcommand\flr[1]{\lfloor  #1 \rfloor}

\def\al{\alpha}
\def\be{\beta}
\def\ep{\varepsilon}
\def\ph{\varphi}
\def\rh{\rho}
\def\si{\sigma}
\def\Z{\mathbb Z}
\def\sgn{\operatorname{sgn}}
\def\coef#1{\left\langle#1\right\rangle}
\def\fl#1{\lfloor #1\rfloor}

\begin{document}

\title[Bounded Littlewood identities]{Bounded Littlewood identities for cylindric Schur functions}

\author{JiSun Huh}
\address{Department of Mathematics, University of Seoul, Seoul, Republic of Korea}
\email{hyunyjia@yonsei.ac.kr}

\author{Jang Soo Kim}
\address{Department of Mathematics, Sungkyunkwan University, Suwon, Republic of Korea}
\email{jangsookim@skku.edu}

\author{Christian Krattenthaler}
\address{Fakult\"at f\"ur Mathematik, Universit\"at Wien,
Oskar-Morgenstern-Platz~1, A-1090 Vienna, Austria.}
\email{Christian.Krattenthaler@univie.ac.at}

\author{Soichi Okada}
\address{Graduate School of Mathematics, Nagoya University,
Furo-cho, Chikusa-ku, Nagoya 464-8602, Japan}
\email{okada@math.nagoya-u.ac.jp}


\thanks{}

\begin{abstract}
The identities which are in the literature often called
``bounded Littlewood identities" are determinantal formulas for the sum of
Schur functions indexed by partitions with bounded height. They have interesting
combinatorial consequences such as connections between standard Young tableaux
of bounded height, lattice walks in a Weyl chamber, and noncrossing matchings.
In this paper we prove affine analogs of the bounded Littlewood
identities. These are determinantal formulas for sums of cylindric Schur
functions. We also study combinatorial aspects of these identities. As a
consequence we obtain an unexpected connection between cylindric standard
Young tableaux and \( r \)-noncrossing and \( s \)-nonnesting matchings.
\end{abstract}

\maketitle



\section{Introduction}

Schur functions \( s_\lambda(\vx) \), where \( \lambda \) is a partition and
$\vx=\{x_1,x_2,\dots\}$ is a set of variables, are an
extensively studied class of symmetric functions which play an important role in
many areas of mathematics including geometry, representation theory and
combinatorics.
(We refer the reader to~\cite{Macdonald} or~\cite[Ch.~7]{EC2}, respectively to
Subsection~\ref{sec:symm} for background on symmetric functions.)
They form a basis of the space of symmetric functions, and the
Jacobi--Trudi formulas
\begin{equation}
  \label{eq:JT}
  s_{\lambda}(\vx) = \det\big(h_{\lambda_i-i+j}(\vx)\big), \qquad
  s_{\lambda'}(\vx) = \det\big(e_{\lambda_i-i+j}(\vx)\big)
\end{equation}
give a way to express Schur functions
in terms of complete homogeneous symmetric functions \( h_k(\vx) \) and
elementary symmetric functions \( e_k(\vx) \), where \( \lambda' \) denotes the
conjugate of the partition \( \lambda \).

It is well known~\cite[Cor.~7.13.8]{EC2} that the sum of all Schur
functions has a simple product formula, namely
\[
  \sum_{\lambda} s_{\lambda}(\vx) =
  \frac{ 1 } { \prod_i ( 1 - x_i  ) \prod_{i<j} ( 1 - x_i x_j ) }.
\]
This identity can be found in Littlewood's book~\cite[p.~238]{Littlewood},
together with other identities of similar kind, where the summation index~$\la$ is
subject to certain restrictions. Altogether, these identities are commonly 
called ``Littlewood identities". The above identity however was stated earlier
in an equivalent form by Schur~\cite[p.~456]{Schur}. 

The term ``bounded Littlewood identity" is used for summation identities in which
the number of parts of the summation index~$\la$ (a partition), or the first part of~$\la$,
is bounded from above by a fixed positive integer. Motivated by
enumeration problems for plane partitions and tableaux, such identities were found
in the 1970s and 1980s. We state two prototypical ``bounded Littlewood
identities" in the theorem below. It is difficult to give a precise
attribution. In the form below, they have been first stated by
Stembridge~\cite[Th.~7.1]{Stembridge1990} (modulo the symmetric
function involution interchanging complete homogeneous and elementary
symmetric functions). However, due to known
properties of symplectic and orthogonal characters, they appear
in an equivalent form in Macdonald's book~\cite[Ch.~I, Sec.~5, Ex.~16]{Macdonald}.
To make the situation even more confusing, Stembridge makes it clear
that all the ingredients to prove the identities are already contained
in determinantal-Pfaffian formulas due to Gordon and Houten,
namely~\cite[Lemma~1]{Gordon2} and~\cite[Lemma~1]{Gordon5}.
We may restate~\cite[Th.~7.1]{Stembridge1990} as follows.

\begin{thm}[\sc Two bounded Littlewood identities]
  \label{thm:BK-intro}
  For a nonnegative integer \( h \), we have
  \begin{equation}
    \label{eq:BK_odd1}
    \sum_{\la:\ell(\lambda)\le 2h+1} s_{\lambda'}(\vx)
    =
    \sum_{k \ge 0} e_k(\vx)
    \det_{1 \le i,j \le h}
    \left(
      f_{-i+j}(\vx) - f_{i+j}(\vx)
    \right)
\end{equation}
and
\begin{equation}
    \label{eq:BK_even1}
    \sum_{\la:\ell(\lambda)\le 2h} s_{\lambda'}(\vx)
    =
    \det_{1 \le i,j \le h}
    \left(
      f_{-i+j}(\vx) + f_{i+j-1}(\vx)
    \right),
  \end{equation}
  where
  \begin{equation}\label{eq:f_k}
    f_r(\vx)=\sum_{n\in \Z}e_n(\vx) e_{n+r}(\vx).
  \end{equation}
\end{thm}

We remark that the right-hand sides
of~\eqref{eq:BK_odd1} and~\eqref{eq:BK_even1} are (essentially)
irreducible characters indexed by a rectangular shape of the 
odd orthogonal group $\text{SO}_{2n+1}(\mathbb C)$
(cf.\ \cite[Cor.~7.4(a)]{Stembridge1990}). More bounded Littlewood identities
can be found in~\cite[Th.~2]{KratBC} and~\cite[Th.~2.3]{Okada1998} (to
some of which we will come back in Section~\ref{sec:symm-funct-ident}).

\medskip
The first main result of our paper is an affine analog of
Theorem~\ref{thm:BK-intro} (see Theorem~\ref{thm:affine_BK_intro} for
a more compact statement).

\begin{thm}[\sc Two affine bounded Littlewood identities]
\label{thm:affine_BK_intro2}
For positive integers \( h \) and \( w \), we have
\begin{multline}
\underset{ \la_1-\la_{2h+1}\le w}{\sum_{\la:\ell(\la)\le2h+1}}\
\underset{k_1+\dots+k_{2h+1}=0}{\sum_{ k_1,\dots,k_{2h+1}\in\Z}}
\det_{1\le i,j\le 2h+1}\left(
e_{\la_i-i+j+(2h+w+1)k_i}(\vx)\right)\\
=
\sum_{k\ge0}e_k(\vx)
\sum_{k_1,\dots,k_{h}\in\Z}
\det_{1\le i,j\le h}\left(
f_{-i+j+(2h+w+1)k_i}(\vx)-f_{i+j+(2h+w+1)k_i}(\vx)
\right)
\label{eq:ABL1}
\end{multline}
and
\begin{multline}
\underset{ \la_1-\la_{2h}\le w}{\sum_{\la:\ell(\la)\le 2h}}\
\underset{k_1+\dots+k_{2h}=0}{\sum_{ k_1,\dots,k_{2h}\in\mathbb Z}}
\det_{1\le i,j\le 2h}\left(
e_{\la_i-i+j+(2h+w)k_i}(\vx)\right)\\
=
\sum_{k_1,\dots,k_{h}\in\mathbb Z}
(-1)^{\sum_{i=1}^hk_i}
\det_{1\le i,j\le h}\big(
f_{-i+j+(2h+w)k_i}(\vx)+
f_{i+j-1+(2h+w)k_i}(\vx)
\big),
\label{eq:ABL2}
\end{multline}
with $f_r(\vx)$ being defined in \eqref{eq:f_k}.
\end{thm}

We explain the meaning of ``affine" in this context in Remark~\ref{rem:aff}(3). 
More affine bounded Littlewood identities are presented in
Section~\ref{sec:symm-funct-ident}; see Theorems~\ref{thm:C} and~\ref{thm:D}.

It is obvious that, in the limit $w\to\infty$, the
identities~\eqref{eq:ABL1} and~\eqref{eq:ABL2} reduce
to~\eqref{eq:BK_odd1} and~\eqref{eq:BK_even1}, respectively.
Nevertheless, the reader may wonder whether the identities in
Theorem~\ref{thm:affine_BK_intro2} would have any significance beyond
being extensions of the bounded Littlewood identities in
Theorem~\ref{thm:BK-intro}. Towards an answer to this question, we first point
out that the summand depending on~$\la$ on the left-hand sides
of~\eqref{eq:ABL1} and~\eqref{eq:ABL2}, 
\begin{equation} \label{eq:CSF} 
\underset{k_1+\dots+k_{h}=0}{\sum_{ k_1,\dots,k_{h}\in\Z}}
\det_{1\le i,j\le h}\left(
e_{\la_i-i+j+(h+w)k_i}(\vx)\right),
\end{equation}
is a ``cylindric" analog of the Schur function $s_{\la'}(\vx)$.
Namely, as the Schur function $s_{\la'}(\vx)$ is equal to a
generating function for semistandard Young tableaux of shape~$\la'$
(see~\eqref{eq:SF}),
by a result of Gessel and the third author~\cite{Gessel1997},
the expression in~\eqref{eq:CSF} is equal to a generating function for
semistandard Young tableaux of shape~$\la'$ that satisfy a
``cylindric" constraint. We call these tableaux {\it cylindric semistandard
  Young tableaux}; see Proposition~\ref{prop:cyl_sch} and
Theorem~\ref{thm:ssyt2} in Subsection~\ref{sec:symm}.
In fact, the {\it cylindric Schur
  functions} --- as we shall call the expressions in~\eqref{eq:CSF}
--- appeared for the first time
explicitly in a geometric context in~\cite[Sec.~6]{Postnikov2005}.
For further occurrences of cylindric semistandard Young tableaux and
(skew) cylindric Schur functions
see~\cite{Bertram1999,SeungJinLee2019,McNamara2006}. 

Returning to the question of the significance of the affine bounded
Littlewood identities~\eqref{eq:ABL1} and~\eqref{eq:ABL2}: as we
already mentioned, the origin of the bounded Littlewood identities
in Theorem~\ref{thm:BK-intro} lies in the enumeration of plane
partitions and tableaux.
Via the Robinson--Schensted algorithm and variations thereof, tableaux
are related to other combinatorial objects, such as
permutations, set partitions, and involutions, where the latter may
also be regarded as (partial) matchings; see  \cite{Chen2007,KratCE}.
Hence, the identities~\eqref{eq:ABL1} and~\eqref{eq:ABL2}
have more combinatorial applications.
One particularly interesting implication
of~\eqref{eq:BK_odd1} (which may be derived by extracting coefficients
of $x_1x_2\cdots x_n$ on both sides of~\eqref{eq:BK_odd1} and
combining the result with Gessel and Zeilberger's
random-walks-in-Weyl-chambers formula~\cite{Gessel1992} and Chen et
al.'s bijection in~\cite{Chen2007} between vacillating tableaux and
matchings) is\footnote{Alternatively, the identity~\eqref{eq:Chen2} can be proved
bijectively, essentially by a variant of the Robinson--Schensted
correspondence. We explain this in
Appendix~\ref{sec:growth},
using the growth
diagrams of Fomin. There, we also present a uniform treatment of the
related results on standard Young tableaux and walks in a Weyl chamber
of type~$A$ in~\cite{Eu2013,Zeilberger_lazy}.}
\begin{equation}
  \label{eq:Chen2}
  |\SYT_n(2h+1)| = |\NC_n(h+1)|,
\end{equation}
where $\SYT_n(2h+1)$
denotes the set of standard Young tableaux of size~$n$ with at
most $2h+1$~rows, and $\NC_n(h+1)$ is the set of (partial) matchings 
on
$\{1,2,\dots,n\}$ without an $(h+1)$-crossing; see
Subsection~\ref{sec:tableaux} for the definition of standard Young
tableaux, and Definition~\ref{def:NCNN} for the definition of crossings. 
Similarly, the affine bounded Littlewood
identities in Theorem~\ref{thm:affine_BK_intro2} are related to
several combinatorial objects, among which cylindric semistandard Young
tableaux, walks in type~$A$ alcoves, and (partial) matchings with
restrictions on their crossings {\it and\/} their nestings.
Consequently they have as well
combinatorial applications. One particular implication is 
\begin{equation}
  \label{eq:SYT-NCNN}
  |\CSYT_n(2h+1,2w+1)| = |\NCNN_n(h+1,w+1)|,
\end{equation}
where $\CSYT_n(2h+1,2w+1)$ denotes the set of
\( (2h+1,2w+1) \)-cylindric standard Young tableaux of size~\( n
\), and $\NCNN_n(h+1,w+1)$ denotes the set of (partial) matchings 
on
$\{1,2,\dots,n\}$
without an \( (h+1) \)-crossing and without a \( (w+1)
  \)-nesting; see Subsection~\ref{sec:tableaux} for the
  definition of cylindric standard Young tableaux, 
  Definition~\ref{def:NCNN} for the definition of crossings and nestings,
  and Equation~\eqref{eq:syt=ncnn1} in Corollary~\ref{cor:syt=ncnn}
  for the result. 

As a matter of fact, the
equality~\eqref{eq:SYT-NCNN} stood at the beginning of our
investigations that in the long run led to~\eqref{eq:ABL1} and the other
results reported in this article. More precisely, our original
(relatively modest) motivation was to find a bijective proof of a seemingly unrelated
result of Mortimer and Prellberg~\cite{Mortimer2015} on lattice walks in a
triangular region and bounded Motzkin paths. We
discovered~\eqref{eq:SYT-NCNN} --- {\it experimentally} ---
as a generalization of their result. How this --- at the time ---
conjecture, more or less ``inevitably", guided us to discover~\eqref{eq:ABL1}
and~\eqref{eq:ABL2} is explained in
Appendix~\ref{sec:motivation}.

\medskip
Our paper is organized as follows. In
Section~\ref{sec:preliminaries} we give basic definitions and some
preliminaries on tableaux and symmetric functions.
In Section~\ref{sec:symm-funct-ident}
we present our affine
bounded Littlewood identities. These include~\eqref{eq:ABL1}
and~\eqref{eq:ABL2} in more compact notation; see
Theorem~\ref{thm:affine_BK_intro}. The section contains moreover
also affine analogs of the known bounded Littlewood
identity in which the sum of Schur functions is restricted to
partitions all of whose parts (row lengths) are even (see \cite[Cor.~7.2]{Stembridge1990}),
and another in which the sum is over partitions all of whose column lengths are
even {\it or} all of them are odd (see \cite[Th.~2.3(3)]{Okada1998}).
Our corresponding results are presented in Theorems~\ref{thm:C} and~\ref{thm:D}.
We prove~\eqref{eq:ABL1} and~\eqref{eq:ABL2} in Section~\ref{sec:first-proof},
following the ``recipe" in \cite[Proof of Th.~7.1]{Stembridge1990},
essentially due to Gordon and Houten. In Section~\ref{sec:second-proof} we
introduce a systematic approach that provides an alternative proof of the affine
bounded Littlewood identities~\eqref{eq:ABL1} and~\eqref{eq:ABL2}. Moreover,
it also gives proofs of the additional identities in Theorems~\ref{thm:C} and~\ref{thm:D}.
The theme of Section~\ref{sec:cylindric_SSYT} is
combinatorial interpretations for the right-hand sides of the affine
bounded Littlewood identities~\eqref{eq:ABL1} and~\eqref{eq:ABL2} in terms of
up-down tableaux. These interpretations are then used in 
Section~\ref{sec:comb-ident} to derive several combinatorial
identities that put cylindric standard Young tableaux in relation with
so-called vacillating tableaux (which may also be considered as walks
in a Weyl chamber of type~$A$); see Corollary~\ref{cor:csyt=VT}.
If this is combined with the bijection in~\cite{Chen2007,KratCE},
further identities are obtained that relate cylindric standard Young
tableaux and matchings with restrictions on their crossings and
nestings; see Corollary~\ref{cor:syt=ncnn}.
In particular, the identity~\eqref{eq:SYT-NCNN} from above is proved
in~\eqref{eq:syt=ncnn1}. 
In the final section, Section~\ref{sec:h=1}, we discuss the special
case where $h=1$. In that case, the cylindric standard Young tableaux
may be related to certain walks in a triangle, and the matchings may
be related to Motzkin and Dyck paths and prefixes. This leads to the
identities in Theorem~\ref{thm:syt=ncnn_h=1}.
In particular,
there we come full circle and explain how the earlier mentioned result
of Mortimer and Prellberg fits into the picture. Courtiel, Elvey
Price and Marcovici~\cite{Courtiel2021} had found a bijective proof of
their result, while we failed to find a bijective proof but instead
found the much more general results presented in this article ---
with highly non-bijective proofs --- as we describe in Appendix~\ref{sec:motivation}.
We do however provide bijective proofs of the
variations of the result of Mortimer and Prellberg that are contained
in Theorem~\ref{thm:syt=ncnn_h=1}, see the second half of
Section~\ref{sec:h=1}.
As a bonus, in Appendix~\ref{sec:growth} we explain how to use Fomin's growth
diagrams to construct --- in a uniform manner --- bijections for the
``marginal" cases of our identities between numbers of standard Young
tableaux, matchings, and vacillating tableaux in
Section~\ref{sec:comb-ident}.

\section{Definitions and preliminaries}
\label{sec:preliminaries}

In this section, we give basic definitions for partitions,
tableaux, and symmetric functions, together with a few preliminary results.
In particular, in Definition~\ref{defn:cSchur} we introduce the
cylindric Schur functions which are central objects in our paper.

\subsection{Partitions and tableaux} \label{sec:tableaux}

A \emph{partition} of a nonnegative integer \( n \) is a weakly decreasing
sequence \( \lambda=(\lambda_1, \lambda_2, \dots, \lambda_k) \) of positive
integers, called \emph{parts}, such that \( \sum_{i\geq 1} \la_i = n \). This
also includes the empty partition \( () \), denoted by \( \emptyset \). If \(
\lambda \) is a partition of \( n \) into \( k \) parts, we write \( |\lambda|=n
\) and \( \ell(\lambda)=k \) and say that \( \lambda \) has \emph{size}~\( n \)
and \emph{height} (or \emph{length})~\( k \). We denote by \(\Par \) the set of all partitions. It
is often convenient to identify a partition \( (\lambda_1,\lambda_2,\dots,
\lambda_k) \) with the infinite sequence \( (\lambda_1,\lambda_2,\dots,
\lambda_k, 0, 0, \dots) \). Using this convention, we define \( \lambda_i=0 \)
for \( i >\ell(\lambda) \).

The \emph{Young diagram} of a partition \( \lambda=(\la_1,\la_2,\dots, \la_k) \)
is the set \( \{ (i,j) \in \Z^2: 1\le i\le k, 1\le j\le \lambda_i\} \). Each element \(
(i,j) \) in a Young diagram is called a \emph{cell}. The Young diagram of \(
\lambda \) is visualized as a left-justified array of unit square cells with \(
\lambda_i \) cells in the \( i \)-th row, \( i=1,2,\dots, k \), from top to
bottom. We identify \( \lambda \) with its Young diagram. The \emph{conjugate}
(or \emph{transpose}) \( \lambda' \) of a partition \( \lambda \) is the
partition whose Young diagram is given by \( \{(j,i): (i,j)\in \lambda\} \). For
two partitions \( \lambda \) and \( \mu \) we write \( \mu\subseteq\lambda \) to
mean that the Young diagram of \( \mu \) is contained in that of \( \lambda \).
If \( \mu\subseteq\lambda \), the \emph{skew shape} \( \lambda/\mu \) is the
set-theoretic difference \( \lambda-\mu \) of the Young diagrams of \( \lambda
\) and \( \mu \). Each partition \( \lambda \) is also considered as the skew
shape \( \lambda/\emptyset \).

A \emph{tableau} of shape \( \lambda/\mu \) is a filling of the cells in \(
\lambda/\mu \) with positive integers. For a tableau \( T \) of shape \(
\lambda=(\lambda_1, \lambda_2, \dots) \), the \emph{size}, \emph{height}, and
\emph{width} are defined to be \( |\lambda| \), \( \ell(\la) \), and \(
\lambda_1 \), respectively. A \emph{semistandard Young tableau} (SSYT) (respectively
\emph{row-strict tableau} (RST)) is a tableau in which the entries along rows
(respectively columns) are weakly increasing and the entries along columns (respectively rows)
are strictly increasing. A \emph{standard Young tableau} (SYT) is a semistandard
Young tableau (or equivalently, row-strict tableau) whose entries are the
integers \( 1,2,\dots,n \), where \( n \) is the size of the tableau.

\begin{defn}\label{defn:cylindric}
  An SSYT (respectively RST) \( T \) is \emph{\( (h,w) \)-cylindric} if \( T \) has at
  most \( h \) rows and \( T\cup (T+(h,-w)) \) is an SSYT (respectively RST) of a valid
  skew shape, where \( T+(h,-w) \) is the SSYT (respectively RST) obtained by
  shifting \( T \) by \( h \) units down and \( w \) units to the left (see
  Figure~\ref{fig:csyt}).
\end{defn}

\begin{figure}
{\small
\centering
\begin{tikzpicture}[scale=.44]
\node[right, color=black!80] at (-0.3,-1.85) {
\begin{ytableau}
1&1&3&5 \\
2&2&4 \\
3&4
\end{ytableau}};

\draw[thick] (0,0) -- (0,-3.69) -- (2.46,-3.69) -- (2.46,-2.46) -- (3.69,-2.46) -- (3.69,-1.23) -- (4.92,-1.23) -- (4.92,0) -- (0,0) -- (0,-1);
\node at (0,-7.4) {};

\end{tikzpicture}
\qquad \qquad
\begin{tikzpicture}[scale=.44]
\node[right, color=black!80] at (-0.3,0) {
\begin{ytableau}
\none&\none&\none&1&1&3&5 \\
\none&\none&\none&2&2&4 \\
\none&\none&\none&3&4 \\
1&1&3&5 \\
2&2&4 \\
3&4
\end{ytableau}};

\draw[thick] (0,0) -- (0,-3.69) -- (2.46,-3.69) -- (2.46,-2.46) -- (3.69,-2.46) -- (3.69,-1.23) -- (4.92,-1.23) -- (4.92,0) -- (0,0) -- (0,-1)
(3.69,3.69) -- (3.69,0) -- (6.15,0) -- (6.15,1.23) -- (7.38,1.23) -- (7.38,2.46) -- (8.61,2.46) -- (8.61,3.69) -- (3.69,3.69) -- (3.69,3);

\end{tikzpicture}
\qquad \qquad
\begin{tikzpicture}[scale=.44]
\node[right, color=black!80] at (-0.3,0) {
\begin{ytableau}
\none&\none&1&1&3&5 \\
\none&\none&2&2&4 \\
\none&\none&3&4 \\
1&1&3&5 \\
2&2&4 \\
3&4
\end{ytableau}};

\draw[thick] (0,0) -- (0,-3.69) -- (2.46,-3.69) -- (2.46,-2.46) -- (3.69,-2.46) -- (3.69,-1.23) -- (4.92,-1.23) -- (4.92,0) -- (0,0) -- (0,-1)
(2.46,3.69) -- (2.46,0) -- (4.92,0) -- (4.92,1.23) -- (6.15,1.23) -- (6.15,2.46) -- (7.38,2.46) -- (7.38,3.69) -- (2.46,3.69) -- (2.46,3);

\end{tikzpicture}

}
\caption{The SSYT \( T \) on the left is \( (3,3) \)-cylindric but not \( (3,2)
  \)-cylindric because the tableau in the middle is an SSYT but the tableau on
  the right is not an SSYT. Moreover, \( T \) is not \( (3,1) \)-cylindric
  because \( T\cup (T+(3,-1)) \) is not of a skew shape. }
  \label{fig:csyt}
\end{figure}
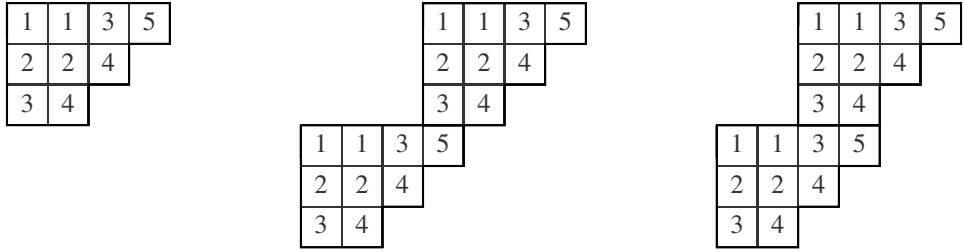

\begin{defn} \label{def:par}
We write $\Par(h)$ for the set of partitions \( \lambda \) with \(
  \ell(\lambda)\le h \), and
  \( \Par(h,w) \) for the subset of $\Par(h)$ consisting of the
  partitions that satisfy in addition \( \lambda_1-\lambda_h\le w \). For \(
  \lambda\in\Par(h,w) \), we denote by \( \CSSYT(\lambda;h,w) \) (respectively \(
  \CRST(\lambda;h,w) \)) the set of all \( (h,w) \)-cylindric SSYTs (respectively RSTs)
  of shape \( \lambda \). Let
  \[
    \CSSYT(h,w) = \bigcup_{\lambda\in \Par(h,w)} \CSSYT(\lambda;h,w),\qquad
    \CRST(h,w) = \bigcup_{\lambda\in \Par(h,w)} \CRST(\lambda;h,w).
  \]
  We also write \( \CSYT_n(h,w) \) for the set of \( (h,w) \)-cylindric standard
  Young tableaux of size~\( n \).
\end{defn}

Note that, for a partition \( \lambda \) with \( \ell(\lambda)\le h \), we have
\( \CSSYT(\lambda;h,w)=\emptyset \) unless \( \lambda_1-\lambda_h\le w \)
because if \( \lambda_1-\lambda_h>w \), then \( T\cup (T+(h,-w)) \) is not of a
skew shape.

\begin{defn}
  Let \( h \) and \( w \) be positive integers. For \( \lambda\in \Par(h,w) \),
  the \emph{(h,w)-cylindric shape} \( \lambda[h,w] \) is defined by
\[
  \lambda[h,w] = \{(i,j)+(kh,-kw): (i,j)\in \lambda, k\in \Z\}/{\sim}
\]
with the relation \( (i,j) \sim (i',j')\) if and only if \( (i,j) =
(i',j')+(kh,-kw) \) for some \( k\in\Z \). The \emph{transpose} \(
\lambda[h,w]' \) of \( \lambda[h,w] \) is defined by \( \lambda[h,w]' =
\{[(j,i)]: [(i,j)]\in \lambda[h,w]\} \), where \( [(i,j)] \) is the equivalence
class containing \( (i,j) \). See Figure~\ref{fig:cyl_sh}.
The \emph{\( (h,w) \)-transpose} \(
\Tr(\lambda;h,w) \) of \( \lambda \) is defined by
\[
  \Tr(\lambda;h,w) = \{(i,j): [(i,j)]\in \lambda[h,w]', 1\le i\le w\}.
\]
\end{defn}

\begin{figure}
{\small
\centering

\begin{tikzpicture}[scale=.44]

\node at (10.4,8.12) {\(\reflectbox{\(\ddots\)}\)};
\node at (-0.5,-4) {\(\reflectbox{\(\ddots\)}\)};
\node at (2.1,-0.65) {\( \lambda+(h,-w) \)};
\node at (3.14,-1.845) {\( \bullet \)};
\node at (3,3) {\( \lambda \)};
\node at (5.6,1.845) {\( \bullet \)};
\node at (7,6.7) {\( \lambda-(h,-w) \)};
\node at (8.06,5.535) {\( \bullet \)};

\draw[thick] (0,0) -- (0,-3.69) -- (2.46,-3.69) -- (2.46,-2.46) -- (3.69,-2.46) -- (3.69,-1.23) -- (4.92,-1.23) -- (4.92,0) -- (0,0) -- (0,-1)
(2.46,3.69) -- (2.46,0) -- (4.92,0) -- (4.92,1.23) -- (6.15,1.23) -- (6.15,2.46) -- (7.38,2.46) -- (7.38,3.69) -- (2.46,3.69) -- (2.46,3)
(4.92,7.38) -- (4.92,3.69) -- (7.38,3.69) -- (7.38,4.92) -- (8.61,4.92) -- (8.61,6.15) -- (9.84,6.15) -- (9.84,7.38) -- (4.92,7.38) -- (4.92,7);

\end{tikzpicture}
\qquad \qquad
\begin{tikzpicture}[scale=.44]

\node at (11.6,5.6) {\(\reflectbox{\(\ddots\)}\)};
\node at (-0.5,-5.2) {\(\reflectbox{\(\ddots\)}\)};
\node at (1.9,-0.65) {\footnotesize{\( \lambda'+(w,-h) \)}};
\node at (1.91,-3.125) {\( \bullet \)};
\node at (4.4,1.8) {\( \lambda' \)};
\node at (5.6,-0.665) {\( \bullet \)};
\node at (9.25,4.3) {\footnotesize{\(\lambda'-(w,-h) \)}};
\node at (9.29,1.795) {\( \bullet \)};

\draw[thick] (0,0) -- (0,-4.92) -- (1.23,-4.92) -- (1.23,-3.69) -- (2.46,-3.69) -- (2.46,-2.46) -- (3.69,-2.46) -- (3.69,0) -- (0,0) -- (0,-1)
(3.69,2.46) -- (3.69,-2.46) -- (4.92,-2.46) -- (4.92,-1.23) -- (6.15,-1.23) -- (6.15,0) -- (7.38,0) -- (7.38,2.46) -- (3.69,2.46) -- (3.69,2)
(7.38,4.92) -- (7.38,0) -- (8.61,0) -- (8.61,1.23) -- (9.84,1.23) -- (9.84,2.46) -- (11.07,2.46) -- (11.07,4.92) -- (7.38,4.92) -- (7.38,4);

\draw[ultra thick] (3.69,2.46) -- (3.69,0) -- (8.61,0) -- (8.61,1.23) -- (9.84,1.23) -- (9.84,2.46) -- (3.69,2.46) -- (3.69,2);

\end{tikzpicture}

}
\caption{The \( (3,2) \)-cylindric shape \( \lambda[3,2] \) of \(
  \lambda=(4,3,2) \) is shown on the left. Its transpose \( \lambda[3,2]' \) is
  shown on the right, where the \( (3,2) \)-transpose \( \Tr(\lambda;3,2) \) of
  \( \lambda \) is drawn with thick boundary. The cells with dots are identified
  by the relation \( \sim \).}
  \label{fig:cyl_sh}
\end{figure}
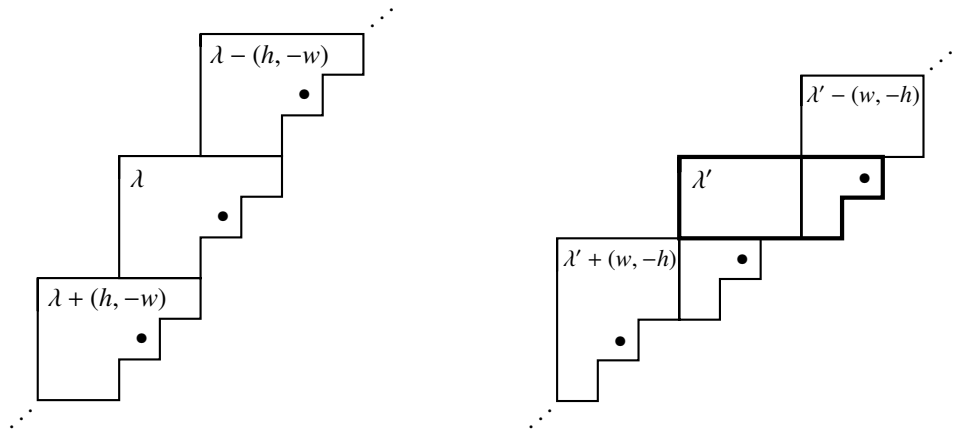

By definition one can easily see that \( \lambda[h,w]'=\Tr(\lambda;h,w)[w,h] \).
The map \( \lambda\mapsto \Tr(\lambda;h,w) \) is essentially the same as the map
\( \Psi \) due to Goodman and Wenzl~\cite[p.~253]{Goodman1990}.

One can naturally identify an \( (h,w) \)-cylindric SSYT (or RST) \( T \) of
shape \( \lambda \) with a filling of the cylindric shape \( \lambda[h,w] \) by
filling each cell \( (i',j')\in [(i,j)] \) with the same entry as the one in \( (i,j) \).
Then the transpose~\( T' \) and the \( (h,w) \)-transpose \( \Tr(T;h,w) \) are
defined in the obvious way as shown in Figure~\ref{fig:T-trans}.

\begin{figure}
{\small
\centering

\begin{tikzpicture}[scale=.44]
\node[right, color=black!80] at (-0.3,1.85) {
\begin{ytableau}
\none&\none&\none&\none&1&1&4&5 \\
\none&\none&\none&\none&2&2&5 \\
\none&\none&\none&\none&3&4 \\
\none&\none&1&1&4&5 \\
\none&\none&2&2&5 \\
\none&\none&3&4 \\
1&1&4&5 \\
2&2&5 \\
3&4
\end{ytableau}};
\node at (10.4,8.12) {\(\reflectbox{\(\ddots\)}\)};
\node at (-0.5,-4) {\(\reflectbox{\(\ddots\)}\)};

\draw[thick] (0,0) -- (0,-3.69) -- (2.46,-3.69) -- (2.46,-2.46) -- (3.69,-2.46) -- (3.69,-1.23) -- (4.92,-1.23) -- (4.92,0) -- (0,0) -- (0,-1)
(2.46,3.69) -- (2.46,0) -- (4.92,0) -- (4.92,1.23) -- (6.15,1.23) -- (6.15,2.46) -- (7.38,2.46) -- (7.38,3.69) -- (2.46,3.69) -- (2.46,3)
(4.92,7.38) -- (4.92,3.69) -- (7.38,3.69) -- (7.38,4.92) -- (8.61,4.92) -- (8.61,6.15) -- (9.84,6.15) -- (9.84,7.38) -- (4.92,7.38) -- (4.92,7);

\end{tikzpicture}
\qquad \qquad
\begin{tikzpicture}[scale=.44]
\node[right, color=black!80] at (-0.3,0) {
\begin{ytableau}
\none&\none&\none&\none&\none&\none&1&2&3 \\
\none&\none&\none&\none&\none&\none&1&2&4 \\
\none&\none&\none&1&2&3&4&5 \\
\none&\none&\none&1&2&4&5 \\
1&2&3&4&5 \\
1&2&4&5 \\
4&5\\
5
\end{ytableau}};
\node at (11.6,5.6) {\(\reflectbox{\(\ddots\)}\)};
\node at (-0.5,-5.2) {\(\reflectbox{\(\ddots\)}\)};

\draw[thick] (0,0) -- (0,-4.92) -- (1.23,-4.92) -- (1.23,-3.69) -- (2.46,-3.69) -- (2.46,-2.46) -- (3.69,-2.46) -- (3.69,0) -- (0,0) -- (0,-1)
(3.69,2.46) -- (3.69,-2.46) -- (4.92,-2.46) -- (4.92,-1.23) -- (6.15,-1.23) -- (6.15,0) -- (7.38,0) -- (7.38,2.46) -- (3.69,2.46) -- (3.69,2)
(7.38,4.92) -- (7.38,0) -- (8.61,0) -- (8.61,1.23) -- (9.84,1.23) -- (9.84,2.46) -- (11.07,2.46) -- (11.07,4.92) -- (7.38,4.92) -- (7.38,4);

\draw[ultra thick] (3.69,2.46) -- (3.69,0) -- (8.61,0) -- (8.61,1.23) -- (9.84,1.23) -- (9.84,2.46) -- (3.69,2.46) -- (3.69,2);

\end{tikzpicture}

}
\caption{The left diagram shows a \( (3,2) \)-cylindric SSYT \( T \) of
    shape \( \lambda=(4,3,2) \) identified with a filling of the cylindric shape
    \( \lambda[3,2] \). The right diagram shows the transpose \( T' \) as a
    filling of \( \lambda[3,2]' \) and the \( (3,2) \)-transpose \(
    \Tr(T;3,2) \) of \( T \) whose shape is \( \Tr(\lambda;3,2)=(5,4) \)
    drawn with thick boundary.}
  \label{fig:T-trans}
\end{figure}
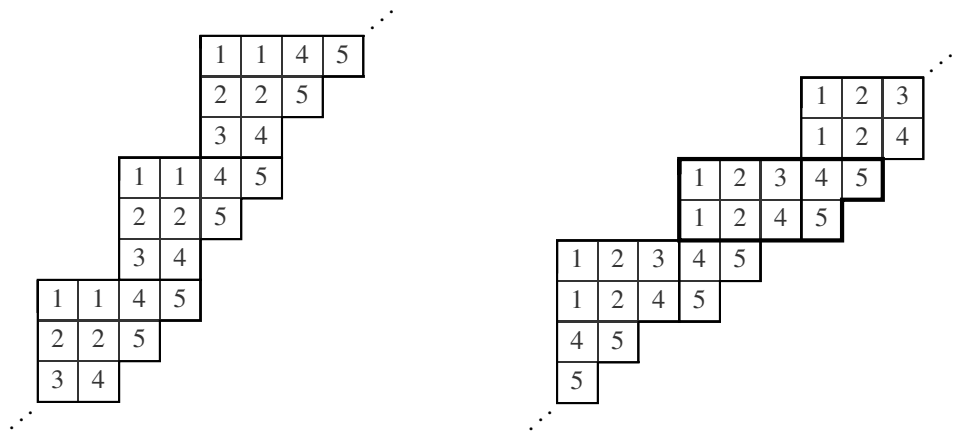

\begin{prop}\label{prop:cssyt=crst}
  Let \( h \) and \( w \) be positive integers. Then the map \( T\mapsto
  \Tr(T;h,w) \) is a bijection between \( \CRST(h,w) \) and \( \CSSYT(w,h) \) (and
  also between \( \CSSYT(h,w) \) and \( \CRST(w,h) \)). Moreover, for \(
  \lambda\in\Par(h,w) \), this map induces a bijection between \(
  \CRST(\lambda;h,w) \) and \( \CSSYT(\Tr(\lambda;h,w);w,h) \) and also a
  bijection between \( \CSYT_n(h,w) \) and \( \CSYT_n(w,h) \). In particular, we
  have
  \begin{align*}
    \label{eq:syt=syt}
    |\CSYT_n(h,w)| = |\CSYT_n(w,h)|.
  \end{align*}
\end{prop}
\begin{proof}
  This is immediate from the construction of the map \( T\mapsto \Tr(T;h,w) \).
\end{proof}

\subsection{Symmetric functions} \label{sec:symm}

In this paper we shall be concerned with the following symmetric functions in
the variables \( \vx= \{ x_1,x_2,\dots \} \). 

The \emph{\( n \)-th complete homogeneous symmetric function} \( h_n(\vx) \) and the \emph{\( n \)-th elementary symmetric function} \( e_n(\vx) \) are defined by
\[
h_n(\vx) =\sum_{i_1\leq i_2 \leq \cdots \leq i_n} x_{i_1}x_{i_2}\cdots x_{i_n},
\qquad
e_n(\vx) =\sum_{i_1<i_2<\cdots<i_n} x_{i_1}x_{i_2}\cdots x_{i_n},
\]
respectively. We set \( h_0(\vx)=e_0(\vx)=1 \) and define \(
h_n(\vx)\) and \(e_n(\vx) \) to be zero for \( n <0 \).

For any tableau \( T\), let \( \vx^T=x_1^{\alpha_1}x_2^{\alpha_2}\cdots
\), where \( \alpha_i \) is the number of \( i \)'s in \( T \). For a partition
\( \lambda \), the \emph{Schur function} \( s_\lambda(\vx) \) is defined by
\begin{equation} \label{eq:SF}
s_{\lambda}(\vx)=\sum_{T} \vx^T,
\end{equation}
where the sum is over all semistandard Young tableaux \( T \) of shape \(
\lambda \). If instead in \eqref{eq:SF} we sum over {\it cylindric}
semistandard Young tableaux of a given shape, then we call the
resulting object ``cylindric Schur function".

\begin{defn}\label{defn:cSchur}
  Let \( h \) and \( w \) be positive integers and let \( \lambda\in \Par(h,w)
  \). The \emph{\( (h,w) \)-cylindric Schur function} \( s_{\lambda[h,w]}(\vx)
  \) of shape \( \lambda[h,w] \) is defined by
  \[
    s_{\lambda[h,w]}(\vx)  =  \sum_{T\in \CSSYT(\lambda;h,w)} \vx^T.
  \]
\end{defn}

By definition, we have
\begin{equation} \label{eq:s-inf}
 \lim_{w\to\infty} s_{\lambda[h,w]}(\vx) = s_{\lambda}(\vx).
\end{equation}

Recall that each tableau \( T\in \CSSYT(\lambda;h,w) \) can be understood as a
filling of the cylindric shape \( \lambda[h,w] \). Therefore our definition of a
cylindric Schur function is equivalent to that of Postnikov~\cite[Sec.~6]{Postnikov2005}.

For our purpose it is more convenient to deal with \( s_{\lambda[h,w]'}(\vx) \),
which is a \( (w,h) \)-cylindric Schur function. The following proposition shows
that \( s_{\lambda[h,w]'}(\vx) \) can also be seen as the generating function for cylindric
row-strict tableaux.

\begin{prop}\label{prop:cyl_sch}
  Let \( h \) and \( w \) be positive integers and \( \lambda\in \Par(h,w) \).
  Then we have
  \[
    s_{\lambda[h,w]'}(\vx) =  \sum_{T\in \CRST(\lambda;h,w)} \vx^T.
  \]
\end{prop}
\begin{proof}
  Since \(\lambda[h,w]'=\Tr(\lambda;h,w)[w,h] \), we have
  \[
    s_{\lambda[h,w]'}(\vx)  =  \sum_{T\in \CSSYT(\Tr(\lambda;h,w);w,h)} \vx^T.
  \]
  On the other hand, by the map in Proposition~\ref{prop:cssyt=crst},
  \[
    \sum_{T\in \CRST(\lambda;h,w)} \vx^T
    =\sum_{T\in \CSSYT(\Tr(\lambda;h,w);w,h)} \vx^T.
  \]
  Combining the above two equations, we obtain the assertion of the proposition.
\end{proof}

The \emph{quantum Kostka numbers} \( K_{\lambda[h,w]}^\alpha \)
are defined by
\[
  s_{\lambda[h,w]}(\vx) = \sum_{\rm\alpha} K_{\lambda[h,w]}^\alpha \vx^\alpha,
\]
where the sum is over all sequences \( \alpha=(\alpha_1,\alpha_2,\dots) \) of
nonnegative integers. By Proposition~\ref{prop:cyl_sch}, \(
K_{\lambda[h,w]'}^\alpha \) is the number of tableaux \( T\in
\CRST(\lambda;h,w) \) in which the number of \( i \)'s is equal to \( \alpha_i
\) for all \( i\ge1 \). There is another description of \(
K_{\lambda[h,w]'}^\alpha \) in terms of lattice paths.

\begin{prop}\label{prop:cssyt=path}
  Let \( h \) and \( w \) be positive integers, \( \lambda \) a partition in
  \( \Par(h,w) \), and \( \alpha=(\alpha_1,\alpha_2,\dots) \) a sequence of
  nonnegative integers. Then \( K_{\lambda[h,w]'}^\alpha \) equals the number of
  paths in $\Z^{h}$ from the origin to \( (\lambda_1,\dots,\lambda_h) \) and
  staying in the region
 \[
  \big\{(x_1,x_2,\dots,x_{h}):x_1\ge x_2\ge \dots\ge x_{h}\ge
  x_1-w\big\},
  \]
  where the $i$-th step is a vector with $\alpha_i$ coordinates equal to~$1$ and
  $h-\alpha_i$ coordinates equal to~$0$.
\end{prop}
\begin{proof}
  We obtain the proposition by reading the vectors of row lengths of the
  subtableaux containing all entries at most \( i \) for $i=0,1,\dots$. For
  example, the \( (2,3) \)-cylindric row-strict tableau in the right diagram of
  Figure~\ref{fig:T-trans} corresponds to the lattice path
  \[
    (0,0)\to (1,1)\to (2,2) \to (3,2)\to (4,3)\to (5,4).
  \]
  The property that the tableau is \( (h,w) \)-cylindric translates into the property
  that the lattice walk is in the given region.
\end{proof}

As reported in~\cite[Eq.~(11)]{Postnikov2005}, the results on
cylindric tableaux of Gessel and the third author~\cite{Gessel1997} imply a
Jacobi--Trudi-type formula for the cylindric Schur functions.
For the sake of completeness, we also provide the corresponding proof.

\begin{thm} \label{thm:ssyt2}
For positive integers \( h \) and \( w \)
and a partition \( \lambda\in\Par(h,w) \), we have
\begin{equation} \label{eq:CJT}
s_{\lambda[h,w]'}(\vx)=
\underset{k_{1}+\cdots+k_{h}=0}{\sum_{k_{1}, \ldots, k_{h} \in \mathbb{Z}}}
    \det_{1 \leq i, j \leq h} \left(e_{\lambda_{i}-i+j+(h+w) k_{i}}(\vx)\right).
\end{equation}
\end{thm}

\begin{proof}
By Proposition~\ref{prop:cyl_sch}, the cylindric Schur function on the
left-hand side of~\eqref{eq:CJT}  is a generating function for
cylindric row-strict tableaux. 
We first show that these tableaux are in bijection with certain
  families of nonintersecting lattice paths. In order to see this, we place
  ourselves in the setting of Section~9 in~\cite{Gessel1997}. Namely, we
  consider the graph with vertices being the points in the integer plane $\Z^2$,
  and with edges being horizontal edges $(i-1,j)\to(i,j)$ and vertical edges
  $(i,j-1)\to(i,j)$. We define the weight of vertical edges to be~1 and the
  weight of a horizontal edge $e$ from $(i-1,j)$ to $(i,j)$ to be
  $\omega(e):=x_{i+j}$. By definition, the weight $\omega(P)$ of a path~$P$ is the product
  of the weights of all its edges, and the weight $\omega(\mathbf P)$ of a family
  $\mathbf P=(P_1,P_2,\dots,P_k)$ of paths is $ \prod _{i=1} ^{k}\omega(P_i)$. Now,
  given \( T\in \CRST(\lambda;w,h) \), we convert it into a family $\mathbf
  P=(P_1,P_2,\dots,P_{h})$ of paths in the graph, by mapping the $i$-th row of
  \( T \) to the path $P_i$ from $(-i+1,i-1)$ to $(\la_i-i+1,\infty)$ such that
  the weights of the horizontal steps of~$P_i$ are indexed by the entries in that row.
  For example, see Figure~\ref{fig:6} for the family of paths corresponding to the
  tableau in Figure~\ref{fig:rst}.

\begin{figure}
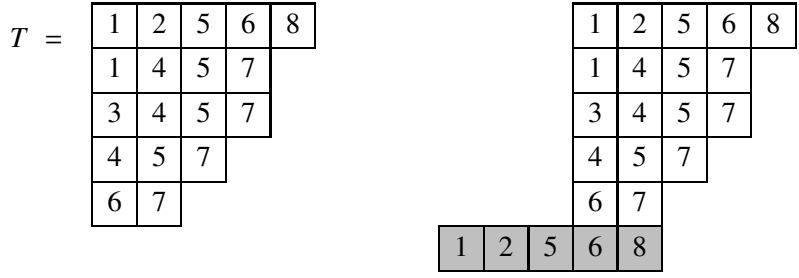

  \centering
 \(T~=\)\quad\begin{ytableau}
    1&2&5&6&8\\
    1&4&5&7\\
    3&4&5&7\\
    4&5&7\\
    6&7
  \end{ytableau}
  \qquad \qquad
  \begin{ytableau}
    \none&\none&\none&1&2&5&6&8\\
    \none&\none&\none&1&4&5&7\\
    \none&\none&\none&3&4&5&7\\
    \none&\none&\none&4&5&7\\
    \none&\none&\none&6&7\\
    *(lightgray) 1&*(lightgray)2&*(lightgray)5&*(lightgray)6&*(lightgray)8
  \end{ytableau}

  \caption{The left diagram shows a row-strict tableau \( T\in\CRST(\lambda;h,w)
    \) for $\la=(5,4,4,3,2)$, \( h=5 \), and \( w=3 \).
  The right diagram shows \( T \) together with a copy of its first row
  translated by \( (h,-w) \).}
  \label{fig:rst}
\end{figure}

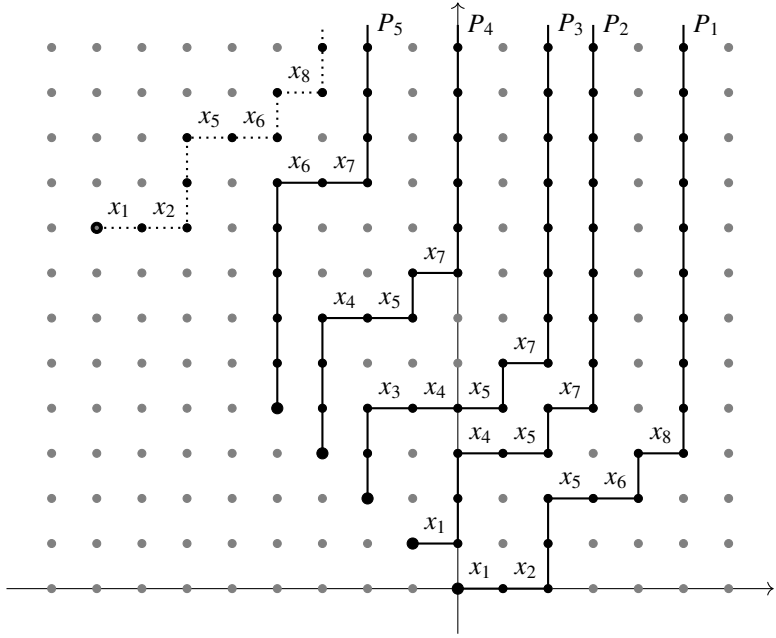
\begin{figure}
\small{
\begin{tikzpicture}[scale=0.6]

\draw[->] (-10,0) -- (7,0);
\draw[->] (0,-1) -- (0,13);

\foreach \i in {-9,...,6}
\foreach \j in {0,...,12}
\filldraw[fill=gray, color=gray] (\i,\j) circle (2.5pt);

\foreach \i in {0,1,2,3,4,8}
\filldraw (-\i,\i) circle (3.5pt);

\foreach \i/\j in {0,1,3,4,5,8}
\filldraw (1-\j,\j) circle (2.5pt);

\foreach \i/\j in {0,2,4,5,6,8}
\filldraw (2-\j,\j) circle (2.5pt);

\foreach \i/\j in {1,3,4,6,7,9}
\filldraw (3-\j,\j) circle (2.5pt);

\foreach \i/\j in {2,3,4,6,8,10}
\filldraw (4-\j,\j) circle (2.5pt);

\foreach \i/\j in {2,3,4,6,9,10}
\filldraw (5-\j,\j) circle (2.5pt);

\foreach \i/\j in {2,4,5,7,9,10}
\filldraw (6-\j,\j) circle (2.5pt);

\foreach \i/\j in {3,4,5,7,9,11}
\filldraw (7-\j,\j) circle (2.5pt);


\foreach \j in {3,4,...,12}
\filldraw (8-3,\j) circle (2.5pt);

\foreach \j in {5,6,...,12}
\filldraw (8-5,\j) circle (2.5pt);

\foreach \j in {6,7,...,12}
\filldraw (8-6,\j) circle (2.5pt);

\foreach \j in {8,9,...,12}
\filldraw (8-8,\j) circle (2.5pt);

\foreach \j in {10,11,12}
\filldraw (8-10,\j) circle (2.5pt);

\draw[thick, dotted]
(-8,8) -- (-6,8) -- (-6,10) -- (-4,10) -- (-4, 11) -- (-3, 11)--(-3,12.5);

\foreach \j in {11,12}
\filldraw (8-11,\j) circle (2.5pt);

\filldraw[color=gray] (-8,8) circle (1pt);

\draw[thick] (0,0) -- (2,0) -- (2,2) -- (4,2) -- (4,3) -- (5,3)--(5,12.5)
(-1,1) -- (0,1) -- (0,3) -- (2,3) -- (2,4) -- (3,4) -- (3,5)--(3,12.5)
(-2,2) -- (-2,4) -- (1,4) -- (1,5) -- (2,5) -- (2,6)--(2,12.5)
(-3,3) -- (-3,6) -- (-1,6) -- (-1,7) -- (0,7) -- (0,8)--(0,12.5)
(-4,4) -- (-4,9) -- (-2,9) -- (-2,10)--(-2,12.5);

\foreach \j in {0,1,8}
\node at (0.5-\j, 0.4+\j) {\(x_1\)};

\foreach \j in {0,8}
\node at (1.5-\j, 0.4+\j) {\(x_2\)};

\foreach \j in {4}
\node at (2.5-\j, 0.4+\j) {\(x_3\)};

\foreach \j in {3,4,6}
\node at (3.5-\j, 0.4+\j) {\(x_4\)};

\foreach \j in {2,3,4,6,10}
\node at (4.5-\j, 0.4+\j) {\(x_5\)};

\foreach \j in {2,9,10}
\node at (5.5-\j, 0.4+\j) {\(x_6\)};

\foreach \j in {4,5,7,9}
\node at (6.5-\j, 0.4+\j) {\(x_7\)};

\foreach \j in {3,11}
\node at (7.5-\j, 0.4+\j) {\(x_8\)};

\node[right] at (5, 12.5) {\(P_1\)};
\node[right] at (3, 12.5) {\(P_2\)};
\node[right] at (2, 12.5) {\(P_3\)};
\node[right] at (0, 12.5) {\(P_4\)};
\node[right] at (-2, 12.5) {\(P_5\)};

\end{tikzpicture}

}
\caption{The family \(\mathbf{P}=(P_1,P_2,P_3,P_4,P_5)\) of paths corresponding to the tableau \( T \) in Figure~\ref{fig:rst}. The shift of \(P_1\) is indicated by the dotted path.}
\label{fig:6}
\end{figure}

The property of \( T \) having weak increase of entries along
columns translates into the property of $\mathbf P$ being nonintersecting, and
the property of \( T \) being \( (h,w) \)-cylindric translates into the property that
the shift of $P_1$ by $(-h-w,h+w)$ does not intersect~$P_{h}$ (and, thus, also
not any other paths). In Figure~\ref{fig:6}, this shift of~$P_1$ is
indicated by the dotted path.

Therefore the left-hand side of~\eqref{eq:CJT} equals the generating function $\sum_{\mathbf
  P}\omega(\mathbf P)$ for these families $\mathbf P$ of nonintersecting lattice
paths. As explained in~\cite[Sec.~9]{Gessel1997} this generating function is
equal to
\[
\underset{k_{1}+\cdots+k_{h}=0}{ \sum_{k_{1}, \ldots, k_{h} \in \mathbb{Z} }}
    \det_{1 \leq i, j \leq h} \left(e_{\lambda_{i}-i+j+(h+w) k_{i}}(\vx)\right),
\]
which is exactly the right-hand side of~\eqref{eq:CJT}.
\end{proof}

\section{Affine bounded Littlewood identities}
\label{sec:symm-funct-ident}

In this section, we present our affine bounded Littlewood identities,
see Theorems~\ref{thm:affine_BK_intro}, \ref{thm:C}, and~\ref{thm:D}.
A first proof of Theorem~\ref{thm:affine_BK_intro} is given in
Section~\ref{sec:first-proof}. In the subsequent
Section~\ref{sec:second-proof} we develop a uniform approach to prove
affine bounded Littlewood identities. This leads to an alternative
proof of Theorem~\ref{thm:affine_BK_intro}, as well as to proofs of
Theorems~\ref{thm:C} and \ref{thm:D}.

\medskip
We start by restating the identities~\eqref{eq:ABL1} and~\eqref{eq:ABL2}
in a more compact form. As explained in the introduction, these are
affine extensions of the bounded Littlewood
identities~\eqref{eq:BK_odd1} and~\eqref{eq:BK_even1}. Moreover, as
pointed out after Theorem~\ref{thm:BK-intro}, the latter two identities 
express odd orthogonal characters indexed by a rectangular shape in terms of
Schur functions.

\begin{thm}[\sc Two affine bounded Littlewood identities: odd
    orthogonal case]
\label{thm:affine_BK_intro}
For positive integers \( h \) and \( w \), we have
\begin{align}
  \label{eq:aGBK1s}
  \sum_{\la\in\Par(2h+1,w)} s_{\lambda[2h+1,w]'}(\vx)
  &= \sum_{k\ge0}e_k(\vx)
    \det_{1\le i,j\le h}\left(
    F_{-i+j,2h+1+w}(\vx) - F_{i+j,2h+1+w}(\vx)
    \right),\\
  \label{eq:aGBK2s}
  \sum_{\la\in\Par(2h,w)} s_{\lambda[2h,w]'}(\vx)
  &= \det_{1\le i,j\le h}\left(
    \overline{F}_{-i+j,2h+w}(\vx) + \overline{F}_{i+j-1,2h+w}(\vx)
    \right),
\end{align}
where \( \Par(m,w) \) is the set of partitions \( \lambda \)
with \( \ell(\lambda)\le m \) and \( \lambda_1-\lambda_m\le w \), and
  \begin{align*}
    F_{r,N}(\vx) &= \sum_{k\in\Z} f_{r+Nk}(\vx),\\
    \overline{F}_{r,N}(\vx) &= \sum_{k\in\Z} (-1)^k f_{r+Nk}(\vx),
  \end{align*}
with
$$
f_r(\vx)=\sum_{i\in \Z}e_i(\vx) e_{i+r}(\vx),
$$
as before.
\end{thm}

\begin{rems} \label{rem:aff}
(1) It should be obvious that the identities~\eqref{eq:aGBK1s}
  and~\eqref{eq:aGBK2s} are equivalent with~\eqref{eq:ABL1}
  and~\eqref{eq:ABL2}, respectively.

\medskip\noindent
(2) Note that \( f_{r}(\vx) = f_{-r}(\vx) \),
\( F_{r,N}(\vx) = F_{-r,N}(\vx) \),
\( \overline{F}_{r,N}(\vx) = \overline{F}_{-r,N}(\vx) \),  and
\[
  \lim_{N\to\infty} F_{r,N}(\vx) =
  \lim_{N\to\infty} \overline{F}_{r,N}(\vx) = f_r(\vx).
\]
In particular, as $w\to\infty$, the identities~\eqref{eq:aGBK1s}
  and~\eqref{eq:aGBK2s} reduce to~\eqref{eq:BK_odd1}
  and~\eqref{eq:BK_even1}, respectively.

\medskip\noindent
(3) Why do we call the identities in Theorem~\ref{thm:affine_BK_intro}
``affine" bounded Littlewood identities, i.e., what is the meaning of
``affine" in this context? To understand this, we observe that,
by~\eqref{eq:JT}, the summand indexed by~$\la$ on the left-hand sides
of~\eqref{eq:ABL1} and~\eqref{eq:ABL2} can be written as
$$
s_{\lambda'}(\vx) = \det_{1\le i,j\le m}\big(e_{\lambda_i-i+j}(\vx)\big)=
\sum_{\si\in\mathfrak S_m}\sgn\si
\prod _{i=1} ^{m}e_{\la_i-i+\si(i)}(\vx),
$$
where $m=2h+1$ respectively $m=2h$.
The right-hand side
is a sum over the symmetric
group~$\mathfrak S_m$. In the classification of finite Coxeter groups,
this is the reflection group of type~$A_{m-1}$
(cf.\ \cite[p.~41]{Humphreys}). On the other hand, 
by~\eqref{eq:CJT},
the summand indexed by~$\la$ on the left-hand sides of~\eqref{eq:aGBK1s}
and~\eqref{eq:aGBK2s} can be written as
$$
s_{\lambda[m,w]'}(\vx) =
\underset{k_1+\dots+k_{m}=0}{\sum_{k_1,\dots,k_{m}\in\Z}}
\det_{1\le i,j\le m}\big(e_{\lambda_i-i+j+(m+w)k_i}(\vx)\big)=
\underset{k_1+\dots+k_{m}=0}{\sum_{ k_1,\dots,k_{m}\in\Z}}
\sum_{\si\in\mathfrak S_m}\sgn\si
\prod _{i=1} ^{m}e_{\la_i-i+\si(i)+(m+w)k_i}(\vx),
$$
where $m$ has the same meaning as before.
The right-hand side is now a sum over the {\it affine} symmetric
group~$\tilde{\mathfrak S}_m$ defined by
$$\tilde{\mathfrak{S}}_m:=\mathfrak{S}_m \ltimes
\{(k_1,k_2,\dots,k_m) \in \Z^m:k_1+k_2+\dots+k_m=0\}.$$
In the classification of affine Coxeter groups,
this is the reflection group of (affine) type~$\tilde A_{m-1}$
(cf.\ \cite[Sec.~8.3]{BjBr}). 
\end{rems}

Our next identities  provide affine extensions of
\begin{equation} \label{eq:even-sp}
\underset{\la\text{ even}}{\sum_{\la:\lambda_1 \le 2h }} s_\lambda(\vx)
 =
\det_{1 \le i, j \le h} \left(
 f_{-i+j}(\vx) - f_{i+j}(\vx)
\right),
\end{equation}
where a partition $\lambda$ is called \emph{even} if all of its parts
are even (cf.\ \cite[Eq.~(7.2)]{Stembridge1990}, modulo an application
of the 
symmetric function involution interchanging complete homogeneous and
elementary symmetric functions).
We point out that the right-hand side of this identity is
(essentially) an irreducible character of rectangular shape
of the  symplectic group $\text{Sp}_{2n}({\mathbb  C})$
(see~\cite[Cor.~7.4(b)]{Stembridge1990}).

\begin{thm}[\sc Two affine bounded Littlewood identities: symplectic case]
\label{thm:C}
Let $h$ and $w$ be positive integers.
For a partition $\lambda$ with length $\le 2h$ and $\lambda_1 - \lambda_{2h} \le w$, we define
\begin{equation}
\label{eq:wt_C}
c^\pm_{2h,w}(\lambda)
 =
\begin{cases}
 1,
 &\text{if $\lambda_{2i-1} = \lambda_{2i}$ for all $1 \le i \le h$,}
\\
 \pm 1,
 &\text{if $\lambda_1 - \lambda_{2h} = w$ and $\lambda_{2i} = \lambda_{2i+1}$ for all $1 \le i \le h-1$,}
\\
 0,
 &\text{otherwise.}
\end{cases}
\end{equation}
Then we have
\begin{align}
  \label{eq:C+}
  \sum_{\lambda \in \Par(2h,w)}
  c^+_{2h,w}(\lambda)\,
  s_{\la[2h,w]'}(\vx)
  &= \det_{1\le i,j\le h}\left(
  \overline{F}_{-i+j,2h+w}(\vx) - \overline{F}_{i+j,2h+w}(\vx)
  \right),\\
  \label{eq:C-}
  \sum_{\lambda \in \Par(2h,w)}
  c^-_{2h,w}(\lambda)\,
s_{\la[2h,w]'}(\vx)
  &= \det_{1\le i,j\le h}\left(
  F_{-i+j,2h+w}(\vx) - F_{i+j,2h+w}(\vx)
  \right),
\end{align}
where $F_{r,N}(\vx)$ and $\overline F_{r,N}(\vx)$ are defined
in Theorem~\ref{thm:affine_BK_intro}.
\end{thm}

Clearly, using~\eqref{eq:s-inf},
we see that, in the limit $w \to \infty$, both~\eqref{eq:C+}
and~\eqref{eq:C-} reduce to~\eqref{eq:even-sp}.

\medskip
The last set of identities in this section consists of affine
extensions of the identities
\begin{align} \label{eq:D1}
\underset{\lambda'\text{ even}}{\sum_{\lambda:\lambda_1 \le 2h}}
   s_\lambda(\vx)
  +
\underset{\lambda'\text{ odd}}{\sum_{\lambda:\lambda_1 = 2h }}
   s_\lambda(\vx)
  &=
    \frac{1}{2}
    \det_{1 \le i, j \le h} \left(
    f_{-i+j}(\vx) + f_{i+j-2}(\vx)
    \right),
    \\
 \label{eq:D2}
\underset{\lambda'\text{ even}}{\sum_{\lambda:\lambda_1 \le 2h}}
   s_\lambda(\vx)
  -
\underset{\lambda'\text{ odd}}{\sum_{\lambda:\lambda_1 = 2h}}
   s_\lambda(\vx)
  &=
    e(\vx) \cdot \overline{e}(\vx)
    \det_{1 \le i, j \le h-1} \left(
    f_{-i+j}(\vx) - f_{i+j}(\vx)
    \right),
    \\
 \label{eq:D3}
\underset{\lambda'\text{ even}}{\sum_{\lambda:\lambda_1 \le 2h+1}}
   s_\lambda(\vx)
  +
\underset{\lambda'\text{ odd}}{\sum_{\lambda:\lambda_1 = 2h+1}}
   s_\lambda(\vx)
  &=
    e(\vx)
    \det_{1 \le i, j \le h} \left(
    f_{-i+j}(\vx) - f_{i+j-1}(\vx)
    \right),
    \\
 \label{eq:D4}
\underset{\lambda'\text{ even}}{\sum_{\lambda:\lambda_1 \le 2h+1}}
   s_\lambda(\vx)
  -
\underset{\lambda'\text{ odd}}{\sum_{\lambda:\lambda_1 = 2h+1}}
   s_\lambda(\vx)
  &=
    \overline{e}(\vx)
    \det_{1 \le i, j \le h} \left(
    f_{-i+j}(\vx) + f_{i+j-1}(\vx)
    \right),
\end{align}
where a partition is called \emph{odd} if all of its parts
are odd, and where
\begin{equation} \label{eq:e(x)}
  e(\vx) = \sum_{i \ge 0} e_i(\vx), \qquad
  \overline{e}(\vx) = \sum_{i \ge 0} (-1)^i e_i(\vx).
\end{equation}
These are universal character identities resulting from
Theorem~2.3(3) in \cite{Okada1998}, if seen in combination with
the Weyl denominator formula (cf.\ \cite[Eq.~(24.40)]{FuHaAA} or
\cite[Prop.~1.1]{Okada1998}) and 
determinantal identities that can be found in
\cite[App.~A.64--A.67]{FuHaAA}. The right-hand sides are (essentially)
sums, or differences, of two irreducible characters of
(signed) rectangular shape of the even orthogonal group
$\text{SO}_{2n}(\mathbb C)$.

\begin{thm}[\sc Four affine bounded Littlewood identities: even
    orthogonal case]
\label{thm:D}
For a partition $\lambda$ of length $\le m$ for which $\lambda_1 - \lambda_m \le w$, we define
\begin{equation}
\label{eq:wt_D}
  d^\pm_{m,w}(\lambda) =
  \begin{cases}
    1,
    &\text{if $\lambda_1, \dots, \lambda_m \in 2 \Z$,}
    \\
    \pm 1,
    &\text{if $\lambda_1, \dots, \lambda_m \in 2 \Z + 1$,}
    \\
    0,
    &\text{otherwise.}
  \end{cases}
\end{equation}
For an integer \( h\ge1 \) and an even integer \( w\ge2 \), we have
\begin{align}
  \sum_{\lambda \in \Par(2h,w)}
  d^+_{2h,w}(\lambda)\,
s_{\la[2h,w]'}(\vx)
  &= \frac{1}{2}\det_{1\le i,j\le h}\left(
  \overline{F}_{-i+j,2h+w}(\vx) + \overline{F}_{i+j-2,2h+w}(\vx)
  \right),
\label{eq:D1a}\\
  \sum_{\lambda \in \Par(2h,w)}
  d^-_{2h,w}(\lambda)\,
s_{\la[2h,w]'}(\vx)
  &= e(\vx) \cdot \overline{e}(\vx)
  \det_{1\le i,j\le h-1}\left(
  F_{-i+j,2h+w}(\vx) - F_{i+j,2h+w}(\vx)
  \right),
\label{eq:D2a}\\
  \sum_{\lambda \in \Par(2h+1,w)}
  d^+_{2h+1,w}(\lambda)\,
s_{\la[2h+1,w]'}(\vx)
  &= e(\vx) \det_{1\le i,j\le h}\left(
  F_{-i+j,2h+1+w}(\vx) - F_{i+j-1,2h+1+w}(\vx)
  \right),
\label{eq:D3a}\\
  \sum_{\lambda \in \Par(2h+1,w)}
  d^-_{2h+1,w}(\lambda)\,
s_{\la[2h+1,w]'}(\vx)
  &= \overline{e}(\vx)\det_{1\le i,j\le h}\left(
  \overline{F}_{-i+j,2h+1+w}(\vx) + \overline{F}_{i+j-1,2h+1+w}(\vx)
  \right),
\label{eq:D4a}
\end{align}
where $F_{r,N}(\vx)$ and $\overline F_{r,N}(\vx)$ are defined
in Theorem~\ref{thm:affine_BK_intro},
the determinant of an empty matrix is defined to be~$1$, and
$e(\vx)$ and $\overline{e}(\vx)$ are given in~\eqref{eq:e(x)}.
\end{thm}

Again using~\eqref{eq:s-inf},
we see that, in the limit $w \to \infty$, the
identities~\eqref{eq:D1a}--\eqref{eq:D4a}
reduce to~\eqref{eq:D1}--\eqref{eq:D4}, respectively.

\begin{rem}
Obviously, the reader will ask what happens when $w$ is odd.
What we can tell is that all of~\eqref{eq:D1a}--\eqref{eq:D4a} do not
hold if $w$ is odd. We do not know whether it is
possible to modify the right-hand sides
in order to obtain valid identities, or whether there are no
determinantal formulas at all for odd~$w$.
\end{rem}

\section{Pfaffians and sums of minors}
\label{sec:aux}

In this section, we recall two Pfaffian/determinantal
formulas, which will be crucial in the proofs of
the affine bounded Littlewood identities
in Section~\ref{sec:symm-funct-ident}, given in the coming two
sections. These are the minor summation formula of Ishikawa and
Wakayama (see Theorem~\ref{thm:IsWa}) --- of which we state an important
special case in Corollary~\ref{cor:IsWa} separately --- and Gordon's
Pfaffian-to-determinant reduction (see Lemma~\ref{lem:Gordon5})
together with simple consequences given in Corollary~\ref{cor:Gordon}.
The reader should recall (cf.\ \cite[Sec.~2]{Stembridge1990}) that
Pfaffians are defined for upper triangular arrays. As usual, we extend
the Pfaffian to skew-symmetric matrices~$A$, with the understanding
that $\Pf A$ is by definition the Pfaffian of the upper triangular
part of~$A$. 
In the sequel, for a matrix \( M \), its transpose is denoted by \( M^t \).

\begin{defn}\label{defn:A_K}
  Let \( A=(a_{i,j})_{i\in I, j\in J} \) be a matrix and let \(R=(r_1,\dots,r_p)
  \) and \( S=(s_1,\dots,s_q) \) be sequences of row and column indices respectively.
  We define
  \[
    A^R_S =\left(a_{r_i,s_j}\right)_{1\le i\le p,1\le j\le q}.
  \]
  We also define \( (r_1, \dots, r_p) \sqcup (r'_1, \dots, r'_{p'}) = (r_1,\dots,r_p,r'_1,\dots,r'_{p'}) \).
\end{defn}

The minor summation formula of Ishikawa and
Wakayama is the following.

\begin{thm}[\sc {\cite[Th.~1]{Ishikawa1995}}]
\label{thm:IsWa}%
Let $m$ and $p$ be positive integers
and $M = ( M_{i,j} )_{1 \le i \le m, 1 \le j \le p}$ any $m \times p$ matrix.
\begin{enumerate}
\item[(1)]
If $m$ is even
and $A = (a_{r,s})_{1 \le r,s \le p}$ is any $p \times p$ skew-symmetric matrix,
then we have
\[
\sum_{K}
 \Pf \left( A^K_K \right)
 \det \left( M_K^{[m]} \right)
 =
\Pf \left( M \,A \, M^t \right)
 =
\Pf_{1 \le i < j \le m}
 \left(
  \sum_{r,s=1}^p a_{r,s} M_{i,r} M_{j,s}
 \right),
\]
where $K=(k_1,\dots,k_m)$ runs over all increasing sequences \( 1\le
k_1<\dots<k_m\le p \) of integers, and $[m]:= (1,2,\dots,m)$.
\item[(2)]
If $m$ is odd
and $A = (a_{r,s})_{0 \le r,s \le p}$ is any $(p+1) \times (p+1)$ skew-symmetric matrix,
then we have
\[
\sum_{K}
 \Pf \left( A^{(0) \sqcup K}_{(0) \sqcup K} \right)
 \det \left( M_K^{[m]} \right)
 =
\Pf_{0 \le i < j \le m}
 \left(
  \begin{cases}
   \sum_{r=1}^p a_{0,r} M_{j,r}, &\text{if $i=0$} \\
   \sum_{r,s=1}^p a_{r,s} M_{i,r} M_{j,s} ,&\text{if $i>0$}
  \end{cases}
 \right),
\]
where $K=(k_1,\dots,k_m)$ runs over all increasing sequences \( 1\le
k_1<\dots<k_m\le p \) of integers.
\end{enumerate}
\end{thm}

We note that Theorem~\ref{thm:IsWa} is also valid in the limiting case
$p=\infty$ provided that the limits of both sides exist.

For convenience, we state the special case where $A$ is
the skew-symmetric matrix with all $1$'s above
the diagonal, which, in abuse of notation, we write as $(1)$.
It should be noted that
$\Pf_{1\le i<j\le 2p}(1)=1$ for all $p$ (see
\cite[Prop.~2.3(c)]{Stembridge1990}). 

\begin{cor}[\sc {\cite[Th.~3]{Okada1989}}]
\label{cor:IsWa}%
Let $m$ and $p$ be positive integers
and $M = ( M_{i,j} )_{1 \le i \le m, 1 \le j \le p}$ any $m \times p$ matrix.
\begin{enumerate}
\item[(1)]
If $m$ is even, then we have
\[
\sum_{K}
 \det M_K^{[m]}
 =
\Pf_{1 \le i < j \le m}
 \left(
  \sum_{1 \le r < s \le p} \left( M_{i,r} M_{j,s} - M_{i,s} M_{j,r} \right)
 \right),
\]
where $K=(k_1,\dots,k_m)$ runs over all increasing sequences \( 1\le
k_1<\dots<k_m\le p \) of integers.
\item[(2)]
If $m$ is odd, then we have
\[
\sum_{K}
 \det M_K^{[m]}
 =
\Pf_{0 \le i < j \le m}
 \left(
  \begin{cases}
   \sum_{1 \le r \le p} M_{j,r}, &\text{if $i=0$} \\
   \sum_{1 \le r < s \le p} \left( M_{i,r} M_{j,s} - M_{i,s} M_{j,r} \right),
    &\text{if $i>0$}
  \end{cases}
 \right),
\]
where $K=(k_1,\dots,k_m)$ runs over all increasing sequences \( 1\le
k_1<\dots<k_m\le p \) of integers.
\end{enumerate}
\end{cor}

The following identity between a Pfaffian and a determinant is due to
Gordon.

\begin{lem}[\sc {\cite[Lem.~1]{Gordon5}}]
\label{lem:Gordon5}
If the quantities $z_i$, $i \in \Z$, satisfy $z_{-i} = - z_i$, then we have
\begin{equation}
  \label{eq:Gordon}
  \Pf_{1 \le i, j \le 2h} \left( z_{j-i} \right)
  =
  \det_{1 \le i, j \le h} \left(  z_{|j-i|+1}+z_{|j-i|+3}+z_{|j-i|+5} +\cdots + z_{i+j-1}\right).
\end{equation}
\end{lem}

It is convenient to rewrite the determinant above in the following form.
Equation~\eqref{eq:Gordon_var1} below
was used in the proof of~\cite[Th.~7.1(a)]{Stembridge1990}.

\begin{cor} \label{cor:Gordon}
If the quantities $z_i$, $i \in \Z$, satisfy $z_{-i} = - z_i$, then we have
\begin{align}
\label{eq:Gordon_var1}
\Pf_{1 \le i, j \le 2h} \left( z_{j-i} \right)
 &=
\det_{1 \le i, j \le h} \left(
 \sum_{s=0}^{2\min(i,j)-2} (-1)^s \left( z_{i+j-1-s} - z_{i+j-2-s} \right)
\right)
\\
\label{eq:Gordon_var2}
 &=
\det_{1 \le i, j \le h} \left(
 \sum_{s=0}^{2\min(i,j)-2} \left( z_{i+j-1-s} + z_{i+j-2-s} \right)
\right)
\\
\label{eq:Gordon_var3}
 &=
\det_{1 \le i, j \le h} \left(
 \begin{cases}
  z_1, &\text{if $i=j=1$} \\
  z_i - z_{i-2}, &\text{if $j=1$ and $i \ge 2$} \\
  z_j - z_{j-2}, &\text{if $i=1$ and $j \ge 2$} \\
  z_{i+j-1} - z_{i+j-3} + z_{|j-i|+1} - z_{|j-i|-1}, &\text{if $i \ge 2$ and $j \ge 2$}
 \end{cases}
\right).
\end{align}
\end{cor}

\begin{proof}
Equation \eqref{eq:Gordon_var1} (respectively \eqref{eq:Gordon_var2}) is obtained from~\eqref{eq:Gordon} by
subtracting the $(i-1)$-st row from the $i$-th
(respectively adding the $(i-1)$-st row to the $i$-th), $i = h, h-1, \dots, 2$,
and then doing the analogous column operations.
Similarly, we obtain \eqref{eq:Gordon_var3} from~\eqref{eq:Gordon} by
subtracting the $(i-2)$-nd row from the $i$-th, $i = h, h-1, \dots, 3$
and performing the same operations on the columns.
\end{proof}

\section{Proof of the affine bounded Littlewood identities in Theorem
  \ref{thm:affine_BK_intro}}
\label{sec:first-proof}

In this section, we prove the affine bounded Littlewood identities
in~\eqref{eq:aGBK1s} and~\eqref{eq:aGBK2s}.

For a positive integer \( N \) and an integer \( t \), let \( R_N(t) \) denote
the remainder of \( t \) when divided by \( N \), that is, \( t = \flr{t/N}N +
R_N(t) \) and $0\le R_N(t)<N$. For a statement \( S \) we let \( \chi[S]=1 \)
if the statement \( S \) is true and \( \chi[S]=0 \) otherwise.

\medskip
We start by rewriting and rearranging the terms on the left-hand sides
of~\eqref{eq:aGBK1s} and~\eqref{eq:aGBK2s}, with the goal of expressing
them as sums of minors, so that the minor summation theorem in
Corollary~\ref{cor:IsWa} can be applied.

\begin{lem}\label{lem:rem}
Let \( m\) and \( N \) be positive integers such that \( m<N \) and \( m \) is odd.
Then
\begin{equation} \label{eq:lem1}
  \underset{\mu_1>\dots>\mu_m \text{ \em and } \mu_1-\mu_m<N}
  {\sum_{\mu_{1}, \ldots, \mu_m \in \Z}}\
\underset{k_{1}+\cdots+k_m=0}{\sum_{k_{1}, \ldots, k_m \in \Z}}
  \det_{1 \leq i, j \leq m}\big(e_{\mu_{i}+N k_{i}+j}(\vx)\big)
  =\underset{R_N(\alpha_1)>\dots>R_N(\alpha_m)}
  {\sum_{\alpha_1,\dots,\alpha_m\in\Z}}
  \det_{1 \leq i, j \leq m}\big(e_{\alpha_i+j}(\vx)\big).
\end{equation}
\end{lem}

\begin{rem} \label{rem:2}
In view of \eqref{eq:CJT}, the left-hand side of~\eqref{eq:aGBK1s}
equals (cf.~\eqref{eq:ABL1})
$$
\underset{ \la_1-\la_{2h+1}\le w}{\sum_{\la:\ell(\la)\le2h+1}}\
\underset{k_1+\dots+k_{2h+1}=0}{\sum_{ k_1,\dots,k_{2h+1}\in\Z}}
\det_{1\le i,j\le 2h+1}\big(
e_{\la_i-i+j+(2h+w+1)k_i}(\vx)\big).
$$
If we now do the substitution $\la_i=\mu_i+i$, then we see that we
obtain the left-hand side of~\eqref{eq:lem1} with $m=2h+1$ and $N=2h+1+w$.
\end{rem}

\begin{proof}[Proof of Lemma \ref{lem:rem}]
  Let \( U \) be the set of pairs \( (\mu,k) \) of \( m \)-tuples \(
  \mu=(\mu_1,\dots,\mu_m)\in\Z^m \) and \( k=(k_1,\dots,k_m)\in\Z^m \) with
  \( \mu_1>\dots>\mu_m \), \( \mu_1-\mu_m<N \), and \( k_{1}+\cdots+k_m=0
  \). Let \( V \) be the set of \( m \)-tuples \(
  \alpha=(\alpha_1,\dots,\alpha_m)\in\Z^m \) for which \(
  R_N(\alpha_1)>\dots>R_N(\alpha_m) \).
  Then what we need to show is
  \begin{equation}
    \label{eq:mu,k=alpha}
    \sum_{(\mu,k)\in U}
    \det_{1 \leq i, j \leq m}\big(e_{\mu_{i}+N k_{i}+j}(\vx)\big)
    = \sum_{\alpha\in V}
    \det_{1 \leq i, j \leq m}\big(e_{\alpha_i+j}(\vx)\big).
  \end{equation}

  Suppose \( (\mu,k)\in U \). Let \( q = \flr{\mu_m/N} \) and, for \( 1\le i\le m
  \), let \( r_i \) be the integer such that \( \mu_i = qN + r_i\).
  Then there is a unique integer \( 0\le t<m \) such that
\begin{equation}\label{eq:rem_seq}
  N+r_m > r_1>\dots>r_t \ge N> r_{t+1}>\dots>r_m \ge0.
\end{equation}
Now let \( \beta=(\beta_1,\dots,\beta_m) \), where \( \beta_i=\mu_i+Nk_i \).
By~\eqref{eq:rem_seq}, we have
\begin{equation}\label{eq:Q_N(beta)}
  \sum_{i=1}^m \flr{\beta_i/N} = \sum_{i=1}^m (\flr{\mu_i/N}+k_i)
  = \sum_{i=1}^m \flr{\mu_i/N} = qm + t.
\end{equation}
Moreover, we have \( R_N(\beta_i) = R_N(\mu_i)\), which is equal to \( r_i \) if
\(t+1\le i\le m \) and to \( r_i-N \) if \( 1\le i\le t \). Since the integers \(
q,t\in\Z \) with \( 0\le t<m \) are determined by \( \beta \) via
to~\eqref{eq:Q_N(beta)}, the pair \( (\mu,k) \) can be recovered from the sequence
\( \beta \) as follows:
\begin{align}
  \label{eq:mu=beta}
\mu_i &=
\begin{cases}
  (q+1)N + R_N(\beta_i), & \mbox{if \( 1\le i\le t \)},\\
  qN + R_N(\beta_i), & \mbox{if \(t+1\le i\le m \)},
\end{cases}\\
  \label{eq:k=b-m}
k_i & = \frac{\beta_i-\mu_i}{N} .
\end{align}

On the other hand, by \eqref{eq:rem_seq}, we have
\begin{equation}\label{eq:R_N(beta)}
  R_N(\beta_{t+1})>\dots>R_N(\beta_m) > R_N(\beta_1)>\dots>R_N(\beta_t).
\end{equation}
Let \( \alpha=(\alpha_1,\dots,\alpha_m)=
(\beta_{t+1},\dots,\beta_m,\beta_1,\dots,\beta_t) \). By~\eqref{eq:R_N(beta)},
we have \( \alpha\in V \).

Observe that \( \det\big(e_{\mu_{i}+N k_{i}+j}(\vx)\big) =
\det\big(e_{\beta_i+j}(\vx)\big)
=\det\big(e_{\alpha_i+j}(\vx)\big) \) because cyclically shifting indices of an \( m\times m
\) matrix does not change its determinant when \( m \) is odd. Thus, to prove~\eqref{eq:mu,k=alpha}, it suffices to show that the map \( (\mu,k)\mapsto \alpha
\) is a bijection between \( U \) and \( V \).

Suppose \( \alpha\in V \). Let \( q \) and \( t \) be the unique integers
satisfying \( \sum_{i=1}^m \flr{\alpha_i/N} = qm+t \) and \( 0\le t<m \). Let \(
\beta=(\alpha_{m-t+1},\dots,\alpha_m,\alpha_1,\dots,\alpha_{m-t}) \). Finally,
define \( (\mu,k) \) using~\eqref{eq:mu=beta} and~\eqref{eq:k=b-m}. It is easy
to check that the map \( \alpha\mapsto (\mu,k) \) is the desired inverse map.
\end{proof}

\begin{lem}\label{lem:rem2}
  Let \( m \) and \( N \) be positive integers with \( m<N \) and \(
  m \) is even. Then
  \begin{equation} \label{eq:lem2}
    \underset{\mu_1>\dots>\mu_m \text{ \em and } \mu_1-\mu_m<N}
    {\sum_{\mu_{1}, \ldots, \mu_m \in \Z }}\
   \underset{k_{1}+\cdots+k_m=0}{ \sum_{k_{1}, \ldots, k_m \in \Z}}
    \det_{1 \leq i, j \leq m}\big(e_{\mu_{i}+N k_{i}+j}(\vx)\big)
    =\underset{R_N(\alpha_1)>\dots>R_N(\alpha_m)}
    {\sum_{\alpha_1,\dots,\alpha_m\in\Z}} (-1)^{\sum_{i=1}^m \flr{\alpha_i/N}}
    \det_{1\le i,j\le m}\big(e_{\alpha_i+j}(\vx)\big).
  \end{equation}
\end{lem}

\begin{rem} \label{rem:3}
Similarly to Remark~\ref{rem:2},
in view of \eqref{eq:CJT}, the left-hand side of~\eqref{eq:aGBK2s}
equals (cf.~\eqref{eq:ABL2})
the left-hand side of~\eqref{eq:lem2} with $m=2h$ and $N=2h+w$.
\end{rem}

\begin{proof}[Proof of Lemma \ref{lem:rem2}]
  This can be proved by the same arguments as in the proof of
  Lemma~\ref{lem:rem}. 
  The only difference is that, if \( (\mu,k) \) corresponds to \(
  \alpha \), then 
  \( \det\big(e_{\mu_{i}+N k_{i}+j}(\vx)\big) =
  \det\big(e_{\beta_i+j}(\vx)\big)
  =(-1)^t\det\big(e_{\alpha_i+j}(\vx)\big) \)
  because cyclically shifting indices of an \(
  m\times m \) matrix changes the sign of its determinant when \( m \) is even.
  Since \( (-1)^t = (-1)^{ \sum_{i=1}^m \flr{\alpha_i/N} - qm} = (-1)^{
    \sum_{i=1}^m \flr{\alpha_i/N}} \), we obtain the desired formula.
\end{proof}

In view of Remark~\ref{rem:2}, the identity~\eqref{eq:aGBK1s} will
follow once we show the following equality.

\begin{prop}\label{thm:odd2}
  For nonnegative integers \( h \) and \( N \) with \( N>2h+1 \), we have
  \[
  \underset{R_N(\alpha_1)>\dots>R_N(\alpha_{2h+1})}
  {\sum_{\alpha_1,\dots,\alpha_{2h+1}\in\Z}} \det_{1\le i,j\le
      2h+1}\big(e_{\alpha_i+j}(\vx)\big)
    =\sum_{k\ge0}e_k(\vx) \det_{1\le i,j\le h}\left(F_{-i+j,N}(\vx) -
    F_{i+j,N}(\vx) \right).
  \]
\end{prop}

Similarly,
in view of Remark~\ref{rem:3}, the identity~\eqref{eq:aGBK2s} will
follow once we show the equality in the proposition below.

\begin{prop}\label{thm:even2}
  For positive integers \( h \) and \( N \) with \( N>2h \), we have
  \[
  \underset{R_N(\alpha_1)>\dots>R_N(\alpha_{2h})}
  {\sum_{\alpha_1,\dots,\alpha_{2h}\in\Z}}
    (-1)^{\sum_{i=1}^{2h} \flr{\alpha_i/N}}\det_{1\le i,j\le
      2h}\big(e_{\alpha_i+j}(\vx)\big)
    = \det_{1 \leq i, j \leq h}\left(\overline{F}_{-i+j,N}(\vx)
    +\overline{F}_{i+j-1,N}(\vx)\right).
  \]
\end{prop}

In the next two subsections, we prove Propositions~\ref{thm:odd2} and
\ref{thm:even2}. Our approach parallels the one in Stembridge's proof of the
bounded Littlewood identities~\eqref{eq:BK_odd1}
and~\eqref{eq:BK_even1} (see \cite[Th.~7.1]{Stembridge1990}).

\subsection{Proof of Proposition~\ref{thm:odd2}}

Let \( h \) and \( N \) be nonnegative integers satisfying \( N>2h+1 \).
Throughout this subsection we write \( e(\vx)=\sum_{k\ge0}e_k(\vx) \), as
in Theorem~\ref{thm:D} before, and \( E\) is the
\( 1\times (2h+1) \) matrix whose entries are all equal to \( e(\vx) \).

For two integers \( i \) and \( j \), we define
\[
  d_N(i,j) =
  \underset{R_N(m-i) > R_N(n-j)}{\sum_{m,n\in \Z}} e_m(\vx) e_n(\vx)
  - \underset{R_N(m-i) < R_N(n-j)}{\sum_{m,n\in \Z}} e_m(\vx) e_n(\vx).
\]
By definition, \( d_N(j,i) = -d_N(i,j) \) and \( d_N(i,i)=0 \). For a
nonnegative integer \( n \), let \( D_N(n) \) denote the \( n\times n \)
skew-symmetric matrix given by
\[
D_{N}(n) = \big(d_N(i,j)\big)_{1\leq i,j \leq n}.
\]

We will prove the following three identities:
\begin{gather}
  \label{eq:odd_id1}
  \underset{R_N(\alpha_1)>\dots>R_N(\alpha_{2h+1})}
  {\sum_{\alpha_1,\dots,\alpha_{2h+1}\in\Z}} \det_{1\le i,j\le 2h+1}
  \big(e_{\alpha_i+j}(\vx)\big)
  = \Pf
  \begin{pmatrix}
    0 & E \\
    -E^t & D_{N}(2h+1)
  \end{pmatrix},\\
  \label{eq:odd_id2}
  \Pf
  \begin{pmatrix}
    0 & E \\
    -E^t & D_{N}(2h+1)
  \end{pmatrix}
  = e(\vx) \Pf_{1\le i,j\le 2h} \left( F_{j-i-1,N}(\vx) - F_{j-i+1,N}(\vx)
  \right),\\ 
  \label{eq:odd_id3}
  \Pf_{1\le i,j\le 2h} \left( F_{j-i-1,N}(\vx) - F_{j-i+1,N}(\vx) \right)
  = \det_{1\le i,j\le h} \left( F_{-i+j,N}(\vx) - F_{i+j,N}(\vx)\right).
\end{gather}
It is obvious that \eqref{eq:odd_id1}--\eqref{eq:odd_id3} together
yield Proposition~\ref{thm:odd2}, as desired.

In the remainder of this subsection, we provide the proofs
of~\eqref{eq:odd_id1}--\eqref{eq:odd_id3}.

\begin{proof}[Proof of \eqref{eq:odd_id1}]
  By taking the transpose we may rewrite the left-hand side
  of~\eqref{eq:odd_id1} as 
\begin{equation}\label{eq:1rN}
  \sum_{N>r_1>\dots>r_{2h+1}\ge0}\sum_{k_1,\dots,k_{2h+1}\in\Z}
  \det_{1\le i,j\le 2h+1} \left( e_{k_jN+r_j+i}(\vx)\right)
  = \sum_{1\le r_1<\dots<r_{2h+1}\le N} (-1)^{\binom{2h+1}{2}}
\det_{1\le i,j\le 2h+1} \left( \sum_{k\in\Z}e_{kN+i+r_j-1}(\vx)\right).
\end{equation}
Let \( M=(M_{i,j})_{1\le i\le 2h+1, 1\le j\le N} \) be the matrix whose \( (i,j)
\)-entry is $M_{i,j}=\sum_{k\in\Z}e_{kN+i+j-1}(\vx)$. Note that, for \( 1\le j\le 2h+1
\), we have
\[ \sum_{1\le r\le N} M_{j,r} = \sum_{1\le r\le N}\sum_{k\in \Z}
e_{kN+r+j-1}(\vx) =\sum_{\ell \in\Z} e_\ell(\vx) = e(\vx), \]
and, for \( 1\le i<j\le 2h+1 \),
\begin{align*}
  \sum_{1 \le r < s \le N} \left( M_{i,r} M_{j,s} - M_{i,s} M_{j,r} \right)
 &=\sum_{1 \le r < s \le N}\sum_{k,\ell\in\Z}
  \left( e_{kN+i+r-1}(\vx)e_{\ell N+j+s-1}(\vx) -
  e_{kN+i+s-1}(\vx)e_{\ell N+j+r-1}(\vx) \right)\\
  &= - d_N(i,j).
\end{align*}
Thus, by the minor summation formula in
Corollary~\ref{cor:IsWa}(2) with $n=2h+1$ and $p=N$, the right-hand side
of~\eqref{eq:1rN} equals
\[
  (-1)^h
  \Pf
  \begin{pmatrix}
    0 & E \\
    -E^t & -D_{N}(2h+1)
  \end{pmatrix}
  = \Pf
  \begin{pmatrix}
    0 & E \\
    -E^t & D_{N}(2h+1)
  \end{pmatrix},
\]
as desired.
\end{proof}

\begin{proof}[Proof of \eqref{eq:odd_id2}]
  Let \( A \) be the matrix on the left-hand side of~\eqref{eq:odd_id2}. In the
  matrix \( A \) subtract row/column \( i-1 \) from row/column \( i \) for \(
  i=2h+2,2h+1,\dots, 2 \). Then the resulting matrix is of the form
  \[
    \begin{pmatrix}
      0 & e(\vx) & 0 \\
      -e(\vx) & * & * \\
      0 & * & B \\
    \end{pmatrix},
  \]
  where \( B=(B_{i,j})_{1\le i,j\le 2h} \) is the matrix whose \( (i,j) \)-entry
  is
\[
  B_{i,j}= d_N(i,j) -d_N(i-1,j) - d_N(i,j-1) + d_N(i-1,j-1).
\]
Since \( \Pf A = e(\vx) \Pf B \), it remains to show that
\(   B_{i,j} = F_{j-i-1,N}(\vx) - F_{j-i+1,N}(\vx) \).

We claim that
\begin{equation}
  \label{eq:dN-dN}
  d_N(i,j) - d_N(i-1,j)
  = 2\sum_{m,n\in \Z} \chi\big[R_N(m-i) = N-1\big]\, e_m(\vx) e_n(\vx)
  - F_{j-i,N}(\vx) - F_{j-i+1,N}(\vx).
\end{equation}
To prove the claim note that \( d_N(i,j) - d_N(i-1,j) = P-Q \), where
\begin{align*}
  P &= \sum_{m,n\in \Z}
      \Big(\chi\big[R_N(m-i) > R_N(n-j)\big] - \chi\big[R_N(m-i+1) > R_N(n-j)\big]\Big)\,
      e_m(\vx) e_n(\vx),\\
  Q &= \sum_{m,n\in \Z}
      \Big(\chi\big[R_N(m-i) < R_N(n-j)\big] - \chi\big[R_N(m-i+1) < R_N(n-j)\big]\Big)\,
      e_m(\vx) e_n(\vx).
\end{align*}
One can easily check that the coefficient of \( e_m(\vx)e_n(\vx) \) in
\( P \) is equal to 
\begin{multline*}
  \chi\big[R_N(m-i) = N-1\ne R_N(n-j)\big] - \chi \big[R_N(m-i)
  = R_N(n-j)\ne N-1\big]\\
 =   \chi\big[R_N(m-i) = N-1\big] - \chi \big[R_N(m-i)
  = R_N(n-j)\big].
\end{multline*}
Thus, we have
\[
    P = \sum_{m,n\in \Z}
    \chi\big[R_N(m-i) = N-1\big]e_m(\vx) e_n(\vx) -
    \sum_{m,n\in \Z} \chi \big[R_N(m-i) = R_N(n-j)\big] e_m(\vx)e_n(\vx),
\]
and similarly,
\[
  Q =
  \sum_{m,n\in \Z} \chi \big[R_N(m-i+1) = R_N(n-j)\big] e_m(\vx)e_n(\vx)
- \sum_{m,n\in \Z} \chi\big[R_N(m-i) = N-1\big]e_m(\vx) e_n(\vx).
\]
 Since
 \[
   \sum_{m,n\in \Z} \chi[R_N(m-i) = R_N(n-j)]e_m(\vx) e_n(\vx)
   =\sum_{m\in\Z}\sum_{k\in\Z} e_m(\vx)e_{kN+m+j-i}(\vx)
   =\sum_{k\in\Z} f_{kN+j-i}(\vx) = F_{j-i,N}(\vx),
 \]
 \( P-Q \) is equal to the right-hand side of~\eqref{eq:dN-dN},
 and we obtain the claim.

 By \eqref{eq:dN-dN} we have
\[
  B_{i,j} = \left( - F_{j-i,N}(\vx) - F_{j-i+1,N}(\vx) \right) -
  \left( - F_{j-i-1,N}(\vx) - F_{j-i,N}(\vx) \right)
  = F_{j-i-1,N}(\vx) - F_{j-i+1,N}(\vx),
\]
and the proof is completed.
\end{proof}

\begin{proof}[Proof of \eqref{eq:odd_id3}]
  If \( z_i=F_{i-1,N}(\vx)-F_{i+1,N}(\vx) \), then
\[
  z_{|j-i|+1}+ z_{|j-i|+3} +\cdots+  z_{i+j-1}
  = F_{|j-i|,N}(\vx) -F_{i+j,N}(\vx)
  = F_{j-i,N}(\vx) -F_{i+j,N}(\vx) .
\]
Thus \eqref{eq:odd_id3} follows from Lemma~\ref{lem:Gordon5}.
\end{proof}

As pointed out before the statement of Proposition~\ref{thm:odd2},
with this proposition being established, the
identity~\eqref{eq:aGBK1s} follows.

\subsection{Proof of Proposition~\ref{thm:even2}}
\label{sec:pf1_even}

Let \( h \) and \( N \) be nonnegative integers satisfying \( N>2h \).
For two integers \( i \) and \( j \), we define
\[
   \bar d_N(i,j) = \underset{R_N(m-i) > R_N(n-j)}{\sum_{m,n\in \Z}}
  (-1)^{\fl{(m-i)/N}+\fl{(n-j)/N}} e_m(\vx) e_n(\vx)
  - \underset{R_N(m-i) < R_N(n-j)}{\sum_{m,n\in \Z}}
  (-1)^{\fl{(m-i)/N}+\fl{(n-j)/N}} e_m(\vx) e_n(\vx).
\]
By definition, \( \bar d_N(j,i) = -\bar d_N(i,j) \) and \( \bar d_N(i,i)=0 \).

We will prove the following three identities:
  \begin{gather}\label{eq:Pf_bar_d}
    \underset{R_N(\alpha_1)>\dots>R_N(\alpha_{2h})}
    {\sum_{\alpha_1,\dots,\alpha_{2h}\in\Z}}
    (-1)^{\sum_{i=1}^{2h} \flr{\alpha_i/N}}
    \det_{1\le i,j\le 2h}\big(e_{\alpha_i+j}(\vx)\big)
    = \Pf_{1\le i,j\le 2h}\left(\bar d_N(i,j)\right),\\
\label{eq:Pf-Pf}
\Pf_{1\le i,j\le 2h}\left(\bar d_N(i,j)\right)
= \Pf_{1\le i,j\le
  2h}\left(\sum_{r=-j+i+1}^{j-i}\overline{F}_{r,N}(\vx)\right),\\
\label{eq:Pf-det}
\Pf_{1\le i,j\le
  2h}\left(\sum_{r=-j+i+1}^{j-i}\overline{F}_{r,N}(\vx)\right)
=  \det_{1\le i,j\le
  h}\left(\overline{F}_{j-i,N}(\vx)+\overline{F}_{i+j-1,N}(\vx)\right). 
\end{gather}
It is obvious that \eqref{eq:Pf_bar_d}--\eqref{eq:Pf-det} together
yield Proposition~\ref{thm:even2}, as desired.

In the remainder of this subsection, we provide the proofs
of~\eqref{eq:Pf_bar_d}--\eqref{eq:Pf-det}.

\begin{proof}[Proof of \eqref{eq:Pf_bar_d}]
We rewrite the left-hand side of~\eqref{eq:Pf_bar_d} as
  \[
    \sum_{N>r_1>\dots>r_{2h}\ge0}
    \det_{1\le i,j\le 2h}\left( \sum_{k\in\Z} (-1)^k e_{kN+r_j+i}(\vx)\right)
    = \sum_{1\le r_1<\dots<r_{2h}\le N} (-1)^{\binom{2h}{2}}
    \det_{1\le i,j\le 2h} \left( \sum_{k\in\Z} (-1)^ke_{kN+i+r_j-1}(\vx)\right).
  \]
  Then, as in the proof of~\eqref{eq:odd_id1}, we
  obtain~\eqref{eq:Pf_bar_d} using the minor summation formula in
  Corollary~\ref{cor:IsWa}(1) with $n=2h$, $p=N$, and $M_{i,j}=\sum_{k\in \Z}
  (-1)^k e_{kN+i+j-1}(\vx)$.
\end{proof}

\begin{proof}[Proof of \eqref{eq:Pf-Pf}]
We claim that, for any integers \( i,j\in\Z \) with \( i\le j \), we have
  \[
    \bar d_N(i,j)=\sum_{r=-j+i+1}^{j-i}\overline{F}_{r,N}(\vx)
    =\sum_{r\in\Z: r\le j-i}\overline{F}_{r,N}(\vx) -  \sum_{r\in\Z: r\le i-j}\overline{F}_{r,N}(\vx).
    \]
Clearly, this would establish \eqref{eq:Pf-Pf}.

  If \( j=i \), then both sides of the equation are equal to zero. Thus, by
  induction on \( j-i \), it suffices to show that, for \( i<j \), we have
\[
  \bar d_N(i,j)-\bar d_N(i,j-1)
  =\overline{F}_{-j+i+1,N}(\vx) +\overline{F}_{j-i,N}(\vx)
  =\overline{F}_{j-i-1,N}(\vx) +\overline{F}_{j-i,N}(\vx).
\]
Using the definition of $\bar d_N(i,j)$, we have
\[
  \bar d_N(i,j)-\bar d_N(i,j-1)
  =\sum_{m,n\in\Z} (-1)^{\fl{(m-i)/N}+\fl{(n-j)/N}} a_{m,n}(i,j)
  e_m(\vx)e_n(\vx) ,
\]
where
\begin{multline*}
  a_{m,n}(i,j)=  \chi\big[R_N(m-i) > R_N(n-j)\big]
  - \chi\big[R_N(m-i) < R_N(n-j)\big]\\
  -  (-1)^{\chi\big[R_N(n-j)=N-1\big]}
  \Big(
  \chi\big[R_N(m-i) > R_N(n-j+1)\big] - \chi\big[R_N(m-i) < R_N(n-j+1)\big]
  \Big).
\end{multline*}
On the other hand, we have
\begin{align*}
  \overline{F}_{j-i,N}(\vx)
  &= \sum_{k\in\Z} (-1)^k f_{kN+j-i}(\vx)
    = \sum_{k\in\Z} (-1)^k \sum_{m\in\Z}e_m(\vx) e_{m+kN+j-i}(\vx)\\
  &= \sum_{m,n\in\Z} (-1)^{(n-m-j+i)/N}
    \chi\big[ R_N(m-i)=R_N(n-j) \big] e_m(\vx)e_n(\vx).
\end{align*}
Note that, if \( R_N(m-i)=R_N(n-j) \), then
\[
  (-1)^{(n-m-j+i)/N}  = (-1)^{((n-j)-(m-i))/N}
  = (-1)^{\flr{(n-j)/N}-\flr{(m-i)/N}}
  = (-1)^{\flr{(m-i)/N}+\flr{(n-j)/N}}.
\]
Thus,
\[
  \overline{F}_{j-i-1,N}(\vx) +\overline{F}_{j-i,N}(\vx)
  =\sum_{m,n\in\Z} (-1)^{\fl{(m-i)/N}+\fl{(n-j)/N}} b_{m,n}(i,j)
  e_m(\vx)e_n(\vx) ,
\]
where
\[
  b_{m,n}(i,j)= \chi\big[R_N(m-i)=R_N(n-j)\big]
  +  (-1)^{\chi\big[R_N(n-j)=N-1\big]}
  \chi\big[R_N(m-i)=R_N(n-j+1)\big]
.
\]
By considering the two cases \( R_N(n-j)=N-1 \) and \( R_N(n-j)\ne N-1 \)
separately, one can easily check that \( a_{m,n}(i,j)=b_{m,n}(i,j) \) for all \(
m,n\in\Z \), which proves~\eqref{eq:Pf-Pf}.
\end{proof}

\begin{proof}[Proof of \eqref{eq:Pf-det}]
Let \( z_i = \sum_{r\le i}\overline{F}_{r,N}(\vx) - \sum_{r\le
  -i}\overline{F}_{r,N}(\vx) \). Then, by~\eqref{eq:Pf-Pf}
and~\eqref{eq:Gordon_var1}, we have
\[
  \Pf_{1\le i,j\le 2h}\left(\bar d_N(i,j)\right)
  =\Pf_{1 \le i, j \le 2h} \left( z_{j-i} \right)
  = \det_{1 \le i, j \le h} \left(
    \sum_{s=0}^{2\min(i,j)-2} (-1)^s \left( z_{i+j-1-s} - z_{i+j-2-s} \right)\right).
\]
Using the fact \(
z_i-z_{i-1}=\overline{F}_{i,N}(\vx)+\overline{F}_{-i+1,N}(\vx)=\overline{F}_{i,N}(\vx)+\overline{F}_{i-1,N}(\vx)
\), we see that the above determinant equals
\[
  \det_{1\le i,j\le h}\left(\overline{F}_{j-i,N}(\vx)+\overline{F}_{i+j-1,N}(\vx)\right),
  \]
as desired.
\end{proof}

As pointed out before the statement of Proposition~\ref{thm:even2},
with this proposition being established, the
identity~\eqref{eq:aGBK2s} follows.

\section{A systematic approach to affine bounded Littlewood identities}
\label{sec:second-proof}

In this section, we develop a general approach to deriving affine
bounded Littlewood identities, see Subsection~\ref{sec:LG}.
It is based on the essentials of the line of argument of the proof of
Theorem~\ref{thm:affine_BK_intro} in the previous section. It is
however more general as we allow the application of the full minor
summation formula of Theorem~\ref{thm:IsWa}, as opposed to ``just" its
special case in Corollary~\ref{cor:IsWa} that we used in
Section~\ref{sec:first-proof}.
This approach allows us to provide an alternative proof of the affine
bounded Littlewood in Theorem~\ref{thm:affine_BK_intro}, that is,
different from the previous section. Moreover, it also
gives us the means to prove the affine bounded Littlewood
identities in Theorem~\ref{thm:C}, see Subsection~\ref{sec:LC},
as well as the affine bounded Littlewood
identities in Theorem~\ref{thm:D}, see Subsection~\ref{sec:LD}.

\subsection{General framework} \label{sec:LG}

Recall that, for positive integers $m$ and $w$, we denote by $\Par(m,w)$
the set of partitions of length at most $m$ satisfying $\lambda_1 -
\lambda_m \le w$. 
In this subsection, we consider sums of the form
\[
\sum_{\lambda \in \Par(m,w)}
\underset{k_1 + \dots + k_m = 0}{\sum_{k_1, \dots, k_m \in \Z}} 
 u(\lambda) \det_{1 \le i, j \le m} \left( e_{\lambda_i-i+j+Nk_i}(\vx) \right),
\]
where $N = m+w$, and $u$ is a statistic on $\Par(m,w)$, and give Pfaffian
formulas for the summations under certain conditions.

Let, as in Definition~\ref{def:par},
$\Par(m)$ be the set of partitions of length at most $m$.
To such a partition~$\lambda$, we associate the sequence $I_m(\lambda)$ given by
\[
I_m(\lambda) = (\lambda_m+1,\lambda_{m-1}+2, \dots,  \lambda_1+m).
\]
Then the correspondence $\lambda \mapsto I_m(\lambda) = (i_1, \dots, i_m)$ gives a bijection
between $\Par(m)$ and the set of increasing sequences of positive integers of length $m$,
and $\lambda \in \Par(m,w)$ if and only if $i_1 < \dots < i_m < i_1 + N$.

Recall the notation in Definition~\ref{defn:A_K}. For brevity we also define \(
A^K := A^K_K \). Recall that $R_N(t)$ is the remainder of $t$ when divided
by $N$ so that $0 \le R_N(t) \le N-1$.

\begin{prop}
\label{prop:general}
Let $m$ and $w$ be positive integers and put $N = m+w$.
Let $u : \Par(m,w) \to \Z$ be a statistic on $\Par(m,w)$.
Let $p$ be a nonnegative integer such that $p+m$ is even.
Suppose that $A$ is a skew-symmetric matrix  with rows/columns
indexed by the totally ordered set $\{ 0_1 < 0_2 < \dots < 0_p < 1 < 2 < \cdots \}$
satisfying the following three conditions:
\begin{enumerate}
\item[(i)]
For $\lambda \in \Par(m,w)$, we have
\[
\Pf A^{I_0 \sqcup I_m(\lambda)} = u(\lambda),
\]
where $I_0 = (0_1, 0_2, \dots, 0_p)$.
\item[(ii)]
For an increasing sequence $(i_1, \dots, i_m)$ of positive integers with $i_m - i_1 > N$,
we have
\[
\Pf A^{I_0 \sqcup (i_1+N, i_2, \dots, i_{m-1}, i_m-N)}
= \Pf A^{I_0 \sqcup (i_1, i_2, \dots, i_m)}.
\]
\item[(iii)]
For an increasing sequence $(i_1, \dots, i_m)$ of positive integers with
$i_k \equiv i_{\ell} \bmod N$ for some $k<\ell$, we have
\[
\Pf A^{I_0 \sqcup (i_1, \dots, i_m)} = 0.
\]
\end{enumerate}
Then we have
\begin{equation*}
\label{eq:general}
\sum_{\lambda \in \Par(m,w)} u(\lambda)\,
s_{\lambda[m,w]'}(\vx)
=
\Pf \left( T_p A T_p^t \right),
\end{equation*}
where $T_p$ is the following matrix with rows indexed by $\{ 0_1, \dots, 0_p, 1, 2, \dots, m \}$
and columns indexed by $\{ 0_1, \dots, 0_p, 1, 2, \dots \}$:
\begin{equation}
  \label{eq:T_p}
  T_p
  =
  \begin{pmatrix}
    \mathbf{I}_p & O \\
    O & \left( e_{j-i}(\vx) \right)_{1 \le i \le m, \, j \ge 1}
  \end{pmatrix}.
\end{equation}
Here $\mathbf{I}_p$ stands for the identity matrix of size~$p$.
\end{prop}

For the proof of the proposition, we need an auxiliary result,
which essentially extracts the
essence of the proof of Lemma~\ref{lem:rem}. 

\begin{lem}
\label{lem:bijection}
Let $m$ and $w$ be positive integers and let $N = m+w$.
There is a bijection $\Phi$ between
\[
X =
\{ (\lambda,k)\in\Par(m,w)\times \Z^m: k_1+\dots+k_m=0,\text{ and }
\lambda_i+m-i+Nk_i \ge 0 \mbox{
  \em for \(1\le i\le m \)} \}
\]
and
\[
Y =
\{ \ka \in \Par(m) : \text{$R_N(\ka_i + m-i)$, $1 \le i \le m$, \em are distinct} \}
\]
such that, if\/ $\Phi(\lambda,k) = \ka$, then
\[
\det_{1 \le i, j \le m} \left(
 e_{\lambda_i - i + j + N k_i}(\vx)
\right)
 =
\sgn ({\sigma}) \det_{1 \le i, j \le m} \left(
 e_{\ka_i - i + j}(\vx) \right),
\]
where $\sigma\in \mathfrak{S}_m$ is the unique permutation
that rearranges the vector $(\la_i+m-i+Nk_i)_{1\le i\le m}$
in decreasing order.
\end{lem}

\begin{proof}
In the proof of Lemma~\ref{lem:rem}, it was
(implicitly) shown that there is a bijection between the set
$$
U=\{(\mu,k)\in\Z^m\times\Z^m:
\mu_1>\dots>\mu_m ,\  \mu_1-\mu_m<N,\text{ and } k_{1}+\cdots+k_m=0\}
$$
and the set
$$
V'=\{\be\in\Z^m:
  R_N(\beta_{t+1})>\dots>R_N(\beta_m) >
  R_N(\beta_1)>\dots>R_N(\beta_t)
  \text{ for some $t$ with }0\le t< m\}.
$$
More precisely, the bijection is given by $\be_i=\mu_i+Nk_i$, for
$1\le i\le m$.

Now, the set $U$ is in bijection with the set
\begin{align*}
X'&=\{(\la,k)\in\Z^m\times\Z^m:
\la_1\ge \dots\ge \la_m ,\  \la_1-\la_m\le w,\text{ and }
k_{1}+\cdots+k_m=0\}\\ 
&=\{(\la,k)\in\Par(m,w)\times\Z^m:k_{1}+\cdots+k_m=0\}\end{align*}
via $\mu_i=\la_i+m-i$, for $1\le i\le m$. On the other hand, the
set~$V'$ is in bijection with the set
$$
Y' =
\{ \gamma \in \Z^m :\gamma_1>\dots>\gamma_m
\text{ and $R_N(\gamma_i)$, $1 \le i \le m$, are distinct} \}.
$$
Indeed, to go from $V'$ to $Y'$, one orders the components of the
elements $\be\in V'$, while, to go from~$Y'$ to~$V'$, one orders the
remainders $R_N(\gamma_1),\dots,R_N(\gamma_m)$ of an element $\gamma\in Y'$.
Furthermore, trivially, the set~$Y'$ is in bijection with
$$
Y'' =
\{ \nu \in \Z^m :\nu_1\ge\dots\ge\nu_m
\text{ and $R_N(\nu_i+m-i)$, $1 \le i \le m$, are distinct} \}.
$$

The asserted bijection arises from the one described above by
restricting~$X'$ to~$X$ and~$Y''$ to~$Y$. The property that it is
claimed to satisfy follows straightforwardly from the construction.
\end{proof}

We use the above lemma to prove Proposition~\ref{prop:general}.

\begin{proof}[Proof of Proposition~\ref{prop:general}]
We extend the statistic $u$ on $\Par(m,w)$ to $\widetilde{u} : \Par(m) \to \Z$ by
\[
\widetilde{u}(\ka)
 =
\begin{cases}
 \sgn({\sigma}) u(\lambda), &\text{if $\ka \in Y$,} \\
 0, &\text{otherwise},
\end{cases}
\]
where $Y$ is the subset of $\Par(m)$ given in Lemma~\ref{lem:bijection} and
${\sigma} \in {\mathfrak{S}}_m$ and $\lambda \in \Par(m,w)$
are determined by the bijection of that lemma.
It then follows from Theorem~\ref{thm:ssyt2} and
Lemma~\ref{lem:bijection} that 
\begin{equation}
\label{eq:sum}
\sum_{\lambda \in \Par(m,w)}
   u(\lambda)\,s_{\lambda[m,w]'}(\vx)
=\sum_{\lambda \in \Par(m,w)}
\underset{k_1+\dots+k_{m}=0}{\sum_{ k_1,\dots,k_{m}\in\mathbb Z}}
   u(\lambda)
   \det_{1 \le i, j \le m} \left(
    e_{\lambda_i - i + j + N k_i}(\vx)
  \right)
 =
\sum_{\ka \in \Par(m)}
 \widetilde{u}(\ka)
 \det_{1 \le i, j \le m} \left(
  e_{\ka_i - i + j}(\vx)
 \right).
\end{equation}

We shall show that, for all \( \ka \in \Par(m) \),
\begin{equation}
\label{eq:u=Pf}
\widetilde{u}(\ka) = \Pf A^{I_0 \sqcup I_m(\ka)}.
\end{equation}

We write $I_m(\ka) = (i_1, \dots, i_m)$.
If $\ka \not\in Y$, then $i_p \equiv i_q \bmod N$ for some $p<q$.
Hence Condition~(iii) implies $\Pf A^{I_0 \sqcup I_m(\ka)} = 0 = \widetilde{u}(\ka)$.

From now on,
we assume $\ka \in Y$ and proceed by induction on $i_m - i_1$.
If $i_m - i_1 < N$, then $\ka \in \Par(m,w)$, and $\widetilde{u}(\ka) = u(\ka)$.
So, by Condition~(i), we obtain $\Pf A^{I_0 \sqcup I_m(\ka)} = \widetilde{u}(\ka)$.

Suppose that $i_m - i_1 > N$. Let $(j_1, \dots, j_m)$ be the rearrangement of
$(i_1+N, i_2, \dots, i_{m-1}, i_m-N)$ in increasing order and $\tau \in
\mathfrak{S}_m$ the permutation achieving this rearrangement. Let $\iota$ be the
partition such that $I_m(\iota) = (j_1, \dots, j_m)$.
Then
$$
\iota = \tau(\kappa-N\theta+\delta)-\delta,
$$
where $\theta = (1,0,\dots,0,-1)$ and $\delta = (m-1, m-2, \dots, 1, 0)$.
By using Condition~(ii) and the
alternating property of the Pfaffian, we obtain $\Pf A^{I_0 \sqcup I_m(\ka)} =
\sgn(\tau) \Pf A^{I_0 \sqcup I_m(\iota)}$.
Since $j_1 = \min \{ i_1 + N, i_2, i_m
- N \} > i_1$ and $j_m = \max \{ i_1 + N, i_{m-1}, i_m - N \} < i_m$,
we see that
$j_m - j_1 < i_m - i_1$. By applying the induction hypothesis to
$\iota$, we have
$\Pf A^{I_0 \sqcup I_m(\iota)} = \sgn({\rho}) u(\lambda)$,
where $\lambda 
\in \Par(m,w)$ and $\iota = \rho(\lambda+Nk+\delta)-\delta$. On the other
hand, since 
\begin{align*}
\ka&=\tau^{-1}(\iota+\delta)+N\theta-\delta\\
&=\tau^{-1}\left(\rho(\lambda+Nk+\delta)\right)+N\theta-\delta\\
&=\tau^{-1}\rho\left(\lambda+N(k+\rho^{-1}\tau\theta)+\delta\right)-\delta,
\end{align*}
we have \(
\widetilde{u}(\ka) = \sgn(\tau^{-1}\rho)u(\lambda) \), or equivalently
\( \sgn({\rho})u(\lambda) = \sgn(\tau)\widetilde{u}(\ka) \). Combining these
facts, we obtain
\[
  \Pf A^{I_0 \sqcup I_m(\ka)} =
  \sgn(\tau) \Pf A^{I_0 \sqcup I_m(\iota)}
  =\sgn(\tau) \sgn({\rho}) u(\lambda)
  =\sgn(\tau)\sgn(\tau)\widetilde{u}(\ka)
  =\widetilde{u}(\ka) ,
\]
 which completes the proof of~\eqref{eq:u=Pf}.

Now, by \eqref{eq:sum} and~\eqref{eq:u=Pf}, we have
\[
\sum_{\lambda \in \Par(m,w)}
 u(\lambda)\,s_{\lambda[m,w]'}(\vx)
 =
\sum_{\ka \in \Par(m)}
\Pf A^{I_0 \sqcup I_m(\ka)}
 \det_{1 \le i, j \le m} \left(
  e_{\ka_i - i + j}(\vx)
 \right).
\]
For an increasing subsequence $K$ of $I_0 \sqcup (1,2,\dots)$ of length $p + m$, we have
\[
\det  (T_p)_K^{[m]}
 =
\begin{cases}
 \det_{1 \le i, j \le m} \left(
  e_{\ka_i - i + j}(\vx)
 \right),
 &\text{if $K = I_0 \sqcup I_m(\ka)$ for some $\ka \in \Par(m)$,} \\
0, &\text{otherwise.}
\end{cases}
\]
Hence we have
\[
\sum_{\lambda \in \Par(m,w)}
 u(\lambda)\,s_{\lambda[m,w]'}(\vx)
 =
\sum_K
 \Pf A^K \det (T_p)_K^{[m]},
\]
where $K$ runs over all
increasing subsequences of $I_0 \sqcup (1,2,\dots)$ of length $p + m$.
The desired result follows from the minor summation formula
in Theorem~\ref{thm:IsWa}.
\end{proof}

\subsection{Proof of Theorem \ref{thm:affine_BK_intro}}
\label{sec:LC}
We apply Proposition~\ref{prop:general} to prove
Theorem~\ref{thm:affine_BK_intro}. For that purpose,
we introduce a skew-symmetric matrix $A$ that satisfies the conditions
in Proposition~\ref{prop:general} with the statistic~$u$ being equal to
the statistic $b$ defined by $b(\lambda) = 1$ for all $\lambda \in \Par(m,w)$.

\begin{lem}
\label{lem:matrix_B}
Let $m$ and $w$ be positive integers and put $N = m+w$.
\begin{enumerate}
\item[(1)]
  Suppose that $m = 2h+1$ is odd.
  Let ${\beta}_\ell $, $\ell  \in \Z$, be the sequence defined by
  \[
    {\beta}_\ell 
    =
    \begin{cases}
      2k+1, &\text{if $\ell =kN+r$, where \( k,r\in\Z \) and $0 < r < N$,} \\
      2k  , &\text{if $\ell =kN$, where \( k\in\Z \),}
    \end{cases}
  \]
  and ${B} = ({B}_{r,s})$ the skew-symmetric matrix with rows/columns indexed by nonnegative integers
  whose $(r,s)$-entry, $0 \le r < s$, is given by
  \[
    {B}_{r,s}
    =
    \begin{cases}
      1, &\text{if $r=0$,} \\
      {\beta}_{s-r}, &\text{if $r \ge 1$.}
    \end{cases}
  \]
  Then ${B}$ satisfies the three conditions in Proposition~\ref{prop:general}
  for $u=b$ and $p=1$.
\item[(2)]
Suppose that $m = 2h$ is even.
Let $\bar{\beta}_\ell $, $\ell  \in \Z$, be the sequence defined by
\[
\bar{\beta}_\ell 
 =
\begin{cases}
 (-1)^k, &\text{if $\ell =kN+r$, where \( k,r\in\Z \) and $0 < r < N$,} \\
 0  ,    &\text{if $\ell =kN$, where \( k\in\Z \),}
\end{cases}
\]
and $\bar{B} = (\bar{B}_{r,s})$ the skew-symmetric matrix with rows/columns indexed by positive integers
whose $(r,s)$-entry is given by $\bar{B}_{r,s} = \bar{\beta}_{s-r}$.
Then $\bar{B}$ satisfies the three conditions in Proposition~\ref{prop:general} for $u=b$
and $p=0$.
\end{enumerate}
\end{lem}

\begin{proof}
Let $I = (i_1, \dots, i_m)$ be an increasing sequence of positive integers.

\medskip
(1) For Condition~(i) of Proposition~\ref{prop:general},
it suffices to
show that, if $i_m - i_1 < N$, then $\Pf B^{(0)\sqcup I} = 1$. Since $i_m - i_1 <
N$ implies that $0 < i_{\ell} - i_k < N$ for any~$k$ and~$\ell$ with $1 \le k
< \ell \le m$, 
we see that ${B}^{(0)\sqcup I}$ is the skew-symmetric matrix with all $1$'s
above the diagonal. So we have $\Pf {B}^{(0)\sqcup I} = 1$. For the second
condition, Condition~(ii), suppose that \( i_m-i_1>N \) and let $J =
(i_1+N, i_2, \dots, 
i_{m-1}, i_m-N)$. By noting ${\beta}_{\ell+N} = {\beta}_\ell+2$,
for $1\le r<s$ we have
\[
{B}_{r+N,s} = {B}_{r,s}-2,
\quad
{B}_{r+N,s-N} = {B}_{r,s}-4,
\quad
{B}_{r,s-N} = {B}_{r,s}-2.
\]
Hence the matrix ${B}^{(0)\sqcup J}$ is obtained from ${B}^{(0)\sqcup I}$ by
adding the $0$-th row/column multiplied by $-2$ to the first row/column and
adding the $0$-th row/column multiplied by $2$ to the $n$-th row/column.
These operations do not change the value of the Pfaffian. Hence we
have \( \Pf{B}^{(0)\sqcup I}=\Pf{B}^{(0)\sqcup J} \). For the third
condition, Condition~(iii),
suppose $i_\ell - i_k = cN$. Since ${B}_{r+cN,s} = {B}_{r,s} - 2c$, we see
that the $0$-th, $k$-th and $\ell$-th rows of ${B}^{(0)\sqcup I}$ are linearly
dependent, so we have $\Pf {B}^{(0)\sqcup I} = 0$.

\medskip\noindent
(2) Condition~(i) of Proposition~\ref{prop:general}
can be proved similarly as in~(1). For Condition~(ii), suppose
$i_m - i_1 > N$ and $J = (i_1+N, i_2, \dots, i_{m-1}, i_m-N)$. Since
$\bar{\beta}_{\ell+N} = - \bar{\beta}_\ell$, we have
\[
  \bar{B}_{r+N,s} = -\bar{B}_{r,s},
  \quad
  \bar{B}_{r+N,s-N} = \bar{B}_{r,s},
  \quad
  \bar{B}_{r,s-N} = - \bar{B}_{r,s}.
\]
Hence, by multiplying the first row/column of $\bar{B}^J$ by $-1$ and the $m$-th
row/column by $-1$, we obtain $\Pf \bar{B}^J = (-1)^2 \Pf \bar{B}^I = \Pf
\bar{B}^I$. For the last condition, Condition~(iii), suppose $i_\ell - i_k = cN$. Since
$\bar{B}_{r+cN,s} = (-1)^c \bar{B}_{r,s}$, we see that the $k$-th and $\ell$-th
rows of $\bar{B}^I$ are proportional to each other, hence we have $\Pf \bar{B}^I
= 0$.
\end{proof}

We are now ready for the proof of Theorem~\ref{thm:affine_BK_intro}.
We start with the second identity in the theorem.

\begin{proof}[Proof of \eqref{eq:aGBK2s}]
Here $m=2h$ is even.
By applying Proposition~\ref{prop:general}
to the statistic $b(\lambda) = 1$, for $\lambda \in \Par(m,w)$,
and the skew-symmetric matrix $\bar{B}$
given in Lemma~\ref{lem:matrix_B}(2),
we get
\[
\sum_{\lambda \in \Par(2h,w)} s_{\lambda[2h,w]'}(\vx)
 =
\Pf \left( T_0 \bar{B} T_0^t \right),
\]
where $T_0 = \left( e_{j-i}(\vx) \right)_{1 \le i \le m, j \ge 1}$ as given in~\eqref{eq:T_p} with $p=0$.
The $(i,j)$-entry of $T_0 \bar{B} T_0^t$ equals
\[
\sum_{r,s \ge 1} e_{r-i}(\vx) \bar{\beta}_{s-r} e_{s-j}(\vx)
 =
\sum_{\ell  \in \Z} \bar{\beta}_\ell  f_{\ell -j+i}(\vx).
\]
If we put $z_r = \sum_{\ell  \in \Z} \bar{\beta}_\ell  f_{\ell -r}(\vx)$, then we have
\[
\sum_{\lambda \in \Par(2h,w)} s_{\lambda[2h,w]'}(\vx)
 =
\Pf_{1 \le i < j \le m} \left( z_{j-i} \right)
 =
\det_{1 \le i, j \le h} \left(
 \sum_{s=0}^{2\min(i,j)-2} (-1)^s \left( z_{i+j-1-s} - z_{i+j-2-s} \right)
\right),
\]
where the last equality follows from~\eqref{eq:Gordon_var1}.
Since we have
\[
\bar{\beta}_\ell  - \bar{\beta}_{\ell -1}
 =
\begin{cases}
 (-1)^k, &\text{if $\ell =Nk$ or $Nk+1$,} \\
 0, &\text{otherwise,}
\end{cases}
\]
we obtain
\begin{multline}\label{eq:z_r}
  z_r - z_{r-1} = \sum_{\ell \in\Z} (\bar{\beta}_\ell  - \bar{\beta}_{\ell -1}) f_{\ell -r}(\vx)
  =\sum_{k\in\Z}(-1)^k \big(f_{kN-r}(\vx)+f_{kN+1-r}(\vx)\big)\\
  = \overline{F}_{-r,N}(\vx) + \overline{F}_{1-r,N}(\vx)
  = \overline{F}_{r,N}(\vx) + \overline{F}_{r-1,N}(\vx).
\end{multline}
Thus,
\[
\sum_{s=0}^{2\min(i,j)-2} (-1)^s \left( z_{i+j-1-s} - z_{i+j-2-s} \right)
 =
\overline{F}_{i+j-1,N}(\vx) + \overline{F}_{|i-j|,N}(\vx)
 =
\overline{F}_{i+j-1,N}(\vx) + \overline{F}_{j-i,N}(\vx).
\]
This completes the proof of~\eqref{eq:aGBK2s}.
\end{proof}

Next we prove the first identity in Theorem~\ref{thm:affine_BK_intro}.

\begin{proof}[Proof of \eqref{eq:aGBK1s}]
Here $m=2h+1$  is odd. By applying
Proposition~\ref{prop:general} to the skew-symmetric matrix ${B}$ given
in Lemma~\ref{lem:matrix_B}(1), we get
\[
  \sum_{\lambda \in \Par(2h+1,w)} s_{\lambda[2h+1,w]'}(\vx)
 =
\Pf \left( T_1 {B} T_1^t \right),
\]
where \( T_1 \) is given by~\eqref{eq:T_p} with $p=1$.
If we put $e(\vx) = \sum_{i \ge 0} e_i(\vx)$, as before in~\eqref{eq:e(x)},
and $w_r = \sum_{\ell  \in \Z} {\beta}_\ell  f_{\ell -r}(\vx)$,
then the $(i,j)$-entry $Q_{i,j}$, $0 \le i < j \le m$, of $Q:=T_1 {B} T_{1}^t$ is given by
\[
Q_{i,j}
 =
\begin{cases}
 e(\vx), &\text{if $i=0$,} \\
 w_{j-i}, &\text{if $i \ge 1$.}
\end{cases}
\]
By subtracting the $(i-1)$-st row/column from the $i$-th row/column for $i=m, m-1, \dots, 2$,
and then expanding the resulting Pfaffian along the $0$-th row/column, we obtain
\[
\Pf Q=
\Pf \left( T_1 {B} T_1^t \right)
 =
e(\vx) \cdot
\Pf_{2 \le i < j \le n}
 \left( w_{j-i} - w_{j-i-1} - w_{j-i+1} + w_{j-i} \right).
\]
Since we have
\[
{\beta}_\ell  - {\beta}_{\ell -1}
=
\begin{cases}
 1, &\text{if $R_N(\ell ) = 0$ or $1$,} \\
 0, &\text{otherwise,}
\end{cases}
\]
by a similar computation as in~\eqref{eq:z_r} we get $w_r - w_{r-1} =
F_{r,N}(\vx) 
+ F_{r-1,N}(\vx)$ and $w_r - w_{r-1} - w_{r+1} + w_r = F_{r-1,N}(\vx)
- F_{r+1,N}(\vx)$.
Now, by using~\eqref{eq:Gordon} with $z_r = F_{r-1,N}(\vx) - F_{r+1,N}(\vx)$, we
complete the proof of~\eqref{eq:aGBK1s}.
\end{proof}

\subsection{Proof of Theorem \ref{thm:C}}

In this subsection, we prove the affine bounded Littlewood
identities in Theorem~\ref{thm:C}.
The idea of the proof is the same as that of the proof of
Theorem~\ref{thm:affine_BK_intro},
so we shall only provide a sketch.

The following lemma gives a skew-symmetric matrix that satisfies the
conditions 
in Proposition~\ref{prop:general} for $u=c^\pm_{2h,w}$, given
by~\eqref{eq:wt_C}. 

\begin{lem}
\label{lem:matrix_C}
Let $m = 2h$ be a positive even integer and $w$ a positive integer, and put $N = 2h+w$.
We define the sequences $\gamma^+_\ell $ and $\gamma^-_\ell $, $\ell  \in \Z$, by
\begin{align*}
  \gamma^+_\ell  &=
  \begin{cases}
    (-1)^k, &\text{if $\ell =kN+1$ or $kN + (N-1)$, where \( k\in\Z \),} \\
    0, &\text{otherwise,}
  \end{cases}\\
  \gamma^-_\ell  &=
        \begin{cases}
          1, &\text{if $\ell =kN+1$, where \( k\in\Z \),} \\
          -1, &\text{if $\ell =kN+(N-1)$, where \( k\in\Z \),} \\
          0, &\text{otherwise.}
        \end{cases}
\end{align*}
Let $C^+ = (C^+_{r,s})$ (respectively $C^- = (C^-_{r,s})$) be the
skew-symmetric matrix 
with rows/columns indexed by positive integers
whose $(r,s)$-entry, $r<s$, is given by $C^+_{r,s} = \gamma^+_{s-r}$
(respectively by $C^-_{r,s} = \gamma^-_{s-r}$).
Then the matrix $C^+$ (respectively $C^-$) satisfies the conditions in
Proposition~\ref{prop:general} 
for $u=c^+_{m,w}$ (respectively $u=c^-_{m,w}$) and \( p=0 \).
\end{lem}

\begin{proof}
  We only prove the first part,
  namely $\Pf (C^\pm)^{I_m(\lambda)} = c^\pm_{m,w}(\lambda)$ for
  $\lambda \in \Par(m,w)$. 
(Since $\gamma^+_{\ell +N} = - \gamma^+_\ell $ and $\gamma^-_{\ell +N} = \gamma^-_\ell $, the other parts can be proved in exactly the same way
as in the proof of Lemma~\ref{lem:matrix_B}(1).)

Let $M^{\pm} = (m^{\pm}_{i,j})_{1 \le i, j \le N}$ be
the skew-symmetric matrix with entries given by
\[
m^{\pm}_{i,j}
 =
\begin{cases}
 1, &\text{if $j=i+1$,} \\
 \pm 1, &\text{if $i=1$ and $j=N$,} \\
 0, &\text{otherwise.}
\end{cases}
\]
Then we have $(C^\pm)^{I_m(\lambda)} = (M^\pm)^{I_m(\lambda^0)}$,
where $\lambda^0 = (\lambda_1 - \lambda_m, \lambda_2 - \lambda_m, \dots, \lambda_{m-1} - \lambda_m, 0)$.
By using~\cite[Lem.~3.4(3)]{Okada1998}, we obtain $\Pf (C^\pm)^{I_m(\lambda)} = c^\pm_{m,w}(\lambda)$.
\end{proof}

\begin{proof}[Proof of Theorem~\ref{thm:C}]
  By Proposition~\ref{prop:general} and Lemma~\ref{lem:matrix_C}
  with \( m=2h \) and \( N=2h+w \), we obtain
\[
\sum_{\lambda \in \Par(2h,w)}
 c^\pm_{m,w}(\lambda)\,s_{\lambda[2h,w]'}(\vx)
 =
\Pf \left( T_0 C^\pm T_0^t \right)
 =
\Pf_{1 \le i < j \le 2h} \left( v^\pm_{j-i} \right),
\]
where
\( v^+_r = \overline{F}_{r-1,N}(\vx) - \overline{F}_{r+1,N}(\vx) \)
and \( v^-_r = F_{r-1,N}(\vx) - F_{r+1,N}(\vx) \).
Then the proof is completed by applying~\eqref{eq:Gordon}.
\end{proof}

\subsection{Proof of Theorem \ref{thm:D}}
\label{sec:LD}
Again, we only provide a sketch here since the idea of the proof is
the same as that of the proof of Theorem~\ref{thm:affine_BK_intro}.
Recall that the statistics $d^+_{m,w}$ and $d^-_{m,w}$ on $\Par(m,w)$ (see~\eqref{eq:wt_D}) are given by
\[
d^\pm_{m,w}(\lambda)
 =
\begin{cases}
 1, &\text{if $\lambda$ is even,} \\
 \pm 1, &\text{if $\lambda$ is odd,} \\
 0, &\text{otherwise,}
\end{cases}
\]
where, as earlier in Section~\ref{sec:symm-funct-ident}, a partition~$\la$
is called \emph{even} (respectively \emph{odd}) if all its parts
$\lambda_1, \dots, \lambda_m$ are even (respectively odd).

\begin{lem}
\label{lem:matrix_D}
Let $m$ be a positive integer and $w$ a positive even integer, and put $N = m+w$.
\begin{enumerate}
\item[(1)]
Suppose that $m=2h$ is even and define a sequence $\delta^+_\ell$, $\ell \in \Z$, by
\[
\delta^+_\ell 
 =
\begin{cases}
 (-1)^k, &\text{if $\ell $ is odd and $\ell =kN+r$, where \( k,r\in\Z \) and \( 0<r<N \),} \\
 0, &\text{if $\ell $ is even.}
\end{cases}
\]
Let $D^+ = ( D^+_{r,s})$ be the skew-symmetric matrix with rows/columns indexed by $\{ 1, 2, \dots \}$
whose $(r,s)$-entry is given by $D^+_{r,s} = \delta^+_{s-r}$.
Then $D^+$ satisfies the conditions of Proposition~\ref{prop:general}
for $u=2^{h-1} \cdot d^+_{2h,w}$ and $p=0$.
\item[(2)]
Suppose that $m=2h$ is even and define a sequence $\delta^-_\ell$, $\ell \in \Z$, by
\[
\delta^-_\ell 
 =
\begin{cases}
  2k+1, &\text{if $\ell $ is odd and $\ell =kN+r$, where \( k,r\in\Z \) and \( 0<r<N \),} \\
 0, &\text{if $\ell $ is even.}
\end{cases}
\]
Let $D^- =  (D^-_{r,s}) $ be the skew-symmetric matrix with rows/columns indexed
by $\{ 0, 0', 1, 2, \dots \}$ whose entries are given by
\[
D^-_{0,0'} = 0,
\quad
D^-_{0,s} = 1,
\quad
D^-_{0',s} = (-1)^{s-1},
\quad
D^-_{r,s} = \delta^-_{s-r},
\]
where $r, s \ge 1$.
Then $D^-$ satisfies the conditions of Proposition~\ref{prop:general}
for $u=2^h \cdot d^-_{2h,w}$ and $p=2$.
\item[(3)]
Suppose that $m=2h+1$ is odd and define a sequence $\delta^+_\ell$, $\ell \in \Z$, by
\[
\delta^+_\ell 
 =
\begin{cases}
 k N' + \lceil r/2 \rceil,
 &\text{if $\ell  = kN+r$, where \( k,r\in \Z \) and $0 < r < N$,}
\\
k N', &\text{if $\ell =kN$, where \( k\in \Z \),}
\end{cases}
\]
where $N' = (N+1)/2$ and $\lceil x \rceil$ is the smallest integer greater than or equal to $x$.
Let $D^+ = ( D^+_{r,s} )$ be the skew-symmetric matrix with 
rows/columns 
indexed by $\{ 0, 1, 2, \dots \}$ whose entries are given by
\[
D^+_{0,s} = 1,
\quad
D^+_{r,s} = \delta^+_{s-r},
\]
where $r$, $s \ge 1$.
Then $D^+$ satisfies the conditions of Proposition~\ref{prop:general}
for $u=d^+_{2h+1,w}$ and $p=1$.
\item[(4)]
Suppose that $m=2h+1$ is odd and define a sequence $\delta^-_\ell$, $\ell \in \Z$, by
\[
\delta^-_\ell 
 =
(-1)^{\ell +1} \delta^+_\ell 
 =
\begin{cases}
 (-1)^{\ell +1} \left( k N' + \lceil r/2 \rceil \right),
 &\text{if $\ell  = kN+r$, where \( k,r\in \Z \) and $0 < r < N$,}
\\
(-1)^{\ell +1} k N', &\text{if $\ell =kN$, where \( k\in \Z \).}
\end{cases}
\]
Let $D^- = ( D^-_{r,s} )$ be the skew-symmetric matrix with 
rows/columns
indexed by $\{ 0, 1, 2, \dots \}$ whose entries are given by
\[
D^-_{0,s} = (-1)^{s-1},
\quad
D^-_{r,s} = \delta^-_{s-r},
\]
where $r$, $s \ge 1$.
Then $D^-$ satisfies the conditions of Proposition~\ref{prop:general}
for $u=d^-_{2h+1,w}$ and $p=1$.
\end{enumerate}
\end{lem}

For example, if $N = 7$, then the sequences $\delta^+$ and $\delta^-$
defined in~(3) and~(4) are the following:
\[
\begin{array}{c|ccccccccccccccccccccccc}
\ell\kern-2pt  & \dots & -9 & -8 & -7 & -6 & -5 & -4 & -3 & -2 & -1 & 0 & 1 & 2 & 3 & 4 & 5 & 6 & 7 & 8 & 9 & 10 & \dots
\\
\hline
\delta^+_\ell \kern-2pt &
    \dots & -5 & -5 & -4 & -3 & -3 & -2 & -2 & -1 & -1 & 0 & 1 & 1 & 2 & 2 & 3 & 3 & 4 & 5 & 5 & 6 & \dots
\\
\delta^-_\ell \kern-.2pt &
    \dots & -5 & \hphantom{-}5 & -4 & \hphantom{-}3 & -3 & \hphantom{-}2 & -2 & \hphantom{-}1 & -1 & 0 & 1 & -1 & 2 & -2 & 3 & -3 & 4 & -5 & 5 & -6 & \dots
\end{array}
\]

\begin{proof}
(1)
We only show that $\Pf (D^+)^{I_m(\lambda)} = d^+_{m,w}(\lambda)$ if $\lambda \in \Par(m,w)$.
(Since $\delta^+_{\ell +N} = - \delta^+_\ell $, the proof of the other
parts is the same as in the proof of Lemma~\ref{lem:matrix_B}(1).)
If $\lambda$ is neither even nor odd, then there is an index $k$ such that $i_k \equiv i_{k+1} \bmod 2$,
where $I_m(\lambda) = (i_1, \dots, i_m)$. In that case,
we see that the $k$-th row of $D^+_{I_m(\lambda)}$ is the same as the $(k+1)$-st row,
hence we have $\Pf (D^+)^{I_m(\lambda)} = 0$.
If $\lambda$ is even or odd, then the $(i,j)$-entry, $i<j$, of $(D^+)^{I_m(\lambda)}$
is equal to $1$ if $j-i \equiv 1 \bmod 2$ and $0$ otherwise.
By performing elementary row/column operations and using the expansion
of the Pfaffian,
we obtain $\Pf (D^+)^{I_m(\lambda)} = 2^{m/2-1}$.

\medskip\noindent
(2) Condition~(i) of Proposition~\ref{prop:general}
for $\Pf (D^-)^{I_m(\lambda)}$ with $\lambda \in \Par(m,w)$
is checked by an argument similar to that of~(1).
Conditions~(ii) and~(iii) can be verified by using the relation
$\delta^-_{\ell +N} = \delta^-_\ell  + 1 - (-1)^\ell $.

\medskip\noindent
(3)
In order to prove $\Pf (D^+)^{I_m(\lambda)} = d^+_{m,w}(\lambda)$ for $\lambda \in \Par(m,w)$,
we note the relation
\[
\delta^+_\alpha - \delta^+_\ell 
 =
\lceil \alpha/2 \rceil - \lceil \ell /2 \rceil.
\]
If $\lambda$ is neither even nor odd, then there is an index $k$ such that $i_k \equiv i_{k+1} \bmod 2$,
where $I_m(\lambda) = (i_1, \dots, i_m)$.
We can use the above relation to see that the $0$-th, $k$-th and $(k+1)$-st rows are linearly dependent,
hence we have $\Pf (D^+)^{I_m(\lambda)} = 0$.
If $\lambda$ is even or odd, then, by using the above relation, we can
perform row/column operations 
to transform the skew-symmetric matrix $(D^+)^{I_m(\lambda)}$ into the
skew-symmetric matrix 
$M = ( M_{i,j} )_{0 \le i, j \le m}$ with entries $M_{i,j}$, $i<j$, given by
\[
M_{i,j}
 =
\begin{cases}
 1, &\text{if $i=0$ and $j=1$,} \\
 1, &\text{if $i=1$ and $j$ is even,} \\
 (-1)^{j-i-1}, &\text{if $i \ge 2$,} \\
 0, &\text{otherwise.}
\end{cases}
\]
Hence, by expanding the Pfaffian along the $0$-th row/column, we see that
\[
\Pf (D^+)^{I_m(\lambda)} = \Pf M = \Pf_{2 \le i<j \le m} \left( (-1)^{j-i-1} \right)
 =
1,
\]
where we used \cite[Lem.~7]{Ishikawa1995} in the last equality.
Therefore we have $\Pf (D^+)^{I_m(\lambda)} = d^+_{2h+1,w}(\lambda)$.

Using the relation $\delta^+_{\ell +N} = \delta^+_\ell  + N'$,
we can prove that $D^+$ satisfies Conditions~(ii) and~(iii).

\medskip\noindent
(4)
The proof is the same as (3) up to a sign, so we omit it.
\end{proof}

We are now in the position to prove the identities in~Theorem~\ref{thm:D}.

\begin{proof}[Proof of \eqref{eq:D1a}]
By applying Proposition~\ref{prop:general} to the skew-symmetric matrix $D^+$ given in Lemma~\ref{lem:matrix_D}~(1), we get
\[
2^{h-1} \sum_{\lambda \in \Par(2h,w)}
 d^+_{2h,w}(\lambda)\,s_{\lambda[2h,w]'}(\vx)
 =
\Pf \left( T_0 D^+ T_0^t \right)
 =
\Pf_{1 \le i < j \le 2h} \left( v^+_{j-i} \right),
\]
where $v^+_r = \sum_{k \in \Z} \delta^+_k f_{-k+r}(\vx)$.
Since $\delta^+_\ell  -
\delta^+_{\ell -2} = 2 (-1)^k$ if $\ell  = kN+1$ and $0$ otherwise, we obtain
\( v^+_1 = \overline{F}_{0,N}(\vx) \) and  \( v^+_r - v^+_{r-2} = 2 \overline{F}_{r-1,N}(\vx) \).
Hence, by using~\eqref{eq:Gordon_var3}, we have
\begin{align*}
\Pf_{1 \le i < j \le 2h} \left( v^+_{j-i} \right)
 &=
\det \begin{pmatrix}
 \overline{F}_{0,N}(\vx) & \left( 2 \overline{F}_{j-1,N}(\vx) \right)_{2 \le j \le h} \\
 \left( 2 \overline{F}_{i-1,N}(\vx) \right)_{2 \le i \le h}
 & \left( 2 \overline{F}_{i+j-2,N}(\vx) + 2 F_{|j-i|,N}(\vx) \right)_{2 \le i, j \le h}
\end{pmatrix}
\\
 &=
2^{h-2}
\det_{1 \le i, j \le h} \left(
 \overline{F}_{-i+j,N}(\vx) + \overline{F}_{i+j-2,N}(\vx)
\right).
\qedhere
\end{align*}
\end{proof}

\begin{proof}[Proof of \eqref{eq:D2a}]
By applying Proposition~\ref{prop:general} to the skew-symmetric matrix $D^-$
given in Lemma~\ref{lem:matrix_D}(2), we get
\[
2^{h-1} \sum_{\lambda \in \Par(2h,w)}
 d^-_{2h,w}(\lambda)\,s_{\lambda[2h,w]'}(\vx)
 =
\Pf \left( T_2 D^- T_2^t \right).
\]
Here the entries $Q_{i,j}$, $i,j \in \{ 0, 0', 1, \dots, m \}$ and $i<j$,
of the matrix \(Q:= T_2 D^- T_2^t \) are given by
\[
Q_{0,0'} = 0,
\quad
Q_{0,j} = e(\vx),
\quad
Q_{0',j} = (-1)^{j-1} \overline{e}(\vx),
\quad
Q_{i,j} = \sum_{\ell  \in \Z} \delta^-_\ell  f_{-\ell +j-i}(\vx),
\]
where $i,j \ge 1$.
In $\Pf Q$,
we subtract the $(i-2)$-nd row/column from the $i$-th row/column for $i=m, m-1, \dots, 3$,
and then add the $0'$-th row/column to the $0$-th row/column.
Then, by expanding the resulting Pfaffian along the $0$-th row/column, we see that
\[
\Pf \left( T_2 D^- T_2^t \right)
 =
2 e(\vx) \overline{e}(\vx)
\Pf_{1 \le i, j \le 2h-2} \left(
 v^-_{j-i} - v^-_{j-i-2} - v^-_{j-i+2} + v^-_{j-i}
\right),
\]
where $v^-_r = \sum_{\ell  \in \Z} \delta^-_\ell  f_{-\ell +j-i}(\vx)$.
Since $\delta^-_\ell  - \delta^-_{\ell -2} = 2$ if $\ell =kN+1$ and $0$ otherwise, we have $v^-_r - v^-_{r-2} = 2 F_{r-1,N}(\vx)$.
Now, by using~\eqref{eq:Gordon}, we obtain the desired identity.
\end{proof}

\begin{proof}[Proof of \eqref{eq:D3a}]
By using Proposition~\ref{prop:general} and a computation similar to the proof of~\eqref{eq:aGBK2s},
we have
\[
\sum_{\lambda \in \Par(2h+1,w)}
 d^+_{2h+1,w}(\lambda)\,s_{\lambda[2h+1,w]'}(\vx)
 =
e(\vx) \cdot \Pf_{2 \le i, j \le 2h+1} \left(
 v^+_{j-i} - v^+_{j-i-1} - v^+_{j-i+1} + v^+_{j-i}
\right).
\]
Since we have
\[
\delta^+_\ell  - \delta^+_{\ell -1} - \delta^+_{\ell +1} + \delta^+_\ell 
 =
\begin{cases}
  1, &\text{if $\ell =Nk+r$, where \( k,r\in \Z \) such that $0<r<N$ and $r$ is odd,} \\
  -1, &\text{if $\ell =Nk+r$, where \( k,r\in \Z \) such that $0<r<N$ and $r$ is even,} \\
 0, &\text{otherwise,}
\end{cases}
\]
we see that
\[
v^+_r - v^+_{r-1} - v^+_{r+1} + v^+_r
 =
\sum_{t=1}^{N-1} (-1)^t F_{r+t,N}(\vx),
\]
where we note that $N=2h+1+w$ is odd since $w$ is even by assumption,
and $F_{r+N,N}(\vx) = F_{r,N}(\vx)$.
Then we can complete the proof by applying~\eqref{eq:Gordon_var2}.
\end{proof}

\begin{proof}[Proof of \eqref{eq:D4a}]
The proof is similar to that of the proof of~\eqref{eq:D3a}.
In this case, we use the relation
\[
\delta^+_\ell  + \delta^+_{\ell -1} + \delta^+_{\ell +1} + \delta^+_\ell 
 =
\begin{cases}
  (-1)^k, &\text{if $\ell =Nk+r$, where \( k,r\in \Z \) and $0 < r < N$,} \\
 0, &\text{otherwise,}
\end{cases}
\]
and the formula \eqref{eq:Gordon_var1}.
\end{proof}

\section{Up-down tableaux}
\label{sec:cylindric_SSYT}

In this section, we provide combinatorial interpretations of the
right-hand sides of the affine bounded Littlewood identities~\eqref{eq:aGBK1s}
and~\eqref{eq:aGBK2s} (which are the same as~\eqref{eq:ABL1}
and~\eqref{eq:ABL2}) using up-down tableaux.
The latter are sequences of partitions satisfying certain conditions. 
Note that there is a combinatorial meaning of the left-hand sides of the affine
bounded Littlewood identities~\eqref{eq:aGBK1s} and~\eqref{eq:aGBK2s} in
terms of cylindric semistandard Young tableaux,
respectively in terms of cylindric row-strict Young tableaux, see
Definition~\ref{defn:cSchur} and Proposition~\ref{prop:cyl_sch}.

To be specific, in Theorem~\ref{thm:OT-L1} we provide a combinatorial
interpretation of the determinant on the right-hand side
of~\eqref{eq:aGBK1s}
for the case where $w$ is odd. It is then 
simple to derive a combinatorial interpretation of the right-hand side
of~\eqref{eq:aGBK1s}
itself, see Corollary~\ref{cor:OT-L1}.
In Theorem~\ref{thm:OT-L1'} we give a combinatorial interpretation of the
determinant on the right-hand side of~\eqref{eq:aGBK1s} when $w$ is
even, as well as for the determinant
on the right-hand side of~\eqref{eq:aGBK2s} for both odd and even~$w$.
The resulting combinatorial interpretations of the full right-hand
sides of~\eqref{eq:aGBK1s} with
$w$~even,
and of~\eqref{eq:aGBK2s}
are the subject of Corollaries~\ref{cor:OT-L2}--\ref{cor:OT-L4}.

\medskip
We start by defining the above-mentioned up-down tableaux precisely.

\begin{defn}
  An \emph{\( (h,w) \)-up-down tableau} is a sequence \(
  (\la^0,\la^1,\dots,\la^{2n}) \) of partitions
  satisfying the following properties:
  \begin{enumerate}
  \item [(i)]  \(
    \emptyset=\la^0\subseteq\la^1\supseteq\la^2\subseteq
    \la^3\supseteq\la^4\subseteq\dots\subseteq\la^{2n-1}\supseteq
    \la^{2n}=\emptyset
    \);
  \item [(ii)]each pair $(\la^{i-1},\la^i)$ differs by a vertical strip
    {\rm(}that is, by a collection of cells which contains at most one
    cell in each row{\rm)}, $i=1,2,\dots,2n$;
  \item [(iii)]each $\la^{i}$ has at most $h$~rows,
    $i=1,2,\dots,2n$;
  \item [(iv)]each $\la^{2i}$ has at most $w$~columns,
    $i=1,2,\dots,n$.
  \end{enumerate}
\end{defn}

Let \( \UD_n(h,w) \) denote the set of \emph{\( (h,w) \)-up-down tableaux} \(
(\la^0,\la^1,\dots,\la^{2n}) \). For \( T=(\la^0,\la^1,\dots,\la^{2n})\in
\UD_n(h,w) \), we define its weight by
\[
  \omega(T)= \prod_{i=1}^n x_i^{-\vert\la^{2i-2}\vert+2\vert\la^{2i-1}\vert-\vert\la^{2i}\vert}.
\]
Note that \( -\vert\la^{2i-2}\vert+2\vert\la^{2i-1}\vert-\vert\la^{2i}\vert \)
is the sum of the differences in sizes of \( (\lambda^{2i-2}, \lambda^{2i-1}) \)
and of \( (\lambda^{2i-1}, \lambda^{2i}) \).

The following theorem provides a combinatorial interpretation of
the determinant on the right-hand side of~\eqref{eq:aGBK1s} in terms
of up-down tableaux for the case where $w$ is odd.

\begin{thm} \label{thm:OT-L1}
For positive integers $h$, $w$, and \( n \), we have
\begin{equation}
  \label{eq:OT-L1} \det_{1\le i,j\le h}\left(
  F_{-i+j,2h+2w+2}(x_1,\dots,x_n)-F_{i+j,2h+2w+2}(x_1,\dots,x_n) \right)
  =\sum_{T\in\UD_n(h,w)} \omega(T).
\end{equation}
\end{thm}

We will prove this theorem with the help of the idea of
{\it nonintersecting lattice paths} (see \cite{LindAA, GeViAA}), with
the particular setting being the one from~\cite{Krattenthaler1996}.
We need some preparations first, however.

\medskip
Following~\cite[Sec.~2]{Krattenthaler1996}, we consider lattice
paths in the plane integer lattice with steps from the set
\begin{equation}
  \label{eq:steps}
  S=\{(i,j)\to(i,j+1),\ (i,2j-2)\to(i+1,2j-1),\ (i,2j-1)\to(i-1,2j): i,j\in\Z  \}.
\end{equation}
In words, the set of steps consists of vertical steps \( s_v \) north, forward
diagonal steps \( s_f \) northeast from even height to odd height, and backward
diagonal steps \( s_b \) northwest from odd height to even
height. Throughout, when we say `lattice path' or simply `path' we
always mean a lattice path with steps from the set $S$. Several such lattice
paths are displayed in Figure~\ref{fig:7}.

\begin{figure}

\small{
\begin{tikzpicture}[scale=0.6]

\draw[->] (-1,0) -- (8,0);
\draw[->] (0,-1) -- (0,11);
\draw[dotted] (7,-1) -- (7,11);
\draw[dotted] (1,-1) -- (1,11);

\foreach \i in {0,...,7}
\foreach \j in {0,...,10}
\filldraw[fill=gray, color=gray] (\i,\j) circle (2.5pt);

\foreach \i in {1,2,3,4}
\filldraw (\i,0) circle (3.5pt);

\foreach \i in {1,2,4,5}
\filldraw (\i,1) circle (2.5pt);

\foreach \i in {1,2,3,5}
\filldraw (\i,2) circle (2.5pt);

\foreach \i in {2,3,4,6}
\filldraw (\i,3) circle (2.5pt);

\foreach \i in {1,3,4,6}
\filldraw (\i,4) circle (2.5pt);

\foreach \i in {1,4,5,7}
\filldraw (\i,5) circle (2.5pt);

\foreach \i in {1,3,5,6}
\filldraw (\i,6) circle (2.5pt);

\foreach \i in {2,3,5,6}
\filldraw (\i,7) circle (2.5pt);

\foreach \i in {2,3,4,5}
\filldraw (\i,8) circle (2.5pt);

\foreach \i in {2,3,4,5}
\filldraw (\i,9) circle (2.5pt);

\foreach \i in {1,2,3,4}
\filldraw (\i,10) circle (3.5pt);

\node[right] at (1.9,8.5) {\( P_1\)};
\node[right] at (2.9,8.5) {\( P_2\)};
\node[right] at (3.9,8.5) {\( P_3\)};
\node[right] at (4.9,8.5) {\( P_4\)};

\foreach \i/\j in {3.65/0.6, 4.65/0.6, 3.65/1.4}
\node[left] at (\i,\j) {\(x_1\)};

\foreach \i/\j in {1.65/2.6, 2.65/2.6, 3.65/2.6, 5.65/2.6,1.65/3.4}
\node[left] at (\i,\j) {\(x_2\)};

\foreach \i/\j in {3.65/4.6, 4.65/4.6, 6.65/4.6,3.65/5.4, 6.65/5.4}
\node[left] at (\i,\j) {\(x_3\)};

\foreach \i/\j in {1.65/6.6, 4.65/7.4, 5.65/7.4}
\node[left] at (\i,\j) {\(x_4\)};

\foreach \i/\j in {1.65/9.4, 2.65/9.4, 3.65/9.4, 4.65/9.4}
\node[left] at (\i,\j) {\(x_5\)};

\draw[thick] (1,0) -- (1,2) -- (2,3) -- (1,4) -- (1,6) -- (2,7) -- (2,9) -- (1,10)
(2,0) -- (2,2) -- (3,3) -- (3,4) -- (4,5) -- (3,6) -- (3,9) -- (2,10)
(3,0) -- (4,1) -- (3,2) -- (4,3) -- (4,4) -- (5,5) -- (5,7) -- (4,8) -- (4,9) -- (3,10)
(4,0) -- (5,1) -- (5,2) -- (6,3) -- (6,4) -- (7,5) -- (6,6) -- (6,7) -- (5,8) -- (5,9) -- (4,10);

\end{tikzpicture}

}
\caption{An example of a family \(\mathbf{P}=(P_1,P_2,P_3,P_4)\) of nonintersecting paths.}
\label{fig:7}
\end{figure}
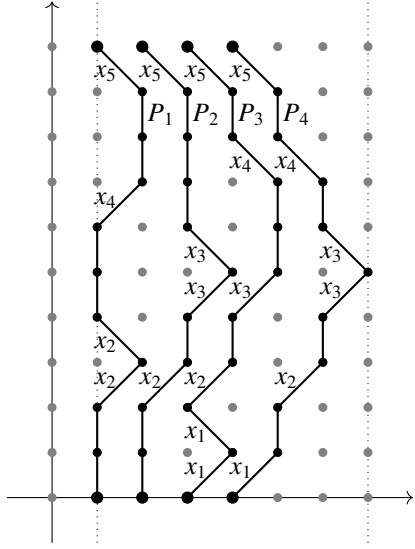

We define the weight $\omega(s_v)$ of a vertical step~$s_v$ to be~1, the weight
$\omega(s_f)$ of a forward diagonal step $s_f$ from $(i,2j-2)$ to $(i+1,2j-1)$
to be $x_j$, and the weight $\omega(s_b)$ of a backward diagonal step from
$(i,2j-1)$ to $(i-1,2j)$ to be also~$x_j$. The weight of a path equals the
product of the weights of its steps.

Let \( L(u\to v) \) denote the set of lattice paths from \( u \) to \( v \). Let
\( L(a,b;u\to v) \) denote the set of \( P\in L(u\to v) \) for which every point
\( (i,2j) \) of even height in \( P \) satisfies \( a\le i\le b \). Note that, if
\( P\in L\big(a,b;(r,0)\to (s,2n)\big) \) for some integers \( r,s \) with \( a\le
r,s\le b \), then every point \( (i,2j+1) \) of odd height in \( P \) always
satisfies \( a\le i\le b+1 \).

\begin{lem}\label{lem:L_ab}
  For positive integers \( i,j,N \) and \( n \) with \( 1\le i,j<N \), we have
  \begin{equation}
    \label{eq:F1}
    F_{-i+j,2N+2}(x_1,\dots,x_n)-F_{i+j,2N+2}(x_1,\dots,x_n)
    = \sum_{P\in L\big(1,N;(i,0)\to (j,2n)\big)} \omega(P).
  \end{equation}
\end{lem}
\begin{proof}
  Since
  \begin{equation}\label{eq:w=f}
    \sum_{P\in L\big((a,0)\to (b,2n)\big)} \omega(P) = \sum_{k\in \Z}
    e_k(x_1,\dots,x_n) e_{b-a+k}(x_1,\dots,x_n) = f_{b-a}(x_1,\dots,x_n),
  \end{equation}
  we can rewrite \eqref{eq:F1} as follows:
  \begin{equation}
    \label{eq:F=L1}
    \underset{P\in L\big((\ep i+k(2N+2),0)\to (j,2n)\big)}
    {\sum_{k \in\Z,\,\,\ep\in\{-1,+1\}}} \ep\cdot\omega(P)
    = \sum_{P\in L\big(1,N;(i,0)\to (j,2n)\big)} \omega(P).
  \end{equation}
We prove~\eqref{eq:F=L1} by constructing a sign-reversing involution, say
  \( \ph \), on the set
  \[
    \mathcal B:=\bigg(\bigcup_{k\in\Z}
    \bigcup_{\ep\in \{-1,+1\}} L\big((\ep i+k(2N+2),0)\to (j,2n)\big)
    \bigg)\,\big\backslash\, L\big(1,N;(i,0)\to (j,2n)\big),
  \]
  where the sign of \( P\in L\big((\ep i+k(2N+2),0)\to (j,2n)\big) \) is defined to be
  \( \ep \).

  Let \( P\in\mathcal{B}\). Then there is at least one point \( (x,y) \) in \( P
  \) such that \( x\in \{0,N+1\} \) and \( y \) is even. Choose such a point \(
  (x,y) \) with maximal \( y \). We let $P'$ be the portion of~$P$ from its
  starting point up to~$(x,y)$ and $P''$ the portion from $(x,y)$ up to $P$'s
  end point. Then we define \( \ph(P)=\mathfrak{R}(P')P'' \), where
  $\mathfrak R$ means the modified reflection from~\cite[Proof
  of~(1.1)]{Krattenthaler1996}. To be specific, beginning from $P$'s starting
  point up to $(x,y)$, we group steps in pairs. If a pair consists of two
  vertical steps then we leave them invariant, as well as if the pair consists
  of a forward diagonal step followed by a backward diagonal step. If a pair
  consists of a forward diagonal step followed by a vertical step, then we
  replace this pair by the pair consisting of a vertical step followed by a
  backward diagonal step, and vice versa. 

  Since the map $\ph$ leaves the portion of the path after $(x,y)$ invariant,
  twofold application of $\ph$ will bring us back to~$P$. Thus, the map $\ph$ is
  indeed an involution. To see that \( \ph \) is sign-reversing
  on the set $\mathcal B$, suppose that the
  starting point of \( P \) is \( (\ep i+k(2N+2),0) \) so that the sign of \( P
  \) is \( \ep \). Then the starting point of \( \ph(P) \) is \( (-\ep
  i-k(2N+2),0) \) if \( x=0 \),
  respectively \( (-\ep i+(1-k)(2N+2),0) \) if \(
  x=N+1 \). Thus the sign of $\ph(P)$ is always \( -\ep \). Furthermore, the
  map $\ph$ is clearly weight-preserving. This completes the proof
  of~\eqref{eq:F=L1}, and thus of \eqref{eq:F1}.
\end{proof}

We are now in the position to prove the claimed combinatorial
interpretation of the determinant on the right-hand side
of~\eqref{eq:aGBK1s}. 

\begin{proof}[Proof of Theorem \ref{thm:OT-L1}]

\begin{figure}
  \ytableausetup{smalltableaux}
\(
\emptyset
~\!\!\subseteq~
{\small
\begin{ytableau}
~\\
~
\end{ytableau}}
~\supseteq~
{\small
\begin{ytableau}
~
\end{ytableau}}
~\!\!\subseteq~
{\small
\begin{ytableau}
~&~\\
~\\
~\\
~
\end{ytableau}}
~\supseteq~
{\small
\begin{ytableau}
~&~\\
~\\
~
\end{ytableau}}
~\!\!\subseteq~
{\small
\begin{ytableau}
~&~&~\\
~&~\\
~&~
\end{ytableau}}
~\supseteq~
{\small
\begin{ytableau}
~&~\\
~&~\\
~
\end{ytableau}}
~\!\!\subseteq~
{\small
\begin{ytableau}
~&~\\
~&~\\
~\\
~
\end{ytableau}}
~\supseteq~
{\small
\begin{ytableau}
~\\
~\\
~\\
~
\end{ytableau}}
~\!\!\subseteq~
{\small
\begin{ytableau}
~\\
~\\
~\\
~
\end{ytableau}}
~\supseteq~
\emptyset
\)
\caption{The up-down tableau corresponding to the family \( \mathbf{P} \) in Figure \ref{fig:7}.}
\label{fig:8}
\end{figure}

By Lemma~\ref{lem:L_ab}, the left-hand side of~\eqref{eq:OT-L1}
is equal to
\begin{equation}\label{eq:PinX}
  \sum_{\sigma\in \mathfrak{S}_h} \sgn(\sigma)
  \prod_{j=1}^h \sum_{P_j\in L\big(1,h+w; (\sigma(j),0)\to (j,2n)\big)} \omega(P_j)
  =\sum_{\mathbf{P}\in X} \sgn(\mathbf{P}) \omega(\mathbf{P}),
\end{equation}
where \( X \) is the set of families of paths $\mathbf P=(P_1,P_2,\dots,P_h)$
for which \( P_j\in L\big(1,h+w; (\sigma(j),0)\to (j,2n)\big) \), \( j=1,2,\dots,h \),
for some \( \sigma\in \mathfrak{S}_h \), \( \sgn(\mathbf{P})=\sgn(\sigma) \) and \(
\omega(\mathbf{P})=\prod _{j=1} ^{h}\omega(P_j) \). We claim that the right-hand
side of~\eqref{eq:PinX} is equal to \( \sum_{\mathbf{P}\in Y} \sgn(\mathbf{P})
\omega(\mathbf{P}) \), where \( Y \) is the set of all families $\mathbf
P=(P_1,P_2,\dots,P_h)$ of nonintersecting lattice paths with $P_j\in
L\big(1,h+w;(j,0)\to (j,2n)\big)$ for $j=1,2,\dots,h$.

To see this, suppose that \( \mathbf P=(P_1,P_2,\dots,P_h)\in X \) has some
intersection points. We choose the intersection point \( (x, y) \) with maximal
\( y \), and if there are several of them, then among these the one with minimal
\( x \). Let $(s,t)$ be lexicographically minimal such that $(x,y)$ is a point
on $P_s$ {\it and\/}~$P_t$. We then do the usual
``Lindstr\"om--Gessel--Viennot'' switching (cf.\ \cite{LindAA,GeViAA})
of the initial portions of $P_s$ and
$P_t$ up to the intersection point $(x,y)$. If \( \sgn(\mathbf{P})=\sgn(\sigma)
\), then the corresponding family \( \mathbf{P}' \) satisfies \(
\sgn(\mathbf{P}')=\sgn(\sigma\circ (s,t))=-\sgn(\sigma) \). Thus the map \(
\mathbf{P}\mapsto \mathbf{P}' \) gives a sign-reversing and weight-preserving
involution on \( X\setminus Y \), and the claim is proved.

\medskip
It remains to show that there is a weight-preserving bijection between \( Y \)
and \( \UD_n(h,w) \). In order to see this, we start from a family $\mathbf
P=(P_1,P_2,\dots,P_h)\in Y$ of nonintersecting lattice paths. See
Figure~\ref{fig:7}, which shows an example for $n=5$, $h=4$, and $w=2$. In
particular, the bounding line $x=h+w+1=7$ is indicated by the dotted line in the
figure.

For $j=0,1,\dots,2n$, we now read the vector of $x$-coordinates
of the points of the lattice paths that are at height~$j$.
For our running example from Figure~\ref{fig:7}, we get the sequence
\begin{multline*}
(1,2,3,4),\
(1,2,4,5),\
(1,2,3,5),\
(2,3,4,6),\
(1,3,4,6),\
(1,4,5,7),\\
(1,3,5,6),\
(2,3,5,6),\
(2,3,4,5),\
(2,3,4,5),\
(1,2,3,4).
\end{multline*}
In the next step, we subtract $i$ from the $i$-th coordinate of each
vector, $i=1,2,\dots,h$, and we reverse the order of the coordinates.
In our example, this leads to
\begin{multline*}
(0,0,0,0),\
(1,1,0,0),\
(1,0,0,0),\
(2,1,1,1),\
(2,1,1,0),\
(3,2,2,0),\\
(2,2,1,0),\
(2,2,1,1),\
(1,1,1,1),\
(1,1,1,1),\
(0,0,0,0).
\end{multline*}
Finally, each of the vectors is interpreted as a Young diagram. In this manner,
in our example we arrive at the \( (h,w) \)-up-down tableau in
Figure~\ref{fig:8}. It is straightforward to verify that this correspondence is
a weight-preserving bijection between \( Y \) and \( \UD_n(h,w) \), which completes
the proof.
\end{proof}

An immediate consequence of Theorem~\ref{thm:OT-L1} is the following corollary,
which gives a combinatorial interpretation of the right-hand side of the affine
bounded Littlewood identity~\eqref{eq:aGBK1s} when \( w \) is odd.

\begin{cor} \label{cor:OT-L1}
  The coefficient of $x_1^{m_1}x_2^{m_2}\cdots x_n^{m_n}$ in
  \[
\sum_{k\ge0}e_k(\vx)
\det_{1\le i,j\le h}\left(
  F_{-i+j,2h+2w+2}(\vx)-F_{i+j,2h+2w+2}(\vx)  \right)
  \]
  equals the number of \( (h,w) \)-up-down tableaux \( (\la^0,\la^1,\dots,\la^{2n}) \)
  which satisfy $-\vert\la^{2i-2}\vert+2\vert\la^{2i-1}\vert-\vert\la^{2i}\vert
  \in \{m_{i},m_i-1\}$, $i=1,2,\dots,n$.
\end{cor}

We now turn to the right-hand sides of the affine
bounded Littlewood identities~\eqref{eq:aGBK1s} with \( w \) even
and~\eqref{eq:aGBK2s}. Our combinatorial interpretations
of these identities are in terms of {\it marked\/} up-down tableaux,
which we define next.

\begin{defn}
  An \emph{\( (h^*,w) \)-up-down tableau} (respectively \emph{\( (h,w^*) \)-up-down
    tableau}) is a pair \( (T,M) \) of an \( (h,w) \)-up-down tableau \(
  T=(\la^0,\la^1,\dots,\la^{2n}) \) and a subset \( M\subseteq \{1,2,\dots,n\}
  \) with the property that if \( j\in M \) then \( \lambda^{2j-1}_h=0 \) (respectively \(
  \lambda^{2j-1}_1=w+1 \)). An \emph{\( (h^*,w^*) \)-up-down tableau} is a
  triple \( (T,M_1,M_2) \) of an \( (h,w) \)-up-down tableau \(
  T=(\la^0,\la^1,\dots,\la^{2n}) \) and subsets \( M_1,M_2\subseteq
  \{1,2,\dots,n\} \) with the property that if \( j\in M_1 \) then \( \lambda^{2j-1}_h=0 \),
  and if \( j\in M_2 \) then \( \lambda^{2j-1}_1=w+1 \).
\end{defn}

One may consider an \( (h^*,w) \)-up-down tableau as an \( (h,w) \)-up-down
tableau \( (\la^0,\la^1,\dots,\la^{2n}) \) in which each last part \(
\la^{2j-1}_h \), if it exists, may be marked if it equals zero. We may
think of \( (h,w^*) \)-up-down tableaux and \( (h^*,w^*) \)-up-down
tableaux in an analogous way.

We define \( \UD_n(h^*,w) \) to be the set of \( (h^*,w) \)-up-down tableaux \(
(T,M) \) with \( T\in \UD_n(h,w) \). The sets \( \UD_n(h,w^*) \) and
\(\UD_n(h^*,w^*) \) are defined similarly. For \( (T,M_1)\in \UD_n(h^*,w) \), \(
(T,M_2)\in \UD_n(h,w^*) \), and \( (T,M_1,M_2)\in \UD_n(h^*,w^*) \), we define
their weights by
\begin{align*}
  \omega(T,M_1) & = \omega(T)\prod_{j\in M_1} x_j,\\
  \omega(T,M_2) & = \omega(T)\prod_{j\in M_2} \left( -\frac{1}{x_j} \right),\\
  \omega(T,M_1,M_2) & = \omega(T)\prod_{j\in M_1} x_j
                        \prod_{j\in M_2} \left( -\frac{1}{x_j} \right).
\end{align*}

In the theorem below, we present combinatorial interpretations of the
determinants on the right-hand sides 
of~\eqref{eq:aGBK1s} with \( w \) even and of~\eqref{eq:aGBK2s}. 

\begin{thm} \label{thm:OT-L1'}
  For positive integers $h$, $w$, and \( n \), we have
  \begin{align}
    \label{eq:OT-L2} \det_{1\le i,j\le h}\left(
    F_{-i+j,2h+2w+1}(x_1,\dots,x_n)-F_{i+j,2h+2w+1}(x_1,\dots,x_n) \right)
    &=\sum_{(T,M)\in\UD_n(h,w^*)} \omega(T,M),\\
    \label{eq:OT-L3}
    \det_{1\le i,j\le h}\left(
    \overline{F}_{-i+j,2h+2w+1}(x_1,\dots,x_n) + \overline{F}_{i+j-1,2h+2w+1}(x_1,\dots,x_n)
    \right)
    &=\sum_{(T,M)\in\UD_n(h^*,w)} \omega(T,M),\\
    \label{eq:OT-L4}
    \det_{1\le i,j\le h}\left(
    \overline{F}_{-i+j,2h+2w}(x_1,\dots,x_n) + \overline{F}_{i+j-1,2h+2w}(x_1,\dots,x_n)
    \right)
    &=\sum_{(T,M_1,M_2)\in\UD_n(h^*,w^*)} \omega(T,M_1,M_2).
  \end{align}
\end{thm}

Again, we need some preparations first before we are able to turn to
the proof of the theorem. It will again be based on nonintersecting
lattice paths. Here, the lattice paths will come with a ``decoration''
though. 

\begin{defn}
  A \emph{$1$-branch point} of a lattice path \( P \) is a point \( (1,2j-1) \)
  with the property that \( P \) passes through \( (1,2j-2)\), \((1,2j-1) \),
  and \( (1,2j) \).
  For an integer \( N>1 \), an \emph{$N$-branch point} of a lattice path \( P \)
  is a point \( (N+1,2j-1) \) with the property that \( P \) passes through \( (N,2j-2)\),
  \((N+1,2j-1) \), and \( (N,2j) \).

  For an integer \( t\ge1 \), a \emph{$t$-marked lattice path} is a pair \(
  (P,M) \) of a lattice path \( P \) and a set \( M \) of \( t \)-branch points
  of \( P \). For integers \( t_1,t_2\ge1 \), a \emph{$(t_1,t_2)$-marked lattice
    path} is a triple \( (P,M_1,M_2) \), where \( P \) is a lattice
  path, and \(
  M_i \) is a set of \( t_i \)-branch points of \( P \) for \( i=1,2 \).
\end{defn}

For a \( t \)-branch point \( u \) of height \( 2j-1 \), we define its weight by
\[
  \omega(u) =
  \begin{cases}
    x_j, & \mbox{if \( t=1 \)},\\
    -1/x_j, & \mbox{if \( t>1 \)}.
  \end{cases}
\]
We also define
\[
\omega(P,M)=\omega(P)\omega(M) \quad\text{and}\quad
\omega(P,M_1,M_2)=\omega(P)\omega(M_1)\omega(M_2),
\]
where \( \omega(M)=\prod_{u\in M}\omega(u) \).

A marked lattice path may be considered as a lattice path in which some branch
points are marked. Note that the weight of a marked lattice path also contains
a sign,
which is the product of \( -1 \) for each marked \( t \)-branch point
for \( t>1 \).

Recall that \( L(u\to v) \) is the set of lattice paths from \( u \) to \( v \)
and \( L(a,b;u\to v) \) is the set of \( P\in L(u\to v) \) with the property that every point
\( (i,2j) \) of even height in \( P \) satisfies \( a\le i\le b \). Let \(
L(1^*,b;u\to v) \) denote the set of \( 1 \)-marked lattice paths \( (P,M) \)
with \( P\in L(1,b;u\to v) \). The sets \( L(1,b^*;u\to v) \) and \(
L(1^*,b^*;u\to v) \) are defined similarly.

\begin{lem}\label{lem:L_ab'}
  For positive integers \( i,j,N \), and \( n \) with \( 1\le i,j<N \), we have
  \begin{align}
    \label{eq:F2}
    F_{-i+j,2N+1}(x_1,\dots,x_n)-F_{i+j,2N+1}(x_1,\dots,x_n)
    &= \sum_{(P,M)\in L\big(1,N^*;(i,0)\to (j,2n)\big)} \omega(P,M),\\
    \label{eq:F3}
    \overline{F}_{-i+j,2N+1}(x_1,\dots,x_n)+\overline{F}_{i+j-1,2N+1}(x_1,\dots,x_n)
    &= \sum_{(P,M)\in L\big(1^*,N;(i,0)\to (j,2n)\big)} \omega(P,M),\\
    \label{eq:F4}
    \overline{F}_{-i+j,2N}(x_1,\dots,x_n)+\overline{F}_{i+j-1,2N}(x_1,\dots,x_n)
    &= \sum_{(P,M_1,M_2)\in L\big(1^*,N^*;(i,0)\to (j,2n)\big)} \omega(P,M_1,M_2).
  \end{align}
\end{lem}
\begin{proof}
  By \eqref{eq:w=f}, we can rewrite the identities 
  \eqref{eq:F2}--\eqref{eq:F4}
  as follows:
  \begin{align}
    \label{eq:F=L2}
    \underset{P\in L\big((\ep i+k(2N+1),0)\to (j,2n)\big)}
    {\sum_{k \in\Z,\,\,\ep\in\{-1,+1\}}} \ep\cdot\omega(P)
    &= \sum_{(P,M)\in L\big(1,N^*;(i,0)\to (j,2n)\big)} \omega(P,M),\\
    \label{eq:F=L3}
    \underset{P\in L\big((\ep (i-1/2)+1/2+k(2N+1),0)\to (j,2n)\big)}
    {\sum_{k \in\Z,\,\,\ep\in\{-1,+1\}}}
    (-1)^k\omega(P)
    &= \sum_{(P,M)\in L\big(1^*,N;(i,0)\to (j,2n)\big)} \omega(P,M),\\
    \label{eq:F=L4}
    \underset{P\in L\big((\ep (i-1/2)+1/2+k(2N),0)\to (j,2n)\big)}
    {\sum_{k \in\Z,\,\,\ep\in\{-1,+1\}}}
    (-1)^k\omega(P)
    &= \sum_{(P,M_1,M_2)\in L\big(1^*,N^*;(i,0)\to (j,2n)\big)} \omega(P,M_1,M_2).
  \end{align}

  For the identity~\eqref{eq:F=L2} we proceed similarly as in the proof of
  Lemma~\ref{lem:L_ab} by constructing a sign-reversing involution using a
  point \( (x,y) \) in \( P \) with \( x\in \{0,2N+1\} \) and \( y \) even, if
  such a point exists. This will give
  \begin{equation}\label{eq:ew2}
    \underset{P\in L\big((\ep i+k(2N+1),0)\to (j,2n)\big)}
    {\sum_{k \in\Z,\,\,\ep\in\{-1,+1\}}} \ep\cdot\omega(P)
    =  \sum_{P\in L\big(1,2N;(i,0)\to (j,2n)\big)} \omega(P)
    -  \sum_{P\in L\big(1,2N;(2N+1-i,0)\to (j,2n)\big)} \omega(P).
  \end{equation}
  Now let \( P\in L\big(1,2N;(i,0)\to (j,2n)\big) \cup
  L\big(1,2N;(2N+1-i,0)\to (j,2n)\big) \).
  We will construct an \( N \)-marked path \( (Q,M)\in
  L\big(1,N^*;(i,0)\to (j,2n)\big)
  \) as follows.

  First, let \( (Q,M)=(P,\emptyset) \). We will modify \( (Q,M) \) repeatedly
  until it becomes an element in \( L\big(1,N^*;(i,0)\to (j,2n)\big) \). If \( (Q,M)\in
  L\big(1,N^*;(i,0)\to (j,2n)\big) \), then there is nothing to do. Otherwise, choose
  the largest even integer \( y \) such that \( (N+1,y) \) is a point in \( Q
  \). Then \( Q \) must pass through all points \( (N+1,y), (N+1,y+1) \) and \(
  (N,y+2) \). Let $Q'$ be the portion of~$Q$ from its starting point up
  to~$(N+1,y)$ and $Q''$ the portion from $(N,y+2)$ up to $Q$'s end point,
  i.e., \( Q=Q's_vs_bQ'' \). Then
  we update \( Q \) to \(\mathfrak{R}(Q')s_fs_bQ'' \) and add the point \(
  (N+1,y+1) \) to the set \( M \), where
  $\mathfrak R$ has the same meaning as in the proof of Lemma~\ref{lem:L_ab}.
  Repeating this process eventually yields \(
  (Q,M)\in L\big(1,N^*;(i,0)\to (j,2n)\big) \). For example, see Figure~\ref{fig:F=L2}.
  It is easy to see that the map \( P\mapsto (Q,M)\) is a weight-preserving bijection between \( L\big(1,2N;(i,0)\to
  (j,2n)\big) \cup L\big(1,2N;(2N+1-i,0)\to (j,2n)\big) \) and \(
  L\big(1,N^*;(i,0)\to (j,2n)\big)
  \). This shows that the right-hand sides of~\eqref{eq:F=L2}
  and~\eqref{eq:ew2} are equal, completing the proof of~\eqref{eq:F=L2}.

\begin{figure}

\small{

\begin{tikzpicture}[scale=0.5]

\draw[->] (-1,0) -- (8,0);
\draw[->] (0,-1) -- (0,15);
\draw[dotted] (7,-1) -- (7,15);
\draw[dotted] (4,-1) -- (4,15);
\draw[dotted] (1,-1) -- (1,15);

\foreach \i in {0,...,7}
\foreach \j in {0,...,14}
\filldraw[fill=gray, color=gray] (\i,\j) circle (2.5pt);

\filldraw (1,0) circle (3.5pt);
\filldraw (3,14) circle (3.5pt);

\foreach \i/\j in {1/0,2/1,2/2,3/3,3/4,4/5,4/6,5/7,4/8,5/9,5/10,5/11,4/12,4/13,3/14}
\filldraw (\i,\j) circle (2.5pt);

\draw[thick] (1,0) --(2,1) -- (2,2) -- (3,3) -- (3,4) -- (4,5) -- (4,6) -- (5,7) -- (4,8) -- (5,9) -- (5,10) -- (5,11) -- (4,12) -- (4,13) -- (3,14);

\node at (-0.5,14) {\( 2n \)};
\node at (-0.5,-0.5) {\( O \)};
\node at (4,-0.5) {\( N+1 \)};
\node at (7,-0.5) {\( 2N+1 \)};
\node at (3,14.5) {\( P \)};
\node at (9,7) {\( \leftrightarrow \)};

\end{tikzpicture}
\begin{tikzpicture}[scale=0.5]

\draw[->] (-1,0) -- (8,0);
\draw[->] (0,-1) -- (0,15);
\draw[dotted] (7,-1) -- (7,15);
\draw[dotted] (4,-1) -- (4,15);
\draw[dotted] (1,-1) -- (1,15);

\foreach \i in {0,...,7}
\foreach \j in {0,...,14}
\filldraw[fill=gray, color=gray] (\i,\j) circle (2.5pt);

\filldraw (6,0) circle (3.5pt);
\filldraw (3,14) circle (3.5pt);

\foreach \i/\j in {6/0,6/1,5/2,5/3,4/4,4/5,3/6,4/7,3/8,3/9,2/10,3/11,3/12,3/14}
\filldraw (\i,\j) circle (2.5pt);

\draw[thick] (6,0) --(6,1) -- (5,2) -- (5,3) -- (4,4) -- (4,5) -- (3,6) -- (4,7) -- (3,8) -- (3,9) -- (2,10) -- (3,11) -- (3,12) -- (4,13) -- (3,14);

\node at (4,13) {\( \circledast \)};

\node at (-0.5,14) {\( 2n \)};
\node at (-0.5,-0.5) {\( O \)};
\node at (4,-0.5) {\( N+1 \)};
\node at (7,-0.5) {\( 2N+1 \)};
\node at (9,7) {\( \leftrightarrow \)};

\end{tikzpicture}
\begin{tikzpicture}[scale=0.5]

\draw[->] (-1,0) -- (5,0);
\draw[->] (0,-1) -- (0,15);
\draw[dotted] (4,-1) -- (4,15);
\draw[dotted] (1,-1) -- (1,15);

\foreach \i in {0,...,4}
\foreach \j in {0,...,14}
\filldraw[fill=gray, color=gray] (\i,\j) circle (2.5pt);

\filldraw (1,0) circle (3.5pt);
\filldraw (3,14) circle (3.5pt);

\foreach \i/\j in {1/0,2/1,2/2,3/3,3/4,3/6,4/7,3/8,3/9,2/10,3/11,3/12,3/14}
\filldraw (\i,\j) circle (2.5pt);

\draw[thick] (1,0) --(2,1) -- (2,2) -- (3,3) -- (3,4) -- (4,5) -- (3,6) -- (4,7) -- (3,8) -- (3,9) -- (2,10) -- (3,11) -- (3,12) -- (4,13) -- (3,14);

\node at (4,13) {\( \circledast \)};
\node at (4,5) {\( \circledast \)};

\node at (-0.5,14) {\( 2n \)};
\node at (-0.5,-0.5) {\( O \)};
\node at (4,-0.5) {\( N+1 \)};
\node at (3,14.5) {\( (Q,M) \)};

\end{tikzpicture}

}
\caption{An example of a path \( P \in L\big(1,2N;(i,0)\to(j,2n)\big)
  \) and its corresponding \( N \)-marked path \( (Q,M)\in
  L\big(1,N^*;(i,0)\to (j,2n)\big) \) with \( N=3, i=1, j=2 \) and \(n=7\). Each \( N \)-branch point 
  \( \in M \)
  is indicated by \( \circledast  \).}
\label{fig:F=L2}
\end{figure}

  \medskip

  The identity \eqref{eq:F=L3} can be proved similarly as~\eqref{eq:F=L2}. We
  first construct a sign-reversing involution using a point \( (x,y) \) in \(
  P \) with \( x\in \{-N,N+1\} \) and \( y \) even, if such a point exists. This
  will give
  \[
  \underset{P\in L\big((\ep (i-1/2)+1/2+k(2N+1),0)\to (j,2n)\big)}
  {\sum_{k \in\Z,\,\,\ep\in\{-1,+1\}}}
    (-1)^k\omega(P)
    =  \sum_{P\in L\big(1-N,N;(i,0)\to (j,2n)\big)} \omega(P)
    +  \sum_{P\in L\big(1-N,N;(1-i,0)\to (j,2n)\big)} \omega(P).
  \]
  For each \( P\in L\big(1-N,N;(i,0)\to (j,2n)\big)\cup
  L\big(1-N,N;(1-i,0)\to (j,2n)\big) \),
  we construct \( (Q,M)\in L\big(1^*,N;(i,0)\to (j,2n)\big) \) as follows. Let \(
  (Q,M)=(P,\emptyset) \). If \( (Q,M)\in L\big(1^*,N;(i,0)\to
  (j,2n)\big) \), then there
  is nothing to do. Otherwise, choose the largest even integer \( y \) such that
  \( (0,y) \) is a point in \( Q \). Then \( Q \) must pass through all points
  \( (0,y), (1,y+1) \) and \( (1,y+2) \). Let $Q'$ be the portion of~$Q$ from
  its starting point up to~$(0,y)$ and $Q''$ the portion from $(1,y+2)$ up to
  $Q$'s end point. Then we update \( Q \) to \(\mathfrak{R}(Q')s_vs_vQ'' \)
  and add the point \( (1,y+1) \) to the set \( M \). Repeating this process
  eventually yields \( (Q,M)\in L\big(1^*,N;(i,0)\to (j,2n)\big) \). An
  example is given in Figure~\ref{fig:F=L3}. 
  Again, the map \( P \mapsto (Q,M)\) is a weight-preserving bijection between \( L\big(1-N,N;(i,0)\to
  (j,2n)\big)\cup L\big(1-N,N;(1-i,0)\to (j,2n)\big) \) and \(
  L\big(1^*,N;(i,0)\to (j,2n)\big) \),
  and we obtain~\eqref{eq:F=L3}.

    \begin{figure}

\small{

\begin{tikzpicture}[scale=0.5]

\draw[->] (-1,0) -- (8,0);
\draw[->] (3,-1) -- (3,15);
\draw[dotted] (7,-1) -- (7,15);
\draw[dotted] (4,-1) -- (4,15);
\draw[dotted] (1,-1) -- (1,15);

\foreach \i in {0,...,7}
\foreach \j in {0,...,14}
\filldraw[fill=gray, color=gray] (\i,\j) circle (2.5pt);

\filldraw (4,0) circle (3.5pt);
\filldraw (5,14) circle (3.5pt);

\foreach \i/\j in {4/0,5/1,4/2,4/3,3/4,3/5,2/6,3/7,3/8,4/9,4/10,4/11,4/12,5/13,5/14}
\filldraw (\i,\j) circle (2.5pt);

\draw[thick] (4,0) --(5,1) -- (4,2) -- (4,3) -- (3,4) -- (3,5) -- (2,6) -- (3,7) -- (3,8) -- (4,9) -- (4,10) -- (4,11) -- (4,12) -- (5,13) -- (5,14);

\node at (2.5,14) {\( 2n \)};
\node at (1,-0.5) {\( 1-N \)};
\node at (2.5,-0.5) {\( O \)};
\node at (7,-0.5) {\( N+1 \)};
\node at (5,14.5) {\( P \)};
\node at (9,7) {\( \leftrightarrow \)};

\end{tikzpicture}
\begin{tikzpicture}[scale=0.5]

\draw[->] (-1,0) -- (8,0);
\draw[->] (3,-1) -- (3,15);
\draw[dotted] (7,-1) -- (7,15);
\draw[dotted] (4,-1) -- (4,15);
\draw[dotted] (1,-1) -- (1,15);

\foreach \i in {0,...,7}
\foreach \j in {0,...,14}
\filldraw[fill=gray, color=gray] (\i,\j) circle (2.5pt);

\filldraw (3,0) circle (3.5pt);
\filldraw (5,14) circle (3.5pt);

\foreach \i/\j in {3/0,4/1,3/2,4/3,4/4,5/5,5/6,5/7,4/8,4/10,4/11,4/12,5/13,5/14}
\filldraw (\i,\j) circle (2.5pt);

\draw[thick] (3,0) --(4,1) -- (3,2) -- (4,3) -- (4,4) -- (5,5) -- (5,6) -- (5,7) -- (4,8) -- (4,10) -- (4,11) -- (4,12) -- (5,13) -- (5,14);

\node at (4,9) {\( \circledast \)};

\node at (2.5,14) {\( 2n \)};
\node at (1,-0.5) {\( 1-N \)};
\node at (2.5,-0.5) {\( O \)};
\node at (7,-0.5) {\( N+1 \)};
\node at (9,7) {\( \leftrightarrow \)};

\end{tikzpicture}
\begin{tikzpicture}[scale=0.5]

\draw[->] (-1,0) -- (5,0);
\draw[->] (0,-1) -- (0,15);
\draw[dotted] (4,-1) -- (4,15);
\draw[dotted] (1,-1) -- (1,15);

\foreach \i in {0,...,4}
\foreach \j in {0,...,14}
\filldraw[fill=gray, color=gray] (\i,\j) circle (2.5pt);

\filldraw (1,0) circle (3.5pt);
\filldraw (2,14) circle (3.5pt);

\foreach \i/\j in {1/0,2/1,1/2,1/4,2/5,2/6,2/7,1/8,1/10,1/11,1/12,2/13,2/14}
\filldraw (\i,\j) circle (2.5pt);

\draw[thick] (1,0) --(2,1) -- (1,2) -- (1,3) -- (1,4) -- (2,5) -- (2,6) -- (2,7) -- (1,8) -- (1,9) -- (1,10) -- (1,11) -- (1,12) -- (2,13) -- (2,14);

\node at (1,9) {\( \circledast \)};
\node at (1,3) {\( \circledast \)};

\node at (-0.5,14) {\( 2n \)};
\node at (-0.5,-0.5) {\( O \)};
\node at (4,-0.5) {\( N+1 \)};
\node at (2,14.5) {\( (Q,M) \)};

\end{tikzpicture}

}
\caption{An example of a path \( P \in L\big(1-N,N;(i,0)\to(j,2n)\big)
  \) and its corresponding \( 1 \)-marked path \( (Q,M)\in
  L\big(1^*,N;(i,0)\to (j,2n)\big) \) with \( N=3, i=1, j=2\), and \(n=7 \). Each \( 1 \)-branch point
  \( \in M \) 
  is indicated by \( \circledast  \).}
\label{fig:F=L3}
\end{figure}

  \medskip

  For the last identity \eqref{eq:F=L4} we do not need a sign-reversing
  involution. Instead, we construct \( (Q,M_1,M_2)\in
  L\big(1^*,N^*;(i,0)\to (j,2n)\big)
  \) directly from \( P\in L\big((\ep (i-1/2)+1/2+k(2N),0)\to (j,2n)\big) \). To do
  this, as before, we first set \( (Q,M_1,M_2)=(P,\emptyset,\emptyset) \). If \(
  (Q,M_1,M_2)\in L\big(1^*,N^*;(i,0)\to (j,2n)\big) \), then we are done. Otherwise,
  find the largest even \( y \) such that \( (0,y) \) or \( (N+1,y) \) is a
  point in \( Q \). Then we modify \( Q \) by the same method as above and add
  \( (1,y+1) \) to \( M_1 \) (respectively \( (N+1,y+1) \) to \( M_2 \)) if \( (0,y) \)
  (respectively \( (N+1,y) \)) is a point of \( Q \). This proves~\eqref{eq:F=L4}.
\end{proof}

We have now all prerequisites at our disposal to embark on the proof of
Theorem~\ref{thm:OT-L1'}. Since this proof is very similar to the
one of Theorem~\ref{thm:OT-L1}, we content ourselves with providing
a brief sketch.

\begin{proof}[Sketch of proof of Theorem \ref{thm:OT-L1'}]
We proceed in a similar manner as in the proof of
Theorem~\ref{thm:OT-L1}, here however using Lemma~\ref{lem:L_ab'}.
The only thing that needs
  careful thought is whether the chosen intersection point \( (x,y) \) is a \( t
  \)-branch point of some path for \( t\in \{1,h+w\} \). This never happens
  because we have chosen the intersection point \( (x,y) \) with \( y \)
  maximal. To see this, suppose that \( (x,y) \) is a common point of \( P_r \)
  and \( P_s \), and \( (x,y) \) is a \( 1 \)-branch point of \( P_r \). Then \(
  (x,y)=(1,2j-1) \) for some \( j \), and \( P_r \) passes through \( (1,2j-2)
  \) and \( (1,2j) \) as well. Recalling the step set~\eqref{eq:steps}, the only
  possible steps starting from \( (1,2j-1) \) are a vertical step or a
  backward diagonal step. Thus \( P_s \) must pass through \( (1,2j) \) or \(
  (0,2j) \). However, since \( (x,y) \) is chosen to be the intersection point
  with \( y \) maximal, \( P_s \) does not pass through \( (1,2j) \), and since
  \( P_s\in L(1,h+w;u\to v) \) for some points \( u \) and \( v \), it does not
  pass through \( (0,2j) \). This is a contradiction and therefore \( (x,y) \)
  is not a \( 1 \)-branch point of any path. Similarly, it is not an \( (h+w)
  \)-branch point. Therefore the ``Lindstr\"om--Gessel--Viennot'' switching
(cf.\ \cite{LindAA,GeViAA}) works.
\end{proof}

As consequences of Theorem~\ref{thm:OT-L1'} we obtain combinatorial
interpretations for the right-hand sides of the affine bounded Littlewood
identities~\eqref{eq:aGBK1s} with \( w \) even and~\eqref{eq:aGBK2s}. Note that
these combinatorial interpretations use up-down tableaux {\it
  without\/} marking.

We start with the right-hand side of \eqref{eq:aGBK1s} with even~$w$.

\begin{cor} \label{cor:OT-L2}
  The coefficient of $x_1^{m_1}x_2^{m_2}\cdots x_n^{m_n}$ in
  \[
    \sum_{k\ge0}e_k(\vx) \det_{1\le i,j\le h}
    \left( F_{-i+j,2h+2w+1}(\vx)-F_{i+j,2h+2w+1}(\vx) \right)
  \]
  equals the number of \( (h,w) \)-up-down tableaux \( (\la^0,\la^1,\dots,\la^{2n}) \)
  satisfying
    \[
      -\vert\la^{2i-2}\vert+2\vert\la^{2i-1}\vert-\vert\la^{2i}\vert
      =
      \begin{cases}
        \mbox{\( m_{i} \) or \( m_i-1 \)}, & \mbox{if \( \la^{2i-1}_1\le w \) and not \(  \la^{2i-2}_1= \la^{2i-1}_1= \la^{2i}_1=w \)},\\
        \mbox{\( m_{i} \) }, & \mbox{if  \(   \la^{2i-2}_1= \la^{2i-1}_1= \la^{2i}_1=w  \)},\\
        \mbox{\( m_{i}-1 \)}, & \mbox{if \( \la^{2i-1}_1= w+1 \)},
      \end{cases}
    \]
for $i=1,2,\dots,n$.
\end{cor}
\begin{proof}
  By \eqref{eq:OT-L2}, the coefficient in question is equal to \( \sum_{(T,M)\in
    X}(-1)^{|M|} \), where \( X \) is the set of \( (T,M)\in\UD_n(h,w^*) \) with
  \( T=(\lambda^0,\lambda^1,\dots,\lambda^{2n}) \) satisfying
  \[
    -\vert\la^{2i-2}\vert+2\vert\la^{2i-1}\vert-\vert\la^{2i}\vert
    =
    \begin{cases}
      \mbox{\( m_{i} \) or \( m_i-1 \)}, & \mbox{if \( i\not\in M \)},\\
      \mbox{\( m_{i} \) or \( m_i+1 \)}, & \mbox{if \( i\in M \)}.
    \end{cases}
  \]
  Given \( T \) and \( i \) with \( \lambda^{2i-1}_1=w+1 \), there are
  the two
  cases \( i\not\in M \) or \( i\in M \). These two cases cancel with each
  other if \( -\vert\la^{2i-2}\vert+2\vert\la^{2i-1}\vert-\vert\la^{2i}\vert =
  m_i \). Thus, if \( \lambda^{2i-1}_1=w+1 \),
  there remains only the case where \(
  -\vert\la^{2i-2}\vert+2\vert\la^{2i-1}\vert-\vert\la^{2i}\vert \in \{m_i+1,
  m_i-1\} \). If \( \lambda^{2i-1}_1\le w \),
  then we always have \( i\not\in M \), and there are the two cases \(
  -\vert\la^{2i-2}\vert+2\vert\la^{2i-1}\vert-\vert\la^{2i}\vert \in \{m_i,
  m_i-1\} \). Moreover, the case where \( \lambda^{2i-1}_1=w+1
  \) and \(
  -\vert\la^{2i-2}\vert+2\vert\la^{2i-1}\vert-\vert\la^{2i}\vert
  =m_i+1\)
 cancels with the case where \(
 \lambda^{2i-2}_1=\lambda^{2i-1}_1=\lambda^{2i}_1=w  \) and \(
  -\vert\la^{2i-2}\vert+2\vert\la^{2i-1}\vert-\vert\la^{2i}\vert =m_i-1 \). 
This establishes the assertion of the corollary.
\end{proof}

Next we address the right-hand side of \eqref{eq:aGBK2s} with odd~$w$.

\begin{cor} \label{cor:OT-L3}
  The coefficient of $x_1^{m_1}x_2^{m_2}\cdots x_n^{m_n}$ in
  \[
    \det_{1\le i,j\le h}\left(
      \overline{F}_{-i+j,2h+2w+1}(\vx) + \overline{F}_{i+j-1,2h+2w+1}(\vx)
    \right)
  \]
  equals the number of \( (h,w) \)-up-down tableaux \(
  (\la^0,\la^1,\dots,\la^{2n}) \) that satisfy the following properties
  for $i=1,2,\dots,n$:
  \[
    -\vert\la^{2i-2}\vert+2\vert\la^{2i-1}\vert-\vert\la^{2i}\vert
    =
    \begin{cases}
      m_i, & \mbox{if \( \la^{2i-1}_h\ge 1 \)},\\
      \mbox{\( m_i \) or \( m_i-1 \)}, & \mbox{if \( \la^{2i-1}_h=0 \)}.
    \end{cases}
  \]
  \end{cor}
\begin{proof}
This is an immediate consequence of~\eqref{eq:OT-L3}.
\end{proof}

Finally, we give the combinatorial interpretation of the right-hand
side of \eqref{eq:aGBK2s} with even~$w$.

\begin{cor} \label{cor:OT-L4}
  The coefficient of $x_1^{m_1}x_2^{m_2}\cdots x_n^{m_n}$ in
  \[
    \det_{1\le i,j\le h}\left(
      \overline{F}_{-i+j,2h+2w}(\vx) + \overline{F}_{i+j-1,2h+2w}(\vx)
    \right)
  \]
  equals \( A-B \), where \( A \) (respectively \( B \)) is the number of \( (h,w)
  \)-up-down tableaux \( (\la^0,\la^1,\dots,\la^{2n}) \)
  that satisfy the following properties:
  \begin{enumerate}
  \item [(i)]for $i=1,2,\dots,n$,
    \[
      -\vert\la^{2i-2}\vert+2\vert\la^{2i-1}\vert-\vert\la^{2i}\vert
      =
      \begin{cases}
        m_{i}, & \mbox{if \( \la^{2i-1}_h\ne 0 \) and \( \la^{2i-1}_1\ne w+1 \)},\\
        \mbox{\( m_{i} \) or \( m_i+1 \)}, & \mbox{if \( \la^{2i-1}_h\ne 0 \) and \( \la^{2i-1}_1= w+1 \)},\\
        \mbox{\( m_{i} \) or \( m_i-1 \)}, & \mbox{if \( \la^{2i-1}_h=0 \) and \( \la^{2i-1}_1\ne  w+1 \)},\\
        \mbox{\( m_{i}-1 \) or \( m_i+1 \)}, & \mbox{if \( \la^{2i-1}_h=0 \) and \( \la^{2i-1}_1= w+1 \)};\\
      \end{cases}
    \]
  \item [(ii)]there is an even (respectively odd) number of integers $i=1,2,\dots,n$
    satisfying \( -\vert\la^{2i-2}\vert+2\vert\la^{2i-1}\vert-\vert\la^{2i}\vert
    =m_{i}+1 \).
  \end{enumerate}
\end{cor}

\begin{proof}
  By \eqref{eq:OT-L4}, the coefficient in question is equal to \(
  \sum_{(T,M_1,M_2)\in X}(-1)^{|M_2|} \), where \( X \) is the set of\break \(
  (T,M_1,M_2)\in\UD_n(h^*,w^*) \) with \(
  T=(\lambda^0,\lambda^1,\dots,\lambda^{2n}) \) satisfying
  \[
    -\vert\la^{2i-2}\vert+2\vert\la^{2i-1}\vert-\vert\la^{2i}\vert
    =
    \begin{cases}
      m_{i}, & \mbox{if \( i\not\in M_1 \) and \( i\not\in M_2 \)},\\
      \mbox{\( m_i+1 \)}, & \mbox{if \( i\not\in M_1 \) and \( i\in M_2 \)},\\
      \mbox{\( m_i-1 \)}, & \mbox{if \( i\in M_1 \) and \( i\not\in M_2 \)},\\
      \mbox{\( m_{i} \)}, & \mbox{if \( i\in M_1 \) and \( i\in M_2 \)}.
    \end{cases}
  \]
Now fix \( T \) and \( i \) and consider the following four cases.
For brevity, let \( L=-\vert\la^{2i-2}\vert+2\vert\la^{2i-1}\vert-\vert\la^{2i}\vert  \).
\begin{enumerate}
\item If \( \la^{2i-1}_h\ne 0 \) and \( \la^{2i-1}_1\ne w+1 \), then \( i\not\in M_1 \) and \( i\not\in M_2 \).
  Thus \( L=m_i \).
\item If \( \la^{2i-1}_h\ne 0 \) and \( \la^{2i-1}_1= w+1 \), then
  \( i\not\in M_1 \) and (\( i\not\in M_2 \) or \( i\in M_2 \)). Thus \( L\in \{m_i,m_i+1\} \).
\item If \( \la^{2i-1}_h=0 \) and \( \la^{2i-1}_1\ne  w+1 \), then
  (\( i\not\in M_1 \) or \( i\in M_1 \)) and \( i\not\in M_2 \).
  Thus \( L\in \{m_i,m_i-1\} \).
\item If \( \la^{2i-1}_h=0 \) and \( \la^{2i-1}_1= w+1 \), then (\( i\not\in M_1
  \) or \( i\not\in M_1 \)) and (\( i\not\in M_2 \) or \( i\in M_2 \)). Thus \(
  L\in \{m_i-1,m_i,m_i+1\} \). The case \( L=m_i \) occurs twice,
  namely for (\(
  i\not\in M_1 \) or \( i\not\in M_2 \)) and for (\( i\in M_1 \) or \(
  i\in M_2 \)). These two cases
  cancel with each other. Therefore it is only the cases where \(
  L\in\{m_i-1,m_i+1\} \) which remain.
\end{enumerate}
This establishes the assertion of the corollary.
\end{proof}

\section{Cylindric standard Young tableaux and noncrossing and nonnesting
  matchings}
\label{sec:comb-ident}

In this section, we concentrate on the coefficients of $x_1x_2\cdots
x_n$ on both sides of the affine
bounded Littlewood identities in Theorem~\ref{thm:affine_BK_intro}.
Clearly, by Proposition~\ref{prop:cyl_sch} and
Corollaries~\ref{cor:OT-L1}, and~\ref{cor:OT-L2}--\ref{cor:OT-L4}, we
obtain enumeration 
results that connect cylindric standard Young tableaux and certain up-down
tableaux which we shall call {\it
  vacillating tableaux}. (The reader must be warned that
our use of the term ``vacillating" deviates from the one in~\cite{Chen2007}.)
These results are presented in
Corollary~\ref{cor:csyt=VT}. 
On the other hand, from~\cite{Chen2007} (and the
alternative~\cite{KratCE}) we know that these vacillating tableaux are
in bijection with ({\it partial}) {\it matchings}. This allows us to connect
cylindric standard Young tableaux with matchings. The corresponding results 
are the subject of Corollary~\ref{cor:syt=ncnn}.

\medskip
We begin by defining vacillating tableaux. 

\begin{defn}\label{defn:vt}
  An \emph{\( (h,w) \)-vacillating tableau} is a sequence \( (\lambda^{0},
  \lambda^{1}, \dots, \lambda^{n}) \) of partitions that satisfies the following
three conditions:
  \begin{enumerate}
  \item [(i)] \(\lambda^{0} = \lambda^{n}=\emptyset \);
  \item [(ii)]  the
    partitions \( \lambda^{i-1} \) and \( \lambda^{i} \) differ by at most one
    cell, \( i=1,2,\dots,n \);
  \item [(iii)] each $\la^{i}$ has at most \( h \) rows and at most \( w \)
    columns, $i=1,2,\dots,n$.
  \end{enumerate}
  Let \( \VT_n(h,w) \) denote the set of \( (h,w) \)-vacillating tableaux \(
  (\lambda^{0}, \lambda^{1}, \dots, \lambda^{n}) \).
\end{defn}

Suppose \( T=(\lambda^{0}, \lambda^{1}, \dots, \lambda^{n})\in\VT_n(h,w) \).
Then, by definition,
we can identify each \( \lambda^{i} \) with an \( h \)-tuple of
nonincreasing integers,
which is an element of \( \Z^h \). Using this identification, we may also
consider \( T \) as a walk of length \( n \) from \( \vec0 \) to \( \vec0 \)
consisting of steps in \( \{\pm \ep_i:1\le i\le h\}\cup\{\vec0\} \), where
$\ep_i$ is the $i$-th standard basis vector, staying in the
region
\begin{equation}\label{eq:alcove}
  \left\{(x_1,\dots,x_h)\in\Z^h: w\ge x_1\ge\dots\ge x_h\ge0\right\}.
\end{equation}
We define \(\VT_n(h,w^*) \) (respectively \(\VT_n(h^*,w) \)) to be the set of walks in
\(\VT_n(h,w) \) with the property that a zero step can never (respectively only) occur on the
hyperplane \( x_1=w \) (respectively \( x_h=0 \)). We also define \(\VT'_n(h,w) \) to
be the set of walks in \(\VT_n(h,w) \) with the property that a zero step can only occur on
the hyperplane \( x_h=0 \) or \( x_1=w \) but not both. For \(p\in \VT'_n(h,w)
\), let \( z(p) \) denote the number of zero steps on the hyperplane \( x_1=w
\).

By a combination of Definition~\ref{defn:cSchur} and
Corollaries~\ref{cor:OT-L1}, and~\ref{cor:OT-L2}--\ref{cor:OT-L4},
taking the coefficients of \( x_1x_2\cdots x_n \)
on both sides of the affine bounded Littlewood identities
in~\eqref{eq:aGBK1s} and~\eqref{eq:aGBK2s} 
leads us to the following corollary.

\begin{cor}\label{cor:csyt=VT}
  For positive integers $h$, $w$, and \( n \), we have
  \begin{align}
    \label{eq:syt=VT1}
    |\CSYT_n(2h+1,2w+1)| &=  |\VT_n(h,w)|,\\
    \label{eq:syt=VT2}
    |\CSYT_n(2h+1,2w)| &=  |\VT_n(h,w^*)|,\\
    \label{eq:syt=VT3}
    |\CSYT_n(2h,2w+1)| &=   |\VT_n(h^*,w)|,\\
    \label{eq:syt=VT4}
    |\CSYT_n(2h,2w)| &=  \sum_{T\in \VT'_n(h,w)} (-1)^{z(T)}.
  \end{align}
\end{cor}

\begin{proof}
  For the first identity, i.e.,~\eqref{eq:syt=VT1},
  we take the coefficients of \(
  x_1x_2\cdots x_n \) on both sides of~\eqref{eq:aGBK1s} with\break \( w \) replaced
  by \( 2w+1 \). By Proposition~\ref{prop:cyl_sch}, the coefficient of
  $x_1x_2\cdots x_n$ on the
  left-hand side is\break \( |\CSYT_n(2h+1,2w+1)| \). By Corollary~\ref{cor:OT-L1},
  the coefficient of $x_1x_2\cdots x_n$ on the right-hand side is
  equal to the number of \( (h,w) 
  \)-up-down tableaux \( (\la^0,\la^1,\dots,\la^{2n}) \) with the property that each
  subsequence $\la^{2i-2}\subseteq\la^{2i-1}\supseteq\la^{2i}$ satisfies
  one of the following:
\begin{enumerate}
\item $\la^{2i-2}$ and $\la^{2i}$ differ by one cell and
  $\la^{2i-1}=\la^{2i-2}\cup \la^{2i}$ (in other words, \( \la^{2i-1} \) is the larger
  partition between \( \la^{2i-2} \) and \( \la^{2i} \));
\item $\la^{2i-2}=\la^{2i-1}=\la^{2i}$.
\end{enumerate}
At this point, we see that the odd-indexed partitions $\la^{2i-1}$ are
redundant. If we suppress them, then the sequence $(\la^0,\la^2,\dots,\la^{2n})$
is an \( (h,w) \)-vacillating tableau. Thus the coefficient is equal to \(
|\VT_n(h,w)| \), and we obtain~\eqref{eq:syt=VT1}.

\medskip
For the identity \eqref{eq:syt=VT2}, consider the coefficient of \( x_1x_2\cdots x_n \) on
the right-hand side of~\eqref{eq:aGBK1s} with \( w \) replaced by \( 2w \). By
Corollary~\ref{cor:OT-L2}, this is equal to the number of \( (h,w) \)-up-down tableaux \(
(\la^0,\la^1,\dots,\la^{2n}) \) with the property that each subsequence
$\la^{2i-2}\subseteq\la^{2i-1}\supseteq\la^{2i}$ satisfies one of the following:
\begin{enumerate}
\item \( \lambda_1^{2i-1}\le w \), not \( \la^{2i-2}_1=\la^{2i-1}_1=\la^{2i}_1=w\), $\la^{2i-2}$ and $\la^{2i}$ differ by one
  cell, and $\la^{2i-1}=\la^{2i-2}\cup \la^{2i}$;
\item \( \lambda_1^{2i-1}< w \) and $\la^{2i-1}=\la^{2i-1}=\la^{2i}$;
\item  \( \la^{2i-2}_1=\la^{2i-1}_1=\la^{2i}_1=w\), and $\la^{2i-2}$ and $\la^{2i}$ differ by one
  cell;
\item \( \lambda_1^{2i-1}=w+1 \) and $\la^{2i-2}=\la^{2i-1}=\la^{2i}$.
\end{enumerate}
In fact, Cases~(3) and~(4) are impossible. Since the remaining
cases give $(\la^0,\la^2,\dots,\la^{2n})\in \VT_n(h,w^*)$, we obtain~\eqref{eq:syt=VT2}.
The third identity~\eqref{eq:syt=VT3} can be proved similarly.

\medskip
Finally, for the identity~\eqref{eq:syt=VT4},we need
some additional arguments. Consider the coefficient of \(
x_1x_2\cdots x_n \) on the right-hand side of~\eqref{eq:aGBK2s} with \( w \)
replaced by \( 2w \). By Corollary~\ref{cor:OT-L4}, this is equal to \( A-B \),
where \( A \) (respectively \( B \)) is the number of \( (h,w) \)-up-down tableaux \(
(\la^0,\la^1,\dots,\la^{2n}) \) with the property that each subsequence
$\la^{2i-2}\subseteq\la^{2i-1}\supseteq\la^{2i}$ satisfies one of the following:
\begin{enumerate}
\item \( \lambda_h^{2i-1}\ne0 \), \( \lambda_1^{2i-1}\ne w+1 \),
  $\la^{2i-2}$ and $\la^{2i}$ differ by one cell, and $\la^{2i-1}=\la^{2i-2}\cup \la^{2i}$;
\item \( \lambda_h^{2i-1}\ne0 \), \( \lambda_1^{2i-1}= w+1 \),
  $\la^{2i-2}$ and $\la^{2i}$ differ by one cell, and $\la^{2i-1}=\la^{2i-2}\cup \la^{2i}$;
\item \( \lambda_h^{2i-1}\ne0 \), \( \lambda_1^{2i-1}= w+1 \), and
  $\la^{2i-2}=\la^{2i}$ is obtained from \( \lambda^{2i-1} \) by deleting one
  cell in the first row;
\item \( \lambda_h^{2i-1}=0 \), \( \lambda_1^{2i-1}\ne w+1 \),
  $\la^{2i-2}$ and $\la^{2i}$ differ by one cell, and $\la^{2i-1}=\la^{2i-2}\cup \la^{2i}$;
\item \( \lambda_h^{2i-1}=0 \), \( \lambda_1^{2i-1}\ne w+1 \), and
  $\la^{2i-2}=\la^{2i-1}=\la^{2i}$;
\item \( \lambda_h^{2i-1}=0 \), \( \lambda_1^{2i-1}= w+1 \), and
  $\la^{2i-2}=\la^{2i}$ is obtained from \( \lambda^{2i-1} \) by deleting one
  cell in the first row;
\item \( \lambda_h^{2i-1}=0 \), \( \lambda_1^{2i-1}= w+1 \), and
  $\la^{2i-2}=\la^{2i-1}=\la^{2i}$,
\end{enumerate}
where the total number of occurrences of Cases~(3) and~(6) is even (respectively odd).
As before, Cases~(2) and~(7) are impossible because \(
\lambda^{2i-2}_1,\lambda^{2i}_1\le w \). Case~(5) with \( \lambda_1^{2i-1}=w \)
cancels with Case~(6). The remaining cases give
$(\la^0,\la^2,\dots,\la^{2n})\in \VT'_n(h,w)$. More precisely, Cases~(1) and~(4) correspond to the case where \( \lambda^{2i-2} \) and \( \lambda^{2i} \)
differ by one cell, Case~(3) corresponds to the case where \( \lambda^{2i-2} =
\lambda^{2i} \) with \( \lambda^{2i}_h\ne0 \) and \( \lambda^{2i}_1=w \), and
Case~(5) with \( \lambda^{2i-1}_1\le w-1 \) corresponds to the case where \(
\lambda^{2i-2} = \lambda^{2i} \) with \( \lambda^{2i}_h=0 \) and \(
\lambda^{2i}_1\ne w \). Thus \( A-B \) is equal to the right-hand side
of~\eqref{eq:syt=VT4}, which completes the proof.
\end{proof}

The limit case as $w\to \infty$ of the identities \eqref{eq:syt=VT1}
and~\eqref{eq:syt=VT2} is the main result in~\cite{Zeilberger_lazy},
while the limit case as $w\to\infty$ of~\eqref{eq:syt=VT3} appears
in \cite[Conjecture~1.2, proved in Theorem~1.3]{Eu2013}.
See Corollary~\ref{thm:SYT-VT} in Appendix~\ref{sec:growth} for a uniform bijective
treatment of both. 

\medskip
Vacillating tableaux are closely related to matchings. 
A \emph{(partial) matching} on 
\( \{1,2,\dots,n\} \) is a set
partition of 
\( \{1,2,\dots,n\} \) 
into blocks of size one or two. The element in a singleton block is called a \emph{fixed
  point} and a pair \( (i,j) \) of integers \( i<j \) that are contained in a
block of size two is called an \emph{arc}. Next we define various
kinds of crossings and nestings for matchings, and various sets of
matchings subject to restrictions on their crossings and nestings.

\begin{defn} \label{def:NCNN}
  Let \( k \) be a positive integer. A \emph{\( k \)-crossing} is a set of \( k
  \) arcs \( (i_{1},j_{1}), \dots, (i_{k},j_{k}) \) for which \(
  i_{1}<\dots<i_{k}<j_{1}<\dots<j_{k} \). A \emph{\( k \)-nesting} is a set of
  \( k \) arcs \( (i_{1},j_{1}), \dots, (i_{k},j_{k}) \) for which \(
  i_{1}<\dots<i_{k}<j_{k}<\dots<j_{1} \). We say that a matching is \emph{\( k
    \)-noncrossing} (respectively \emph{\( k \)-nonnesting}) if it does not have any \(
  k \)-crossing (respectively \( k \)-nesting).
\end{defn}

We denote the set of \( r \)-noncrossing and \( s
\)-nonnesting matchings on 
\( \{1,2,\dots,n\} \) 
by \( \NCNN_n(r,s) \).

\begin{defn}\label{defn:ncnn}
  Let \( k \) be a positive integer. A \emph{\( (k+1/2) \)-crossing} is a set of
  \( k \) arcs \( (i_1,j_1), \dots, (i_k,j_k) \) and a fixed point \( v \) for
 which \( i_{1}<\dots<i_{k}<v<j_{1}<\dots<j_{k} \). A \emph{\( (k+1/2)
    \)-nesting} is a set of \( k \) arcs \( (i_1,j_1), \dots, (i_k,j_k) \) and a
  fixed point \( v \) for which \( i_{1}<\dots<i_{k}<v<j_{k}<\dots<j_{1} \). We
  say that a matching is \emph{\( (k+1/2) \)-noncrossing} (respectively \emph{\(
    (k+1/2) \)-nonnesting}) if it does not have any \( t \)-crossing (respectively \( t
  \)-nesting) for \( t\ge k+1/2 \).
\end{defn}

Note that a matching is \( (k+1/2) \)-noncrossing if and only if it has neither
a \( (k+1/2) \)-crossing nor a \( (k+1) \)-crossing. A similar
remark holds for $(k+1/2)$-nonnesting matchings.

For an integer \( n \) and integers or half-integers \( r \) and \( s \), we
denote by \( \NCNN_n(r,s) \) the set of \( r \)-noncrossing and \( s
\)-nonnesting matchings on 
\( \{1,2,\dots,n\} \).

\begin{defn}\label{defn:ncnn'}
  Let \( \NCNN'_n(h+1,w+1) \) to be the set of matchings
  in \( \NCNN_n(h+1,w+1) \) in which
  every fixed point \( v \) satisfies one of the following conditions:
  \begin{itemize}
  \item \( v \) is not contained in any \( (h+1/2) \)-crossing and any \( (w+1/2)
    \)-nesting,
  \item \( v \) is contained in both an \( (h+1/2) \)-crossing
    and a \( (w+1/2) \)-nesting.
  \end{itemize}
  For \( M\in \NCNN'_n(h+1,w+1) \), let \( z(M) \) be the number of fixed
  points in \( M \) that are contained in both an \( (h+1/2) \)-crossing and a
  \( (w+1/2) \)-nesting.
\end{defn}

The following lemma connects $(h+1)$-noncrossing and
$(w+1)$-nonnesting matchings and vacillating tableaux in
\( \VT_n(h,w) \).

\begin{lem}\label{lem:ncnn-walk}
  Let $h$, $w$, and \( n \) be positive integers. There is a bijection \( \phi:
  \NCNN_n(h+1,w+1)\to \VT_n(h,w) \) such that, if \( \phi(M) = T = (\lambda^{0},
  \lambda^{1}, \dots, \lambda^{n}) \), then the following hold:
  \begin{itemize}
  \item \( i \) is a fixed point of \( M \) if and only if the \( i \)-th step
    of \( T \) is a zero step, i.e., \( \lambda^{i-1} = \lambda^{i} \),
  \item \( i \) is a fixed point of \( M \) contained in an \( (h+1/2) \)-crossing if
    and only if the \( i \)-th step of \( T \) is a zero step not on the
    hyperplane \( x_h=0 \), i.e., \( \lambda^{i-1} = \lambda^{i} \) and \(
    \lambda^{i}_h>0 \),
  \item \( i \) is a fixed point of \( M \) contained in a \( (w+1/2) \)-nesting if
    and only if the \( i \)-th step of \( T \) is a zero step on the
    hyperplane \( x_1=w \), i.e., \( \lambda^{i-1} = \lambda^{i} \) and \(
    \lambda^{i}_1=w \).
  \end{itemize}
\end{lem}
\begin{proof}
  Consider the bijection due to Chen et al.~\cite[Sec.~5]{Chen2007} between
  complete matchings and oscillating tableaux, which are vacillating tableaux
  without zero steps. We can extend this map to partial matchings and
  vacillating tableaux by doing nothing when we encounter a fixed
  point.
Alternatively, consider the growth diagram version of the same
bijection in~\cite[end of Sec.~3]{KratCE}, and treat fixed points as
indicated in~\cite[Fig.~6]{KratCE}.
  
  It is
  straightforward to check that this bijection satisfies the given conditions.
\end{proof}

\begin{example}
  Let \( n=3 \), \( h=1 \), and \( w=1 \). Then \( \NCNN_n(h+1,w+1/2)=
  \NCNN_n(h+1/2,w+1) \) has three elements, namely, \( \emptyset, \{(1,2)\},
  \{(2,3)\} \), where we only consider the arcs and omit the fixed points. Here,
  \( \{(1,3)\} \) is the only forbidden matching. The corresponding lattice
  paths of length 3 in \( \{1\ge x_1\ge 0\} \) are \( (0, 0, 0, 0) \), \( (0, 1,
  0, 0) \), and \( (0, 0, 1, 0) \). Note that \( (0, 1, 1, 0) \) is not allowed
  because a zero step is used when \( x_1=1 \).
\end{example}

The following proposition is an immediate consequence of
Lemma~\ref{lem:ncnn-walk}.

\begin{prop}\label{prop:NCNN=VT}
For positive integers $h$, $w$, and $n$, we have
  \begin{align}
    \label{eq:NCNN=VT1}
    |\NCNN_n(h+1,w+1)|  &=  |\VT_n(h,w)|,\\
    \label{eq:NCNN=VT2}
    |\NCNN_n(h+1,w+1/2)| &= |\VT_n(h,w^*)|, \\
    \label{eq:NCNN=VT3}
    |\NCNN_n(h+1/2,w+1)|&= |\VT_n(h^*,w)|, \\
    \label{eq:NCNN=VT4}
    \sum_{M\in \NCNN'_n(h+1,w+1)} (-1)^{z(M)} &= \sum_{T\in \VT'_n(h,w)} (-1)^{z(T)}.
  \end{align}
\end{prop}

By Corollary~\ref{cor:csyt=VT} and Proposition~\ref{prop:NCNN=VT}, we
are able to connect cylindric standard Young tableaux and matchings.

\begin{cor}\label{cor:syt=ncnn}
  For positive integers $h$, $w$, and \( n \), we have
  \begin{align}
    \label{eq:syt=ncnn1}
    |\CSYT_n(2h+1,2w+1)| &=  |\NCNN_n(h+1,w+1)|,\\
    \label{eq:syt=ncnn2}
    |\CSYT_n(2h+1,2w)| &=  |\NCNN_n(h+1,w+1/2)|,\\
    \label{eq:syt=ncnn3}
    |\CSYT_n(2h,2w+1)| &=   |\NCNN_n(h+1/2,w+1)|,\\
    \label{eq:syt=ncnn4}
    |\CSYT_n(2h,2w)| &=  \sum_{M\in \NCNN'_n(h+1,w+1)} (-1)^{z(M)}.
  \end{align}
\end{cor}

Clearly, the identities \eqref{eq:syt=ncnn1} and~\eqref{eq:syt=ncnn2}
reduce to~\eqref{eq:Chen2} for $w\to\infty$. The latter identity,
together with the limit case as
$w\to\infty$ of~\eqref{eq:syt=ncnn3} are discussed in Appendix~\ref{sec:growth} from
a bijective point of view; see~Corollary~\ref{thm:SYT-NC}.

In the next section we will show that~\eqref{eq:syt=ncnn1}
and~\eqref{eq:syt=ncnn2} reduce to a result of Mortimer and
Prellberg~\cite{Mortimer2015} if \( h=1 \). We will also give a
bijective proof 
of~\eqref{eq:syt=ncnn3} and~\eqref{eq:syt=ncnn4} for \( h=1 \).
Finding a bijective
proof for the general case is an open problem.

\begin{problem}
  Find a bijective proof of Corollary~\ref{cor:syt=ncnn} for general \( h \) and
  \( w \).
\end{problem}

The results of Corollary~\ref{cor:syt=ncnn} --- or rather the results
missing there --- raise two further questions.

\begin{problem}
Is there a ``signless" relation between the cylindric standard Young
tableaux in $\CSYT(2h,2w)$ and matchings with restrictions on their
crossings and nestings?
\end{problem}

\begin{problem}
Are the matchings in $\NCNN(h+1/2,w+1/2)$ related to cylindric
standard Young tableaux?
\end{problem}

\begin{problem}
Find an explicit formula for the number of elements in
$\NCNN(h+1/2,w+1/2)$.\end{problem}

\section{Mortimer and Prellberg's result and bijective proofs}
\label{sec:h=1}

Here we focus our attention
on the special case of
Corollary~\ref{cor:syt=ncnn} where \( h=1 \).
We show that if \( h=1 \) then the first two
identities, \eqref{eq:syt=ncnn1} and~\eqref{eq:syt=ncnn2}, are equivalent to a
result of Mortimer and Prellberg~\cite{Mortimer2015}, for which a
bijective proof has
been given by Courtiel, Elvey Price and
Marcovici~\cite{Courtiel2021}. We then give a 
bijective proof of the last two identities, \eqref{eq:syt=ncnn3}
and~\eqref{eq:syt=ncnn4}, for \( h=1 \).

\medskip
The above mentioned result of Mortimer and Prellberg involves walks in
a triangle respectively bounded Motzkin paths, which we define next.

\begin{defn}\label{defn:T_n}
  We let \( T_n(m) \) denote the set of walks \( p \) of length \( n \) from \(
  (0,0) \) to any point in \( \Z^2 \) consisting of steps in \( \{ (1,0),
  (0,1), (-1,-1)\} \) with the property that \( p \) is contained in the triangular region \(
  \{(x_1,x_2)\in \RR^2: m\ge x_1\ge x_2\ge 0 \} \).
\end{defn}

\begin{defn}\label{defn:mot}
  A \emph{Motzkin path of length \( n \)} is a path from \( (0,0) \) to \( (n,0)
  \) consisting of \emph{up steps} \( (1,1) \), \emph{down steps} \( (1,-1) \),
  and \emph{horizontal steps} \( (1,0) \) that never goes below the \( x \)-axis.
  The \emph{height} of a Motzkin path is the largest \( y \)-coordinate of a point
  in it. We denote by \( \Mot_n(m) \) the set of Motzkin paths of length \( n \)
  with height at most \( m \).
\end{defn}

We consider the following three subsets of \( \Mot_n(w) \):
 \begin{align*}
    \Mot'_n(w)
    & = \{ p\in \Mot_n(w):
      \mbox{\( p \) has no horizontal step on the line \( x_2=w \)}\},\\
    \Mot^1_n(w)
    & = \{ p\in \Mot_n(w):
      \mbox{every horizontal step of \( p \) lies on the line \( x_2=0 \)}\},\\
    \Mot^2_n(w)
    & = \{ p\in \Mot_n(w):
      \mbox{every horizontal step of \( p \) lies on the lines \( x_2=0 \) and \( x_2=w \)}\}.
 \end{align*}

 A Motzkin path without horizontal steps is called a \emph{Dyck path}. A
 \emph{Dyck prefix} is a sequence of points that can be extended to a Dyck path.
 Let \( \DP_n(w) \) denote the set of Dyck prefixes of length \( n \) contained
 in the region \( \{(x_1,x_2)\in \RR^2: 0\le x_2\le w \} \). We now give an
 equivalent statement of Corollary~\ref{cor:syt=ncnn} with \( h=1 \).

\begin{thm}\label{thm:syt=ncnn_h=1}
  For positive integers \( w \) and \( n \),
  \begin{align}
    \label{eq:t=Mot1}
    |T_n(2w+1)| &= |\Mot_n(w)|,\\
    \label{eq:t=Mot2}
    |T_n(2w)| &= |\Mot'_n(w)|,\\
    \label{eq:DP=Mot1}
    |\DP_n(2w+1)| &= |\Mot_n^1(w)|,\\
    \label{eq:DP=Mot2}
    |\DP_n(2w)| &= \sum_{p\in\Mot_n^2(w)} (-1)^{k(p)},
  \end{align}
  where \( k(p) \) is the number of horizontal steps of \( p \) on the line \(
  x_2=w \).
\end{thm}
\begin{proof}
  First, we give a bijection between \( \CSYT_n(3,w) \) and \( T_n(w) \).
  Let \( T\in \CSYT_n(3,w) \). Then the corresponding path \(
  p=(p_0,\dots,p_n)\in T_n(w) \) is constructed as follows.
  Let \( p_0=(0,0) \)
  and
  define \( p_i \) by
  \[
    p_i = \begin{cases}
      p_{i-1}+(1,0), & \mbox{if \( i \) is in the first row of \( T \)},\\
      p_{i-1}+(0,1), & \mbox{if \( i \) is in the second row of \( T \)},\\
      p_{i-1}+(-1,-1), & \mbox{if \( i \) is in the third row of \( T \)},
    \end{cases}
  \]
  for \( i=1,2,\dots, n \). 
  It is easy to check that the map \( T\mapsto p \) is a bijection between \( \CSYT_n(3,w) \) and \( T_n(w) \).

\medskip
  Second, we give a bijection between \( \CSYT_n(2,w) \) and \( \DP_n(w) \) in a
  similar way. For \( T\in \CSYT_n(2,w) \), the corresponding path \(
  p=(p_0,\dots,p_n)\in \DP_n(w) \) is defined by \( p_0=(0,0) \) and
  \[
    p_{i} = \begin{cases}
      p_{i-1}+(1,1), & \mbox{if \( i \) is in the first row of \( T \)},\\
      p_{i-1}+(1,-1), &\mbox{if \( i \) is in the second row of \( T \)},
    \end{cases}
  \]
  for 
  \( i=1,2,\dots, n \).
  It is also easy to check that the map \(
  T\mapsto p \) is a bijection between \( \CSYT_n(2,w) \) and \( \DP_n(w) \).

\medskip
  Third, we find a bijection between \( \NCNN_n(2,w+1) \) and \(
  \Mot_n(w) \). This 
  is in fact a well-known bijection between noncrossing matchings
  and Motzkin paths.
  Let \( M\in \NCNN_n(2,w+1) \). Recall that \( i \) is a fixed point of \( M \)
  if \( i \) is not connected to any integer by an arc. An integer \( i \)
  is called an \emph{opener} (respectively \emph{closer}) of \( M \) if it is connected to
  \( j \) by an arc for some \( j>i \) (respectively \( j<i \)). Since \( M \) does
  not have any 2-crossings, it is determined by its openers, closers, and fixed
  points. Therefore we can construct the corresponding Motzkin path \( p=(p_0,\dots,p_n) \)
  as follows. Let \( p_0=(0,0) \) and 
  define
  \[
    p_i = \begin{cases}
      p_{i-1}+(1,1), & \mbox{if \( i \) is an opener of  \( M \)},\\
      p_{i-1}+(1,-1), & \mbox{if \( i \) is a closer of \( M \)},\\
      p_{i-1}+(1,0), & \mbox{if \( i \) is a fixed point of \( M \)},
    \end{cases}
  \]
  for \( i=1,2,\dots, n\).
  Since \( M \) does not have a \( (w+1) \)-nesting, the Motzkin path \( p \)
  stays in the region \( \{(x_1,x_2)\in \RR^2: 0\le x_2\le w \} \). One can easily check that the map \(
  M\mapsto p \) is a bijection between \( \NCNN_n(2,w+1) \) and \( \Mot_n(w) \).

  Similarly one can check that the same map \( M\mapsto p \) also induces a
  bijection between \( \NCNN_n(2,w+1/2) \) and \( \Mot'_n(w) \),
  a bijection between \(
  \NCNN_n(1+1/2,w+1) \) and \( \Mot^1_n(w) \), and a bijection between \(
  \NCNN'_n(2,w+1) \) and \( \Mot^2_n(w) \). Moreover, if \(M\in \NCNN'_n(2,w+1)
  \) corresponds to \( p\in \Mot^2_n(w) \), then \( z(M)=k(p) \).

  Applying the above bijections to Corollary~\ref{cor:syt=ncnn} with \( h=1 \), we
  immediately obtain the desired identities.
\end{proof}

The identities~\eqref{eq:t=Mot1} and~\eqref{eq:t=Mot2} were first proved by
Mortimer and Prellberg~\cite{Mortimer2015} using the kernel method. Recently,
Courtiel, Elvey Price and Marcovici~\cite{Courtiel2021} found a bijective proof of
these two identities. In the rest of this section, we give
bijective proofs
of~\eqref{eq:DP=Mot1} and of~\eqref{eq:DP=Mot2}. To simplify the right-hand side
of~\eqref{eq:DP=Mot2}, we need the following definition and lemma.

\begin{defn}
For \( p=(p_0,p_1,\dots,p_n) \in \Mot_n^2(w) \), a horizontal step \( p_{j+1}-p_{j} \) is called \emph{special} if one of the following conditions holds:
\begin{itemize}
\item \( p_{j+1}-p_{j} \) lies on the line \( x_2=w \);
\item \( p_{j+1}-p_{j} \) lies on the line \(x_2=0 \), \( (p_0,p_1,\dots,p_j)
  \in \Mot^1_{j}(w) \) with an even number of horizontal steps, and there is an
  integer \( i \) with  \( j < i < n\) such that \( (p_{j+1},p_{j+2},\dots,
  p_{i+1}) \) is a (translated) Dyck prefix from \((j+1,0)\) to \( (i+1,w) \).
\end{itemize}
We define
\[
 \Mot^3_n(w)= \{ p\in \Mot^2_n(w):\mbox{\( p \) has no special horizontal step}\},
\]
so that \( \Mot^3_n(w) \subseteq \Mot^1_n(w) \).
\end{defn}

\begin{lem} \label{lem:special}
 For positive integers \( w \) and \( n \),
 \[
  |\Mot^3_n(w)| = \sum_{p\in\Mot_n^2(w)} (-1)^{k(p)},
 \]
 where \( k(p) \) is the number of horizontal steps of \( p \) on the line \(
 x_2=w \).
\end{lem}

\begin{proof}
  For \( p \in \Mot^2_n(w) \), we define \( \sgn(p)=(-1)^{k(p)} \). It suffices to find a sign-reversing involution \( \phi \) on  \( \Mot^2_n(w) \) whose fixed points are exactly those in \( \Mot^3_n(w) \).

Let \(p \in \Mot^2_n(w) \). If \( p \in \Mot^3_n(w) \), then define \(
\phi(p)=p \). Otherwise, we can find the smallest integer \( j \) with
\(0 \le j < n \) such that \( p_{j+1}-p_{j} \) is special. Suppose
that the first special horizontal step \( p_{j+1}-p_j \) lies on the
line \( x_2=w \). Then we can find the largest integer \( i\) with \(
0 \le i < j \) such that \( (p_i, p_{i+1}, \dots, p_{j}) \) is a
(translated) Dyck prefix. Define \( \phi(p)=(p_0, \dots, p_i, r(p_j),
r(p_{j-1}),\dots, r(p_{i+1}), p_{j+1}, \dots, p_n) \), where \( r \)
is the reflection about the point \( ((j-i+1)/2, w/2) \). Then \(
q=\phi(p) \) has the first special horizontal step \( q_{i+1}-q_i \)
on the line \( x_2=0 \) and \(k(q)=k(p)-1 \). Now suppose that \( p \)
has the first special horizontal step \( p_{j+1}-p_j \) on the line \(
x_2=0 \). Then \( (p_0,p_1,\dots,p_j) \in \Mot^1_{j}(w) \) with an even number of horizontal steps, and we can find the smallest integer \( i \) with  \( j < i < n\) such that \( (p_{j+1},p_{j+2},\dots, p_{i+1}) \) is a translated Dyck prefix from \((j+1,0)\) to \( (i+1,w) \). Define \( \phi(p)=(p_0, \dots, p_j, r(p_{i}), r(p_{i-1}),\dots, r(p_{j+1}), p_{i+1}, \dots, p_n) \), where \( r \) is the reflection about the point \( ((i-j+1)/2, w/2) \). In this case, \( q=\phi(p) \) has the first special horizontal step \( q_{i+1}-q_i \) on the line \( x_2=w \) and \(k(q)=k(p)+1 \). One can easily check that \( \phi \) is an involution, which completes the proof.
\end{proof}

Figure \ref{fig:special involution} shows an example of the involution \( \phi
\) defined in the proof of Lemma~\ref{lem:special}.

\begin{figure}
\centering
\small{
\begin{tikzpicture}[scale=0.4]

\foreach \i in {0,1,...,26}
\foreach \j in {0,1,...,3}
\filldraw[fill=gray!70, color=gray!70] (\i,\j) circle (2.5pt);

\draw[->] (-1,0) -- (26.5,0);
\draw[->] (0,-1) -- (0,3.5);

\node at (27,0) {\( x_1 \)};
\node at (0,4) {\( x_2 \)};

\foreach \i in {5,10,15,20,25}
\node[below] at (\i,0) {\i};

\node[left] at (0,3) {\( 3 \)};

\coordinate (0) at (0,0);
\coordinate (1) at (1,0);
\coordinate (2) at (2,1);
\coordinate (3) at (3,0);
\coordinate (4) at (4,0);
\coordinate (5) at (5,0);
\coordinate (6) at (6,1);
\coordinate (7) at (7,0);
\coordinate (8) at (8,1);
\coordinate (9) at (9,2);
\coordinate (10) at (10,3);
\coordinate (11) at (11,2);
\coordinate (12) at (12,1);
\coordinate (13) at (13,0);
\coordinate (14) at (14,0);
\coordinate (15) at (15,1);
\coordinate (16) at (16,0);
\coordinate (17) at (17,1);
\coordinate (18) at (18,2);
\coordinate (19) at (19,3);
\coordinate (20) at (20,2);
\coordinate (21) at (21,3);
\coordinate (22) at (22,3);
\coordinate (23) at (23,3);
\coordinate (24) at (24,2);
\coordinate (25) at (25,1);
\coordinate (26) at (26,0);

\foreach \i in {0,1,...,26}
\filldraw (\i) circle (2.5pt);

\draw[thick]
(0) -- (1) -- (2) -- (3) -- (4) -- (5) -- (6) -- (7) -- (8) -- (9)
-- (10) -- (11) -- (12) -- (13) -- (14) -- (15) -- (16) -- (17) -- (18) --(19)
-- (20) -- (21) -- (22) -- (23) -- (24) -- (25) -- (26);

\draw[ultra thick]
(4) -- (5)
(21) -- (22)
(22) -- (23);

\end{tikzpicture}\\[2ex]

\begin{tikzpicture}[scale=0.4]

\draw[<->, thick]
(0,0) -- (0,-1);

\end{tikzpicture}\\

\begin{tikzpicture}[scale=0.4]

\foreach \i in {0,1,...,26}
\foreach \j in {0,1,...,3}
\filldraw[fill=gray!70, color=gray!70] (\i,\j) circle (2.5pt);

\draw[->] (-1,0) -- (26.5,0);
\draw[->] (0,-1) -- (0,3.5);

\node at (27,0) {\( x_1 \)};
\node at (0,4) {\( x_2 \)};

\foreach \i in {5,10,15,20,25}
\node[below] at (\i,0) {\i};

\node[left] at (0,3) {\( 3 \)};

\coordinate (0) at (0,0);
\coordinate (1) at (1,0);
\coordinate (2) at (2,1);
\coordinate (3) at (3,0);
\coordinate (4) at (4,0);
\coordinate (5) at (5,1);
\coordinate (6) at (6,2);
\coordinate (7) at (7,3);
\coordinate (8) at (8,2);
\coordinate (9) at (9,3);
\coordinate (10) at (10,3);
\coordinate (11) at (11,2);
\coordinate (12) at (12,1);
\coordinate (13) at (13,0);
\coordinate (14) at (14,0);
\coordinate (15) at (15,1);
\coordinate (16) at (16,0);
\coordinate (17) at (17,1);
\coordinate (18) at (18,2);
\coordinate (19) at (19,3);
\coordinate (20) at (20,2);
\coordinate (21) at (21,3);
\coordinate (22) at (22,3);
\coordinate (23) at (23,3);
\coordinate (24) at (24,2);
\coordinate (25) at (25,1);
\coordinate (26) at (26,0);

\foreach \i in {0,1,...,26}
\filldraw (\i) circle (2.5pt);

\draw[thick]
(0) -- (1) -- (2) -- (3) -- (4) -- (5) -- (6) -- (7) -- (8) -- (9)
-- (10) -- (11) -- (12) -- (13) -- (14) -- (15) -- (16) -- (17) -- (18) --(19)
-- (20) -- (21) -- (22) -- (23) -- (24) -- (25) -- (26);

\draw[ultra thick]
(9) -- (10)
(21) -- (22)
(22) -- (23);

\end{tikzpicture}
}

\caption{An example of the involution \(\phi \) on \( \Mot^2_{26}(3) \).
  The thick horizontal lines indicate the special horizontal steps.}\label{fig:special involution}
\end{figure}
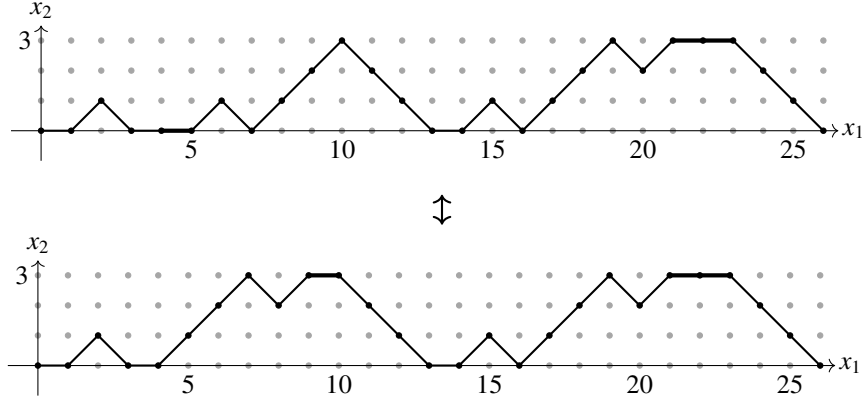

By Lemma \ref{lem:special}, the identity~\eqref{eq:DP=Mot2} is equivalent to \(
|\DP_n(2w)|=|\Mot^3_n(w)| \). To prove~\eqref{eq:DP=Mot1}
and~\eqref{eq:DP=Mot2}, we adopt a bijection between the set of bounded Dyck prefixes and the set of bounded up-down paths due to Dershowitz.

An \emph{up-down path of length \( n \)} is a path from \( (0,0) \) to
\( (n, (1+(-1)^{n+1})/2) \) consisting of upward diagonal steps $(1,1)$
and downward diagonal steps $(1,-1)$.
Let \( \GD_n(w) \) denote the set of up-down paths of length \( n \) staying in the region
  \[
   \left \{ (x_1,x_2)\in \RR^2: -\left \lfloor \frac{w}{2} \right \rfloor \le x_2\le \left \lfloor \frac{w+1}{2} \right \rfloor \right \}.
  \]
A standard reflection argument (cf.\ \cite[Th.~10.3.3]{KratCL}
under rotation by $45^\circ$) yields
\[
  |\DP_n(w)|=|\GD_n(w)|=\sum_{j \in \Z} (-1)^j
  \binom{n}{\left\lfloor \frac{n+(w+2)j}{2} \right\rfloor}.
\]
Recently, Gu and Prodinger~\cite{Gu2021} and Dershowitz~\cite{Dershowitz} found
simple bijections between \( \DP_n(w) \) and \( \GD_n(w) \) independently. Here,
we briefly introduce Dershowitz's bijection.

Let \( p \) be an up-down path from \( (0,0) \) to \( (n,m) \).
For an integer \( k \), the \emph{\(
  TA \) representation of \( p \) relative to \( k \)} is the word \(
r_1r_2\dots r_n \), where \( r_i \) is the letter \( T \) (respectively \( A \)) if the
\( i \)-th step of \( p \) moves towards (respectively away from) the line \( x_2=k \).
Note that any up-down path (whose starting point is \( (0,0) \)) is uniquely
determined by its TA representation. Now let \( p \in \DP_n(w) \). Suppose that
\( r_1r_2\dots r_n \) is the \( TA \) representation of \( p \) relative to \(
\lceil h(p)/2 \rceil -1/2 \), where \( h(p) \) is the largest \( x_2 \)-coordinate
of the points in \( p \). Let \( j \) be the smallest integer such that the
height of the ending point of the \( j \)-th step of \( p \) equals \( \lfloor
h/2 \rfloor \). Then define \( \phi(p) \) to be the up-down path  whose \( TA
\) representation relative to \( 1/2 \) is \( r_{j+1} r_{j+2}\dots r_n r_j
r_{j-1}\dots r_1\). Dershowitz~\cite{Dershowitz} showed that \( \phi : \DP_n(w)
\rightarrow \GD_n(w) \) is a bijection.

Our last ingredient is a map \(\psi : \GD_n(w) \rightarrow \Mot^1_n(\lfloor w/2 \rfloor) \). For
\( p=(p_0, p_1, \dots, p_n) \in \GD_n(w) \), define \(
\psi(p)=(q_0,q_1,\dots,q_n) \), where
\[
q_i=
\begin{cases}
p_i+(0,-1), & \mbox{if \( p_i \) is above the line \( x_2=1/2 \)},\\
-p_i, & \mbox{otherwise},
\end{cases}
\]
for \( i=0,1,\dots, n \). One can easily see that \(\psi : \GD_n(2w+1)
\rightarrow \Mot^1_n(w) \) and \(\psi : \GD_n(2w) \rightarrow \Mot^3_n(w) \) are
bijections. The combination of
the maps \( \phi \) and \( \psi \)  completes the proof
of~\eqref{eq:DP=Mot1} and~\eqref{eq:DP=Mot2}. Figure \ref{fig:(1.5)l involution}
shows examples of the maps \(\phi : \DP_{26}(6) \rightarrow \GD_{26}(6) \) and
\(\psi : \GD_{26}(6) \rightarrow \Mot^3_{26}(3) \).

\begin{figure}
\centering
\begin{tikzpicture}[scale=0.4]

\foreach \i in {0,1,...,26}
\foreach \j in {0,1,...,6}
\filldraw[fill=gray!70, color=gray!70] (\i,\j) circle (2.5pt);

\draw[->] (-1,0) -- (26.5,0);
\draw[->] (0,-1) -- (0,6.5);
\draw[] (0,2.5) -- (26,2.5);

\node at (27,0) {\( x_1 \)};
\node at (0,7) {\( x_2 \)};

\foreach \i in {5,10,15,20,25}
\node[below] at (\i,0) {\i};

\node[left] at (0,3) {\( 3 \)};
\node[left] at (0,6) {\( 6 \)};
\node[left] at (-1,0) {};

\coordinate (0) at (0,0);
\coordinate (1) at (1,1);
\coordinate (2) at (2,2);
\coordinate (3) at (3,1);
\coordinate (4) at (4,0);
\coordinate (5) at (5,1);
\coordinate (6) at (6,2);
\coordinate (7) at (7,3);
\coordinate (8) at (8,2);
\coordinate (9) at (9,3);
\coordinate (10) at (10,4);
\coordinate (11) at (11,5);
\coordinate (12) at (12,6);
\coordinate (13) at (13,5);
\coordinate (14) at (14,4);
\coordinate (15) at (15,3);
\coordinate (16) at (16,2);
\coordinate (17) at (17,3);
\coordinate (18) at (18,4);
\coordinate (19) at (19,3);
\coordinate (20) at (20,2);
\coordinate (21) at (21,1);
\coordinate (22) at (22,0);
\coordinate (23) at (23,1);
\coordinate (24) at (24,2);
\coordinate (25) at (25,1);
\coordinate (26) at (26,2);

\foreach \i in {0,1,...,26}
\filldraw (\i) circle (2.5pt);

\filldraw (7) circle (5pt);

\draw[thick]
(0) -- (1) -- (2) -- (3) -- (4) -- (5) -- (6) -- (7) -- (8) -- (9)
-- (10) -- (11) -- (12) -- (13) -- (14) -- (15) -- (16) -- (17) -- (18) --(19)
-- (20) -- (21) -- (22) -- (23) -- (24) -- (25) -- (26);

\foreach \i in {1,2,5,6,8,7,9,13,14,15,16,17,19,20,23,24,26}
\node[left] at (\i) {\( T \)};

\foreach \i in {3,4,10,11,12,18,21,22,25}
\node[left] at (\i) {\( A \)};

\end{tikzpicture}\\[2ex]

\begin{tikzpicture}[scale=0.4]

\draw[->, thick]
(0,0) -- (0,-1);

\node at (1, -0.5) {\( \phi \)};

\end{tikzpicture}\\

\begin{tikzpicture}[scale=0.4]

\foreach \i in {0,1,...,26}
\foreach \j in {-3,...,3}
\filldraw[fill=gray!70, color=gray!70] (\i,\j) circle (2.5pt);

\draw[->] (-1,0) -- (26.5,0);
\draw[->] (0,-3.5) -- (0,3.5);
\draw[] (0,0.5) -- (26,0.5);

\node at (27,0) {\( x_1 \)};
\node at (0,4) {\( x_2 \)};

\node[left] at (0,3) {\( 3 \)};
\node[left] at (0,-3) {\( -3 \)};
\node[left] at (-1,0) {};

\coordinate (0) at (0,0);
\coordinate (1) at (1,1);
\coordinate (2) at (2,0);
\coordinate (3) at (3,-1);
\coordinate (4) at (4,-2);
\coordinate (5) at (5,-3);
\coordinate (6) at (6,-2);
\coordinate (7) at (7,-1);
\coordinate (8) at (8,0);
\coordinate (9) at (9,1);
\coordinate (10) at (10,0);
\coordinate (11) at (11,-1);
\coordinate (12) at (12,0);
\coordinate (13) at (13,1);
\coordinate (14) at (14,2);
\coordinate (15) at (15,3);
\coordinate (16) at (16,2);
\coordinate (17) at (17,1);
\coordinate (18) at (18,2);
\coordinate (19) at (19,1);
\coordinate (20) at (20,0);
\coordinate (21) at (21,1);
\coordinate (22) at (22,0);
\coordinate (23) at (23,-1);
\coordinate (24) at (24,-2);
\coordinate (25) at (25,-1);
\coordinate (26) at (26,0);

\foreach \i in {0,1,...,26}
\filldraw (\i) circle (2.5pt);

\filldraw (19) circle (5pt);

\draw[thick]
(0) -- (1) -- (2) -- (3) -- (4) -- (5) -- (6) -- (7) -- (8) -- (9)
-- (10) -- (11) -- (12) -- (13) -- (14) -- (15) -- (16) -- (17) -- (18) --(19)
-- (20) -- (21) -- (22) -- (23) -- (24) -- (25) -- (26);

\foreach \i in {1,2,6,7,8,9,10,12,13,16,17,19,20,21,22,25,26}
\node[left] at (\i) {\( T \)};

\foreach \i in {3,4,5,11,14,15,18,23,24}
\node[left] at (\i) {\( A \)};

\end{tikzpicture}\\[2ex]

\begin{tikzpicture}[scale=0.4]

\draw[->, thick]
(0,0) -- (0,-1);
\node at (1, -0.5) {\( \psi \)};

\end{tikzpicture}\\

\begin{tikzpicture}[scale=0.4]

\foreach \i in {0,1,...,26}
\foreach \j in {0,1,...,3}
\filldraw[fill=gray!70, color=gray!70] (\i,\j) circle (2.5pt);

\draw[->] (-1,0) -- (26.5,0);
\draw[->] (0,-1) -- (0,3.5);

\node at (27,0) {\( x_1 \)};
\node at (0,4) {\( x_2 \)};

\foreach \i in {5,10,15,20,25}
\node[below] at (\i,0) {\i};

\node[left] at (0,3) {\( 3 \)};
\node[left] at (-1,0) {};

\coordinate (0) at (0,0);
\coordinate (1) at (1,0);
\coordinate (2) at (2,0);
\coordinate (3) at (3,1);
\coordinate (4) at (4,2);
\coordinate (5) at (5,3);
\coordinate (6) at (6,2);
\coordinate (7) at (7,1);
\coordinate (8) at (8,0);
\coordinate (9) at (9,0);
\coordinate (10) at (10,0);
\coordinate (11) at (11,1);
\coordinate (12) at (12,0);
\coordinate (13) at (13,0);
\coordinate (14) at (14,1);
\coordinate (15) at (15,2);
\coordinate (16) at (16,1);
\coordinate (17) at (17,0);
\coordinate (18) at (18,1);
\coordinate (19) at (19,0);
\coordinate (20) at (20,0);
\coordinate (21) at (21,0);
\coordinate (22) at (22,0);
\coordinate (23) at (23,1);
\coordinate (24) at (24,2);
\coordinate (25) at (25,1);
\coordinate (26) at (26,0);

\foreach \i in {0,1,...,26}
\filldraw (\i) circle (2.5pt);

\draw[thick]
(0) -- (1) -- (2) -- (3) -- (4) -- (5) -- (6) -- (7) -- (8) -- (9)
-- (10) -- (11) -- (12) -- (13) -- (14) -- (15) -- (16) -- (17) -- (18) --(19)
-- (20) -- (21) -- (22) -- (23) -- (24) -- (25) -- (26);
\end{tikzpicture}
\caption{An example of the bijection between  \(\ \DP_{26}(6)\) and \(\Mot^3_{26}(3) \).}\label{fig:(1.5)l involution}
\end{figure}
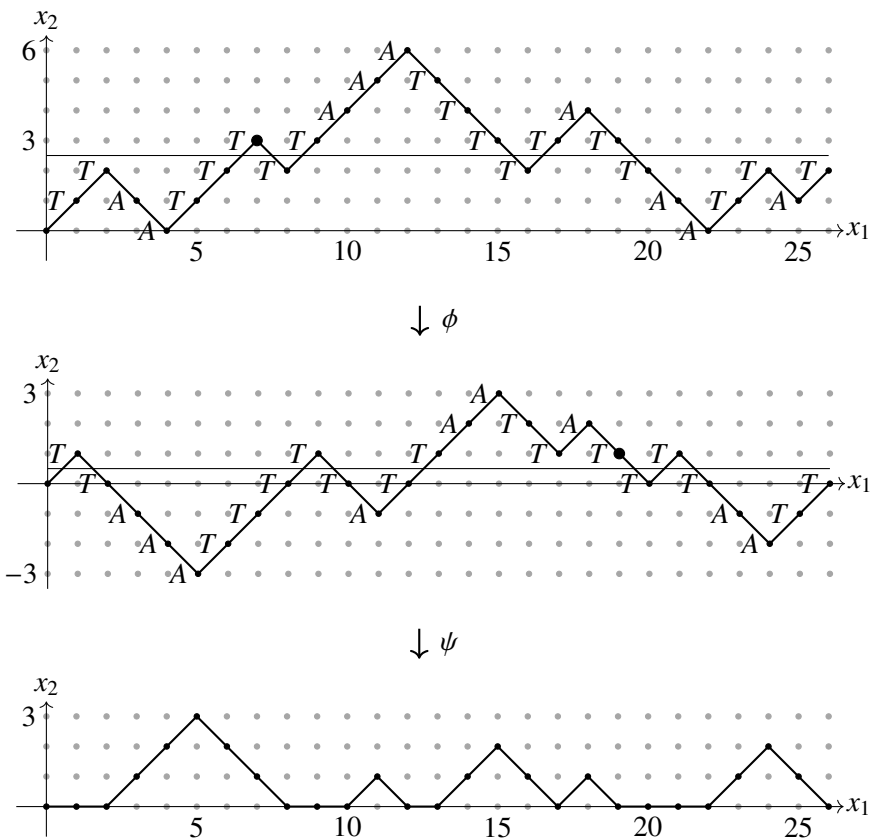

\appendix
\section{The discovery of the bounded Littlewood identities for cylindric Schur functions}
\label{sec:motivation}

Since we believe that the path of discovery of the new bounded Littlewood identities
in Theorem~\ref{thm:affine_BK_intro2} (restated
in Theorem~\ref{thm:affine_BK_intro}) is quite interesting (and
unexpected), and therefore should not be kept a secret from the reader,
we dedicate this appendix to a  description of our line of thought that ---
in the end --- led us to the identities in~\eqref{eq:aGBK1s}
and~\eqref{eq:aGBK2s}, and the further identities in
Theorems~\ref{thm:C} and~\ref{thm:D}.

\medskip
Our original goal was to find a bijective proof of the enumeration
result of Mortimer and 
Prellberg~\cite{Mortimer2015} that we discussed in 
Section~\ref{sec:h=1}, and which these authors established by some generating
function calculus based on the kernel method.
In order to state their result, we must recall 
from Definition~\ref{defn:T_n} that $T_n(m)$
denotes the set of walks of length~$n$ starting at the origin,
consisting of steps $(1,0)$, $(0,1)$, and $(-1,-1)$, and staying in
the triangular region $\{(x_1,x_2)\in \RR^2: m\ge x_1\ge x_2\ge 0 \}$,
and from Definition~\ref{defn:mot} that $\Mot_n(h)$ is the set of
Motzkin paths of length~$n$ of height at most~$h$. 

Mortimer and Prellberg~\cite{Mortimer2015} proved the following
surprising identity using the kernel method:
\begin{equation}
  \label{eq:MP}
  |T_n(2h+1)| = |\Mot_n(h)|.
\end{equation}

As given away above, we wanted to construct a bijective proof of this intriguing
identity. We first observed that~\eqref{eq:MP} is not the most transparent formulation
of this identity. In order to motivate our reformulation below, let us first
consider the limiting case where \( h\to\infty \), i.e.,
\begin{equation}
  \label{eq:MP_infty}
  |T_n(\infty)| = |\Mot_n(\infty)|.
\end{equation}
Each walk \( p=(p_0,\dots,p_n)\in T_n(\infty) \) can be identified with the sequence
\( (\lambda^{0},\dots, \lambda^{n}) \) of partitions with
at most three rows, where \( \lambda^{0}=\emptyset \) and \( \lambda^{i}
\) is obtained from \( \lambda^{i-1} \) by adding a cell to the first
(respectively second and third) row if \( p_i \) is \( (1,0) \) (respectively \( (0,1) \) and
\( (-1,-1) \)). This sequence of partitions is in turn naturally identified with
a standard Young tableau of size~\( n \) with height at most \( 3 \);
see the proof of Theorem~\ref{thm:syt=ncnn_h=1} for more details. On the
other hand, there is a simple and well-known bijection between \( \Mot_n(\infty)
\) and the set of \( 2 \)-noncrossing matchings on 
\( \{1,2,\dots, n \} \). 
(We described it in the proof of Theorem~\ref{thm:syt=ncnn_h=1}.)
Therefore \eqref{eq:MP_infty} is equivalent to
\[
 |\SYT_n(3)| = |\NC_n(2)|,
\]
where as before \( \SYT_n(3) \) is the set of standard Young tableaux
of size~\( n \) with 
at most \( 3 \)~rows and \( \NC_n(2) \) is the set of \( 2 \)-noncrossing
matchings on 
\( \{1,2,\dots, n \} \). 

It is not hard to check that the bijection between \( T_n(\infty) \)
and \( \SYT_n(3) \) induces a bijection between\break \( T_n(2w+1) \)
and \( \CSYT_n(3,2w+1) \).\footnote{Elizalde~\cite{Elizalde_cylindric}
  also found this bijection independently.} On the other hand, the
bijection between Motzkin paths and \( 2 \)-noncrossing matchings also
implies \( |\Mot_n(w)|=|\NCNN_n(2,w+1)| \). Therefore we can restate
Mortimer and Prellberg's result \eqref{eq:MP} as follows:
\begin{equation*}
  \label{eq:MP2}
  |\CSYT_n(3,2w+1)| =  |\NCNN_n(2,w+1)|.
\end{equation*}
One may now speculate that this identity holds in greater generality,
namely that we have
\begin{equation}\label{eq:MP_gen}
  |\CSYT_n(2h+1,2w+1)| =  |\NCNN_n(h+1,w+1)|
\end{equation}
(which was stated as \eqref{eq:SYT-NCNN} in the introduction, and
proved in Corollary~\ref{cor:syt=ncnn}).
Indeed, computer experiments confirmed the  truth of~\eqref{eq:MP_gen}.

The question now was how   to prove this more general identity.
Obviously, the most desirable proof would consist in finding a bijection
between   $\CSYT_n(2h+1,2w+1)$ and $\NCNN_n(h+1,w+1)$.
In lack of a good idea, we went instead for a computational proof based on
explicit formulas. Indeed, as we explained in
Proposition~\ref{prop:cssyt=path}, the tableaux in 
$\CSYT_n(2h+1,2w+1)$ are in bijection with
lattice paths in $\Z^{2h+1}$ starting at the origin, consisting of
positive unit steps in coordinate directions, and staying in
the region
$$
\big\{(x_1,x_2,\dots,x_{2h+1}):x_1\ge x_2\ge \dots\ge x_{2h+1}\ge 
x_1-(2w+1)\big\}.
$$
If the end point is also given, say $(\la_1,\la_2,\dots,\la_{2h+1})$,
then there is a formula which gives the number of such paths due
to Filaseta~\cite{Filaseta1985} (which turned out to be a special case
of the more general random-walks-in-Weyl-chambers formula of
Gessel and Zeilberger~\cite{Gessel1992}). Using this formula, we get
the following formula for the number of our tableaux:
\begin{equation}
|\CSYT_n(2h+1,2w+1)|=
\underset {\la_1-\la_{2h+1}\le 2w+1}{\sum_{|\la|= n}}
\underset{k_1+\dots+k_{2h+1}=0} {\sum_{k_1,\dots,k_{2h+1}\in\Z}}
n!\det_{1\le i,j\le 2h+1}
\left(\frac {1} {(\la_i-i+j+(2h+2w+2)k_i)!}\right).
\label{eq:A1}
\end{equation}

On the other hand, by a result of Chen et al.~\cite[Sec.~3]{Chen2007}
(see~\cite{KratCE} for a more transparent presentation),
the matchings in $\NCNN_n(h+1,w+1)$
are in bijection with vacillating tableaux
$\emptyset=\rh_0,\rh_1,\dots,\rh_{n-1},\rh_n=\emptyset$, where each
$\rh_i$ has at most $h$ rows and at most 
$w$ columns.
(Here, ``vacillating" does not have the same meaning as
in~\cite{Chen2007}, but rather means
that two successive partitions in the above sequence differ by at most
one cell; one obtains these sequences from those of
\cite[Sec.~3]{Chen2007} by ignoring the partitions in odd positions
and just keeping those in even positions; no information is lost since
it is only the special case of partial matchings that we are
interested in.) In turn, these vacillating tableaux
can be seen as lattice paths in $\Z^h$ starting at and returning to 
the origin, consisting of positive
{\it and negative} unit steps in coordinate directions {\it and
zero steps}, and staying
in the region
$$
\big\{(x_1,x_2,\dots,x_{h}):w\ge x_1\ge x_2\ge \dots\ge x_{h}\ge 
0\big\}.
$$
Also for these paths there exists a formula (also following from
the result of Gessel and Zeilberger),
namely~\cite[Eq.~(19)]{Grabiner2002}. Applied to our situation, it gives the
following formula for the above number of matchings:
\begin{multline}
|\NCNN_n(h+1,w+1)|=\sum_{m\ge0}\binom n{2m}\\
\cdot
\coef{\frac {x^{2m}} {(2m)!}}
\sum_{k_1,\dots,k_{h}\in\Z}
\det_{1\le i,j\le h}\left(
I_{-i+j+(2h+2w+2)k_i}(2x)-I_{i+j+(2h+2w+2)k_i}(2x)
\right),
\label{eq:A2}
\end{multline}
where \( I_\al(x) \) is the modified Bessel function given by
$$
I_\al(x)=\sum_{\ell\ge0}\frac {(x/2)^{2\ell+\al}} {\ell!\,(\ell+\al)!},
$$
and $\coef{x^M}g(x)$ denotes the coefficient of $x^M$ in the power
series $g(x)$.
The task now is to establish equality of the expressions
in~\eqref{eq:A1} and~\eqref{eq:A2}. Again, we were short of a good
idea.

On the other hand, a ``general principle'' says that, by making things
more general, they may become easier. Now, there is another ``general
principle" that says that, whenever one encounters an identity related
to, respectively involving {\it standard\/}
Young tableaux, then there should exist a more general identity for
{\it semistandard\/} Young tableaux! In its turn, this more general
identity would be   formulated in terms of
{\it symmetric functions}.
This ``principle" is based on the simple fact that
$$
\frac {n!} {m_1!\,m_2!\cdots m_k!}=\coef{x_1x_2\cdots x_n}
e_{m_1}(\vx)\cdots e_{m_k}(\vx),
$$
with $m_1+m_2+\dots+m_k=n$.
In other words, a multinomial coefficient is a product of elementary
symmetric functions {\it in disguise}. This is straightforward to
carry out for~\eqref{eq:A1}, which becomes
\begin{equation}
|\CSYT_n(2h+1,2w+1)|=\coef{x_1x_2\cdots x_n}
\underset {\la_1-\la_{2h+1}\le 2w+1}{\sum_{|\la|= n}}
\underset{k_1+\dots+k_{2h+1}=0}{\sum_ {k_1,\dots,k_{2h+1}\in\Z}}
\det_{1\le i,j\le 2h+1}
\left(e_{\la_i-i+j+(2h+2w+2)k_i}(\vx)\right).
\label{eq:A3}
\end{equation}
It takes not much more effort to see that, under this perspective, \eqref{eq:A2} becomes
\begin{multline}
|\NCNN_n(h+1,w+1)|=\coef{x_1x_2\cdots x_n}
\sum_{k\ge0}e_k(\vx)
\sum_{k_1,\dots,k_{h}\in\Z}
\det_{1\le i,j\le h}\left(
f_{-i+j+(2h+2w+2)k_i}(\vx)-f_{i+j+(2h+2w+2)k_i}(\vx)
\right),
\label{eq:A4}
\end{multline}
where, as before,
$$
f_\al(\vx)=\sum_{\ell\ge0}e_\ell(\vx)e_{\ell+\al}(\vx).
$$
One may now speculate that the equality of~\eqref{eq:A1}
and~\eqref{eq:A2} is just the ``shadow'' of a symmetric function
identity. In other words, maybe the symmetric functions on the
right-hand sides of~\eqref{eq:A3} and~\eqref{eq:A4} are the same:
\begin{multline}
\underset{ \la_1-\la_{2h+1}\le 2w+1}{\sum_{\la:\ell(\la)\le2h+1}}\
\underset{k_1+\dots+k_{2h+1}=0}{\sum_{ k_1,\dots,k_{2h+1}\in\Z}}
\det_{1\le i,j\le 2h+1}\left(
e_{\la_i-i+j+(2h+2w+2)k_i}(\vx)\right)\\
=
\sum_{k\ge0}e_k(\vx)
\sum_{k_1,\dots,k_{h}\in\Z}
\det_{1\le i,j\le h}\left(
f_{-i+j+(2h+2w+2)k_i}(\vx)-f_{i+j+(2h+2w+2)k_i}(\vx)
\right).
\label{eq:A5}
\end{multline}
Computer experiments confirmed that. At this point, one becomes cocky:
is it really important to have $2w+1$ as constraint on the difference
between $\la_1$ and~$\la_{2h+1}$? Phrased differently,
does~\eqref{eq:A5} continue to hold if we replace $2w+1$ by~$w$, i.e.,
is it true that
\begin{multline}
\underset{ \la_1-\la_{2h+1}\le w}{\sum_{\la:\ell(\la)\le2h+1}}\
\underset{k_1+\dots+k_{2h+1}=0}{\sum_{ k_1,\dots,k_{2h+1}\in\Z}}
\det_{1\le i,j\le 2h+1}\left(
e_{\la_i-i+j+(2h+w+1)k_i}(\vx)\right)\\
=
\sum_{k\ge0}e_k(\vx)
\sum_{k_1,\dots,k_{h}\in\Z}
\det_{1\le i,j\le h}\left(
f_{-i+j+(2h+w+1)k_i}(\vx)-f_{i+j+(2h+w+1)k_i}(\vx)
\right),
\label{eq:A6}
\end{multline}
for positive integers~$w$? Obviously, we consulted again the computer,
and it said ``yes". We had found~\eqref{eq:ABL1}!

In view of these findings, another obvious question was what would
happen if, instead of replacing $2w+1$ by~$w$, we would replace $2h+1$
by~$2h$. (One could not hope for an identity that would be uniform in
odd {\it and\/} even bounds on the difference of the first and last
part of~$\la$ in the sum on the left-hand side.)
It did not escape our attention that, for $w\to\infty$, the
identity~\eqref{eq:A6} reduces to the bounded Littlewood
identity~\eqref{eq:BK_odd1}. As is well known,
this identity in turn comes with a
companion identity, namely~\eqref{eq:BK_even1}. We thus had an
``obvious" candidate for an ``even" analog of~\eqref{eq:A6}:
$$
\underset{ \la_1-\la_{2h}\le w}{\sum_{\la:\ell(\la)\le 2h}}\
\underset{k_1+\dots+k_{2h}=0}{\sum_{ k_1,\dots,k_{2h}\in\mathbb Z}}
\det_{1\le i,j\le 2h}\left(
e_{\la_i-i+j+(2h+w)k_i}(\vx)\right)
\overset?=
\sum_{k_1,\dots,k_{h}\in\mathbb Z}
\det_{1\le i,j\le h}\big(
f_{-i+j+(2h+w)k_i}(\vx)+
f_{i+j-1+(2h+w)k_i}(\vx)
\big).
$$
Alas, computer calculations quickly told us that this identity is
wrong. However, not too much! Inspection of the first terms in the expansion
and further computer calculations led us to
discover that there was only a sign missing on the right-hand side:
\begin{multline}
\underset{ \la_1-\la_{2h}\le w}{\sum_{\la:\ell(\la)\le 2h}}\
\underset{k_1+\dots+k_{2h}=0}{\sum_{ k_1,\dots,k_{2h}\in\mathbb Z}}
\det_{1\le i,j\le 2h}\left(
e_{\la_i-i+j+(2h+w)k_i}(\vx)\right)\\
=
\sum_{k_1,\dots,k_{h}\in\mathbb Z}
(-1)^{\sum_{i=1}^hk_i}
\det_{1\le i,j\le h}\big(
f_{-i+j+(2h+w)k_i}(\vx)+
f_{i+j-1+(2h+w)k_i}(\vx)
\big).
\label{eq:A7}
\end{multline}
We had found \eqref{eq:ABL2}!

Of course, there was still the problem of finding proofs
of~\eqref{eq:A6} and~\eqref{eq:A7}. However, now we were in a much
better situation: we stood on firm grounds, by knowing where all this
belonged to, namely to the area of {\it bounded Littlewood identities}.
Stembridge~\cite[Th.~7.1]{Stembridge1990} had provided a blueprint
for the line of proof: write the left-hand side as a sum of minors of
a given matrix; apply the minor summation formula in Corollary~\ref{cor:IsWa}
to obtain a Pfaffian; apply Gordon's reduction in Lemma~\ref{lem:Gordon5}
to reduce the Pfaffian
to a determinant of half the size; simplify the determinant by
applying elementary row and column operations, see
Corollary~\ref{cor:Gordon}. We managed to rewrite our
left-hand sides as sums of minors of a given matrix; see the proofs
of~\eqref{eq:odd_id1} and~\eqref{eq:Pf_bar_d}. It is
interesting to note that, from there on, we ``just" had to do the same
steps as in Stembridge's blueprint, except that the details of the
calculations turned out to be more demanding.

Further investigations along the same lines --- but somewhat more general
since we now based the calculations on the full minor summation
formula in Theorem~\ref{thm:IsWa} (instead of just
Corollary~\ref{cor:IsWa}) --- in the end led us to come up with the
further affine
bounded Littlewood identities in Theorems~\ref{thm:C} and~\ref{thm:D}.

\medskip
In conclusion, while we still had (and have) no bijective proof of Mortimer and
Prellberg's identity~\eqref{eq:MP} (as mentioned earlier,
such a bijection was in fact found
by Courtiel, Elvey~Price and Marcovici~\cite{Courtiel2021}), nor of
the generalization~\eqref{eq:MP_gen},
instead we found new conceptual symmetric
function identities that in particular {\it implied\/}~\eqref{eq:MP_gen}
(and thus also~\eqref{eq:MP}),
namely the affine bounded Littlewood identities in
Theorem~\ref{thm:affine_BK_intro}. These led us to many more
combinatorial and algebraic results; see Sections~\ref{sec:comb-ident}
and~\ref{sec:h=1}, and a forthcoming sequel to this article.

\section{Bijections between standard Young tableaux,
  matchings, and vacillating tableaux using growth diagrams}
\label{sec:growth}

The purpose of this appendix is to review and clarify the relations between
standard Young tableaux, matchings, and vacillating tableaux that
one finds in the literature. We do this by using Fomin's growth
diagrams, which allow us to give a uniform presentation.

More precisely, with $h=\fl{m/2}$, by means of explicit bijections,
we are going to relate:

\begin{enumerate} 
\item[$*$]the set $\SYT_n(m)$ of standard Young tableaux of size~$n$
  with at most $m$~rows;
\item[$*$]the set $\NNest_n(h+1)$ of (partial) matchings 
on
  $\{1,2,\dots,n\}$ without an $(h+1)$-nesting;
\item[$*$]the set $\NC_n(h+1)$ of (partial) matchings 
on
  $\{1,2,\dots,n\}$ without an $(h+1)$-crossing;
\item[$*$]the set $\VT_n(h)$ of vacillating tableaux with at most
  $h$~rows,  that is, the set of sequences\break
  \( (\emptyset=\lambda^{0},
  \lambda^{1}, \dots, \lambda^{n}=\emptyset) \) of partitions, where $\lambda^{i-1}$
  and $\lambda^i$ differ by at most one cell for $i=1,2,\dots,n$, and
  where $\lambda^i$ has at most $h$~rows for $i=0,1,\dots,n$.
\end{enumerate}

We refer to Subsection~\ref{sec:tableaux} for the definition of
standard Young tableaux, and to Definition~\ref{def:NCNN} for the
definition of crossings and nestings in matchings. We shall also need the
subset $\NNest_n(h+1/2)$ of $\NNest_n(h+1)$, which by definition is the set of all
elements of $\NNest_n(h+1)$ without an $(h+1/2)$-nesting
(cf.\ Definition~\ref{defn:ncnn}).
Similarly, we let $\NC_n(h+1/2)$ be the subset
of $\NC_n(h+1)$ consisting of all
elements of $\NC_n(h+1)$ without an $(h+1/2)$-crossing
(cf.\ again Definition~\ref{defn:ncnn}). Furthermore, we write $\VT_n(h^*)$
for the subset of $\VT_n(h)$ consisting of all those elements (walks)
for which a zero step cannot occur on the hyperplane $x_h=0$.
The reader should also recall the meaning of $\NCNN_n(r,s)$
from the paragraph above Definition~\ref{defn:ncnn'}.

\medskip
We are going to provide bijections for the identities in
the theorems and corollaries below. We start with the (enumerative)
symmetry of matchings concerning crossings and nestings of matchings
that has been the main theme in~\cite{Chen2007} (in the more general
context of set partitions). Although not stated explicitly
in~\cite{Chen2007}, identity~\eqref{eq:NC-NN} below for integers $r$ and~$s$
is a special case
of the main theorem of~\cite{Chen2007} (namely
\cite[Th.~1.1]{Chen2007} restricted to matchings). 

\begin{thm} \label{thm:NC-NN}
For positive integers $n$ and positive integers of
half-integers $r$ and~$s$, we have
\begin{equation}
\label{eq:NC-NN}
|\NCNN_n(r,s)|=|\NCNN_n(s,r)|.
\end{equation}
\end{thm}

The next theorem connects standard Young tableaux and matchings.

\begin{thm} \label{thm:SYT-NN}
For positive integers $n$ and $h$, we have
\begin{align}
\label{eq:SYT-NN-odd}
|\SYT_n(2h+1)|=|\NNest_n(h+1)|,\\
\label{eq:SYT-NN-even}
|\SYT_n(2h)|=|\NNest_n(h+1/2)|.
\end{align}
\end{thm}

By combining Theorems~\ref{thm:NC-NN} and~\ref{thm:SYT-NN}, we obtain
the following equalities. Since we are going to provide bijective
proofs for the two theorems above, the appropriate combination of
bijections also yields bijective proofs of these equalities.

\begin{cor} \label{thm:SYT-NC}
For positive integers $n$ and $h$, we have
\begin{align}
\label{eq:SYT-NC-odd}
|\SYT_n(2h+1)|=|\NC_n(h+1)|,\\
\label{eq:SYT-NC-even}
|\SYT_n(2h)|=|\NC_n(h+1/2)|.
\end{align}
\end{cor}

Identity \eqref{eq:SYT-NC-odd} was stated already as \eqref{eq:Chen2}
in the introduction and, as explained in Appendix~\ref{sec:motivation}, it inspired us
to discover the more general identity in~\eqref{eq:syt=ncnn1}.
Identity~\eqref{eq:SYT-NC-odd} arises from~\eqref{eq:syt=ncnn3} as the
limit case where $w\to\infty$.

Next, we relate matchings with restrictions on their nestings to
vacillating tableaux (walks). 

\begin{thm} \label{thm:NN-VT}
For positive integers $n$ and $h$, we have
\begin{align}
\label{eq:NN-VT-odd}
|\NNest_n(h+1)|=|\VT_n(h)|,\\
\label{eq:NN-VT-even}
|\NNest_n(h+1/2)|=|\VT_n(h^*)|.
\end{align}
\end{thm}

Here, identity \eqref{eq:NN-VT-odd} arises from~\eqref{eq:NCNN=VT1}
and~\eqref{eq:NCNN=VT3} in the limit as $h\to\infty$, whereas
identity~\eqref{eq:NN-VT-even} arises from~\eqref{eq:NCNN=VT2} in the
limit as $h\to\infty$.

Clearly, by combining Theorems~\ref{thm:SYT-NN} and~\ref{thm:NN-VT}
we obtain the following
corollary. Since we are going to present bijections that prove
Theorems~\ref{thm:SYT-NN} and~\ref{thm:NN-VT}, the appropriate
combinations of bijective proofs also provide bijections for the identities
below. 

\begin{cor} \label{thm:SYT-VT}
For positive integers $n$ and $h$, we have
\begin{align}
\label{eq:SYT-VT-odd}
|\SYT_n(2h+1)|=|\VT_n(h)|,\\
\label{eq:SYT-VT-even}
|\SYT_n(2h)|=|\VT_n(h^*)|.
\end{align}
\end{cor}

Identity \eqref{eq:SYT-VT-odd} is the result
of~\cite{Zeilberger_lazy}. It is also stated in an equivalent form
in \cite[Th.~1.1]{Eu2013}, together with a bijective proof.
Identity \eqref{eq:SYT-VT-even} is also stated in an equivalent form
in~\cite{Eu2013}, see
Conjecture~1.2 and Theorem~1.3 there, and it is proved
in~\cite{Eu2013} by a bijection.

\medskip
We are now going to outline bijective proofs of
\eqref{eq:NC-NN}--\eqref{eq:SYT-NN-even} and of~\eqref{eq:NN-VT-odd}
and~\eqref{eq:NN-VT-even} using Fomin's growth diagrams. 
We assume that the reader is familiar with the concept and application
of growth diagrams, say as described in \cite[Sec.~2]{KratCE}.
In particular, we need the consequence of Greene's theorem that
prescribes the lengths of the first row and the first column of a
partition in the growth diagram in terms of the lengths of longest NE-
and SE-chains of the configuration of X's in the region to the left
and below the partition; see \cite[Th.~2]{KratCE}.

Our first observation is that the identities in
Theorem~\ref{thm:SYT-NN} result directly from the application of the
(inverse) Robinson--Schensted correspondence to the pair $(T,T)$,
where $T$ is the standard Young tableau from $\SYT_n(2h+1)$ or
$\SYT_n(2h)$ that we are considering. 

\begin{proof}[Proof of \eqref{eq:SYT-NN-odd}]
Let $T\in \SYT_n(2h+1)$. We illustrate all steps in parallel by
considering the running example
$$
\Einheit.5cm
\Pfad(0,0),11116\endPfad
\Pfad(0,-1),1111\endPfad
\Pfad(0,-2),11\endPfad
\Pfad(0,-3),11\endPfad
\Pfad(0,0),6666112222\endPfad
\Pfad(1,0),6666652\endPfad
\Pfad(3,0),6\endPfad
\Label\ru{1\vrule width0pt depth3.5pt}(0,0)
\Label\ru{2\vrule width0pt depth3.5pt}(1,0)
\Label\ru{5\vrule width0pt depth3.5pt}(2,0)
\Label\ru{6\vrule width0pt depth3.5pt}(3,0)
\Label\ru{3\vrule width0pt depth3.5pt}(0,-1)
\Label\ru{7\vrule width0pt depth3.5pt}(1,-1)
\Label\ru{4\vrule width0pt depth3.5pt}(0,-2)
\Label\ru{9\vrule width0pt depth3.5pt}(1,-2)
\Label\ru{8\vrule width0pt depth3.5pt}(0,-3)
\Label\ru{10\vrule width0pt depth3.5pt}(1,-3)
\Label\ru{11\vrule width0pt depth3.5pt}(0,-4)
\Label\r{,}(4,-3)
\hskip2cm
$$
which is a standard Young tableau of size~11 with 5~rows, that is,
an element of $\SYT_{11}(5)=\SYT_{11}(2\cdot 2+1)$.

\begin{figure}
$$
\Einheit.36cm
\Pfad(-3,0),111111111111111111111111111111111\endPfad
\Pfad(-3,3),111111111111111111111111111111111\endPfad
\Pfad(-3,6),111111111111111111111111111111111\endPfad
\Pfad(-3,9),111111111111111111111111111111111\endPfad
\Pfad(-3,12),111111111111111111111111111111111\endPfad
\Pfad(-3,15),111111111111111111111111111111111\endPfad
\Pfad(-3,18),111111111111111111111111111111111\endPfad
\Pfad(-3,21),111111111111111111111111111111111\endPfad
\Pfad(-3,24),111111111111111111111111111111111\endPfad
\Pfad(-3,27),111111111111111111111111111111111\endPfad
\Pfad(-3,30),111111111111111111111111111111111\endPfad
\Pfad(-3,33),111111111111111111111111111111111\endPfad
\Pfad(-3,0),222222222222222222222222222222222\endPfad
\Pfad(0,0),222222222222222222222222222222222\endPfad
\Pfad(3,0),222222222222222222222222222222222\endPfad
\Pfad(6,0),222222222222222222222222222222222\endPfad
\Pfad(9,0),222222222222222222222222222222222\endPfad
\Pfad(12,0),222222222222222222222222222222222\endPfad
\Pfad(15,0),222222222222222222222222222222222\endPfad
\Pfad(18,0),222222222222222222222222222222222\endPfad
\Pfad(21,0),222222222222222222222222222222222\endPfad
\Pfad(24,0),222222222222222222222222222222222\endPfad
\Pfad(27,0),222222222222222222222222222222222\endPfad
\Pfad(30,0),222222222222222222222222222222222\endPfad
\PfadDicke{2pt}
%
\Label\ro{\text {\Huge X}}(13,31)
\Label\ro{\text {\Huge X}}(10,28)
\Label\ro{\text {\Huge X}}(4,25)
\Label\ro{\text {\Huge X}}(25,22)
\Label\ro{\text {\Huge X}}(1,19)
\Label\ro{\text {\Huge X}}(-2,16)
\Label\ro{\text {\Huge X}}(16,13)
\Label\ro{\text {\Huge X}}(22,10)
\Label\ro{\text {\Huge X}}(19,7)
\Label\ro{\text {\Huge X}}(7,4)
\Label\ro{\text {\Huge X}}(28,1)
\Label\u{\small\emptyset}(-3,0)
\Label\u{\small\emptyset}(0,0)
\Label\u{\small\emptyset}(3,0)
\Label\u{\small\emptyset}(6,0)
\Label\u{\small\emptyset}(9,0)
\Label\u{\small\emptyset}(12,0)
\Label\u{\small\emptyset}(15,0)
\Label\u{\small\emptyset}(18,0)
\Label\u{\small\emptyset}(21,0)
\Label\u{\small\emptyset}(24,0)
\Label\u{\small\emptyset}(27,0)
\Label\ru{\small\emptyset}(30,0)
\Label\r{\small\emptyset}(30,3)
\Label\r{\small\emptyset}(30,6)
\Label\r{\small\emptyset}(30,9)
\Label\r{\small\emptyset}(30,12)
\Label\r{\small\emptyset}(30,15)
\Label\r{\small\emptyset}(30,18)
\Label\r{\small\emptyset}(30,21)
\Label\r{\small\emptyset}(30,24)
\Label\r{\small\emptyset}(30,27)
\Label\r{\small\emptyset}(30,30)
\Label\r{\small\emptyset}(30,33)
\Label\lo{\small1}(27,3)
\Label\lo{\small1}(24,3)
\Label\lo{\small1}(21,3)
\Label\lo{\small1}(18,3)
\Label\lo{\small1}(15,3)
\Label\lo{\small1}(12,3)
\Label\lo{\small1}(9,3)
\Label\lo{\small1}(6,3)
\Label\lo{\small1}(3,3)
\Label\lo{\small1}(0,3)
\Label\lo{\small1}(-3,3)
\Label\lo{\small1}(27,6)
\Label\lo{\small1}(24,6)
\Label\lo{\small1}(21,6)
\Label\lo{\small1}(18,6)
\Label\lo{\small1}(15,6)
\Label\lo{\small1}(12,6)
\Label\lo{\small1}(9,6)
\Label\lo{\small2}(6,6)
\Label\lo{\small2}(3,6)
\Label\lo{\small2}(0,6)
\Label\lo{\small2}(-3,6)
\Label\lo{\small1}(27,9)
\Label\lo{\small1}(24,9)
\Label\lo{\small1}(21,9)
\Label\lo{\small2}(18,9)
\Label\lo{\small2}(15,9)
\Label\lo{\small2}(12,9)
\Label\lo{\small2}(9,9)
\Label\lo{\small21\ }(6,9)
\Label\lo{\small21\ }(3,9)
\Label\lo{\small21\ }(0,9)
\Label\lo{\small21\ }(-3,9)
\Label\lo{\small1}(27,12)
\Label\lo{\small1}(24,12)
\Label\lo{\small2}(21,12)
\Label\lo{\small21\ }(18,12)
\Label\lo{\small21\ }(15,12)
\Label\lo{\small21\ }(12,12)
\Label\lo{\small21\ }(9,12)
\Label\lo{\small211\ \ }(6,12)
\Label\lo{\small211\ \ }(3,12)
\Label\lo{\small211\ \ }(0,12)
\Label\lo{\small211\ \ \ }(-3,12)
\Label\lo{\small1}(27,15)
\Label\lo{\small1}(24,15)
\Label\lo{\small2}(21,15)
\Label\lo{\small21\ }(18,15)
\Label\lo{\small31\ }(15,15)
\Label\lo{\small31\ }(12,15)
\Label\lo{\small31\ }(9,15)
\Label\lo{\small311\ \ }(6,15)
\Label\lo{\small311\ \ }(3,15)
\Label\lo{\small311\ \ }(0,15)
\Label\lo{\small311\ \ \ }(-3,15)
\Label\lo{\small1}(27,18)
\Label\lo{\small1}(24,18)
\Label\lo{\small2}(21,18)
\Label\lo{\small21}(18,18)
\Label\lo{\small31}(15,18)
\Label\lo{\small31}(12,18)
\Label\lo{\small31\ }(9,18)
\Label\lo{\small311\ \ }(6,18)
\Label\lo{\small311\ \ }(3,18)
\Label\lo{\small311\ \ }(0,18)
\Label\lo{\small411\ \ \ }(-3,18)
\Label\lo{\small1}(27,21)
\Label\lo{\small1}(24,21)
\Label\lo{\small2\ }(21,21)
\Label\lo{\small21\ \ }(18,21)
\Label\lo{\small31\ \ }(15,21)
\Label\lo{\small31\ \ }(12,21)
\Label\lo{\small31\ \ }(9,21)
\Label\lo{\small311\ \ }(6,21)
\Label\lo{\small311\ \,\ \ }(3,21)
\Label\lo{\small411\ \ }(0,21)
\Label\lo{\small421\ \ \ }(-3,21)
\Label\lo{\small1}(27,24)
\Label\lo{\small2}(24,24)
\Label\lo{\small21\ }(21,24)
\Label\lo{\small211\ \ }(18,24)
\Label\lo{\small311\ \ }(15,24)
\Label\lo{\small311\ \ }(12,24)
\Label\lo{\small311\ \ }(9,24)
\Label\lo{\small3111\ \,\ \ }(6,24)
\Label\lo{\small3111\ \,\ \ }(3,24)
\Label\lo{\small4111\ \,\ \ }(0,24)
\Label\lo{\small4211\ \,\ \ \ }(-3,24)
\Label\lo{\small1}(27,27)
\Label\lo{\small2}(24,27)
\Label\lo{\small21\ }(21,27)
\Label\lo{\small211\ \ }(18,27)
\Label\lo{\small311\ \ }(15,27)
\Label\lo{\small311\ \ }(12,27)
\Label\lo{\small311\ \ }(9,27)
\Label\lo{\small3111\ \,\ \ }(6,27)
\Label\lo{\small4111\ \,\ \ }(3,27)
\Label\lo{\small4211\ \,\ \ }(0,27)
\Label\lo{\small4221\ \,\ \ \ }(-3,27)
\Label\lo{\small1}(27,30)
\Label\lo{\small2}(24,30)
\Label\lo{\small21\ }(21,30)
\Label\lo{\small211\ \ }(18,30)
\Label\lo{\small311\ \ }(15,30)
\Label\lo{\small311\ \ }(12,30)
\Label\lo{\small411\ \ }(9,30)
\Label\lo{\small4111\ \ \ \ }(6,30)
\Label\lo{\small4211\ \ \ \ }(3,30)
\Label\lo{\small4221\ \ \ \ }(0,30)
\Label\lo{\small4222 \ \ \ \ \ }(-3,30)
\Label\lo{\small1}(27,33)
\Label\lo{\small2}(24,33)
\Label\lo{\small21\ }(21,33)
\Label\lo{\small211\ \ }(18,33)
\Label\lo{\small311\ \ }(15,33)
\Label\lo{\small411\ \ }(12,33)
\Label\lo{\small421\ \ }(9,33)
\Label\lo{\small4211\ \ \ \ }(6,33)
\Label\lo{\small4221\ \ \ \ }(3,33)
\Label\lo{\small4222\ \ \ \ }(0,33)
\Label\lo{\small42221\ \ \ \ \ \ }(-3,33)
\hskip10.8cm
$$
\caption{Example for the bijection between $\SYT_{11}(5)$ and $\NNest_{11}(3)$.}
\label{fig:1}
\end{figure}

Given $T$, we now fill an $n\times n$ growth diagram by putting the
sequence of partitions arising by recording the subshapes of~$T$ formed
by the entries $1,2,\dots,i$ for $i=0,1,\dots,n$, along the left side
of the $n\times n$ square of cells, in increasing order from the
bottom to the top, and along the top side of the square, in increasing
order from right to left. In our running example, this sequence of
subshapes is
$$
\emptyset\subseteq 1\subseteq 2\subseteq 21\subseteq 211\subseteq 311
\subseteq 411\subseteq 421\subseteq 4211\subseteq 4221\subseteq 4222
\subseteq 42221.
$$
We have put this sequence to the left and on top of the $11\times 11$
square in Figure~\ref{fig:1}. (The other labels in the diagram should
be ignored at this point).

Now we apply the inverse growth diagram algorithm in direction
bottom/right. This produces a collection of crosses that is symmetric
with respect to the right/down diagonal of the square; see
Figure~\ref{fig:1}. By labeling the rows by $1,2,\dots,n$ from top to
bottom and the columns from left to right, these crosses correspond to
a matching (involution) 
on
$\{1,2,\dots,n\}$. In our running example,
this matching is
\begin{equation} \label{eq:match1} 
\Big\{(1,6),\ (2,5),\ (4,10),\ (8,9)
\Big\},
\end{equation}
with $3,7,11$ being fixed points of the matching. 
The largest nestings have size~2 --- namely $\big\{(1,6),\ (2,5)\big\}$
and $\big\{(4,10),\ (8,9)\big\}$ ---
hence this matching is an element of $\NNest_{11}(3)$.

Greene's theorem implies that the number of rows of the original
standard Young tableau~$T$ equals the length of the longest NE-chain
of X's. Since $T\in \SYT_n(2h+1)$, this length is bounded
above by~$2h+1$. Since we are in a symmetric situation, the length of
a longest NE-chain will occur among symmetric NE-chains. This in turn
implies that the largest nesting of the matching encoded by the
crosses has size at most~$h$. In other words, the matching is an
element of $\NNest_n(h+1)$.
This establishes the identity.
\end{proof}

\begin{proof}[Proof of \eqref{eq:SYT-NN-even}]
We proceed in the same manner as in the previous proof. Here, our
running example is the standard Young tableau
$$
\Einheit.5cm
\Pfad(0,0),11116\endPfad
\Pfad(0,-1),1111\endPfad
\Pfad(0,-2),11\endPfad
\Pfad(0,-3),11\endPfad
\Pfad(0,0),6666112222\endPfad
\Pfad(1,0),6666\endPfad
\Pfad(3,0),6\endPfad
\Label\ru{1\vrule width0pt depth3.5pt}(0,0)
\Label\ru{2\vrule width0pt depth3.5pt}(1,0)
\Label\ru{5\vrule width0pt depth3.5pt}(2,0)
\Label\ru{6\vrule width0pt depth3.5pt}(3,0)
\Label\ru{3\vrule width0pt depth3.5pt}(0,-1)
\Label\ru{7\vrule width0pt depth3.5pt}(1,-1)
\Label\ru{4\vrule width0pt depth3.5pt}(0,-2)
\Label\ru{9\vrule width0pt depth3.5pt}(1,-2)
\Label\ru{8\vrule width0pt depth3.5pt}(0,-3)
\Label\ru{10\vrule width0pt depth3.5pt}(1,-3)
\Label\r{,}(4,-3)
\hskip2cm
$$
which is an element of $\SYT_{10}(4)=\SYT_{10}(2\cdot 2)$.
The growth diagram that results in this case is shown in
Figure~\ref{fig:3}. The matching that corresponds to
the crosses in the figure is
\begin{equation} \label{eq:match2} 
\Big\{
(1,5),\ (2,4),\ (3,9),\ (7,8)
\Big\},
\end{equation}
with 6 and 10 being fixed points of the matching.
The largest nestings have size~2 --- namely $\big\{(1,5),\ (2,4)\big\}$
and $\big\{(3,9),\ (7,8)\big\}$. Moreover, the largest ``half-nesting"
is $\big\{(3,9)\big\}\cup\{6\}$, which is a $(1+1/2)$-nesting.
Hence, this matching is an element of $\NNest_{10}(2+1/2)$.

\begin{figure}
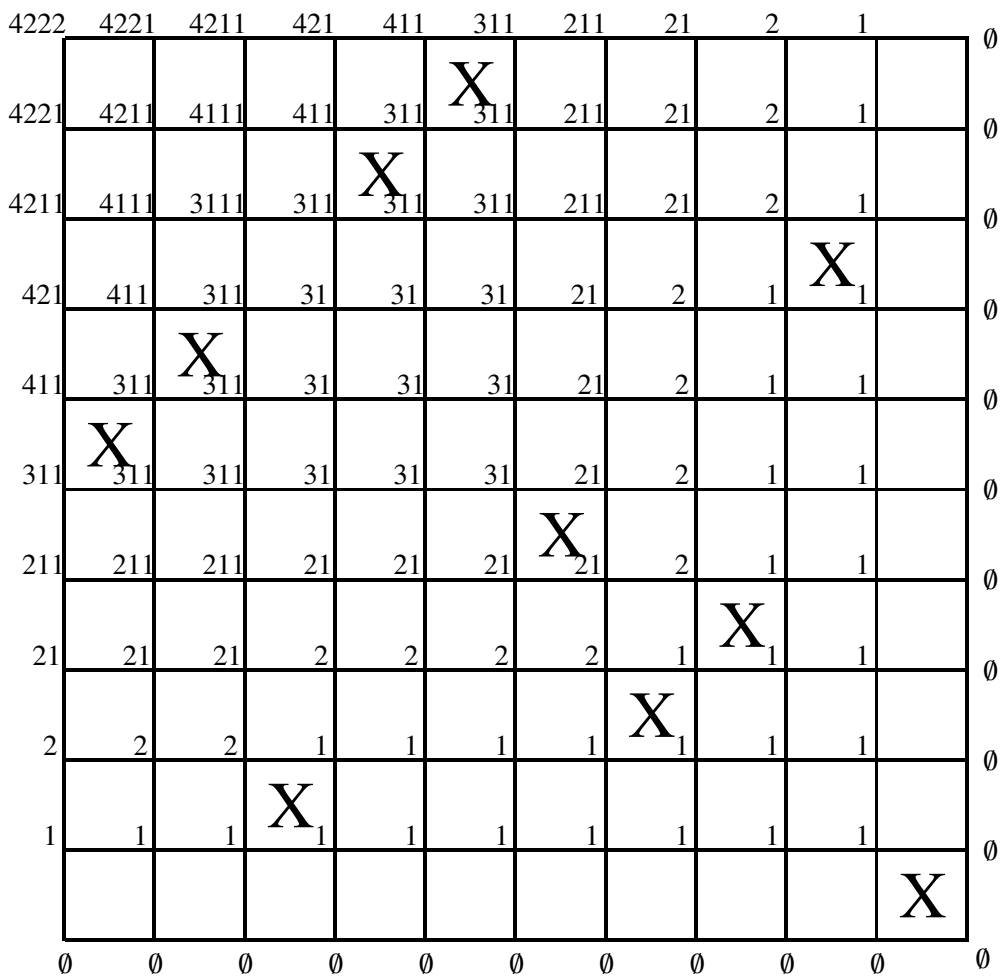

$$
\Einheit.4cm
\Pfad(0,0),111111111111111111111111111111\endPfad
\Pfad(0,3),111111111111111111111111111111\endPfad
\Pfad(0,6),111111111111111111111111111111\endPfad
\Pfad(0,9),111111111111111111111111111111\endPfad
\Pfad(0,12),111111111111111111111111111111\endPfad
\Pfad(0,15),111111111111111111111111111111\endPfad
\Pfad(0,18),111111111111111111111111111111\endPfad
\Pfad(0,21),111111111111111111111111111111\endPfad
\Pfad(0,24),111111111111111111111111111111\endPfad
\Pfad(0,27),111111111111111111111111111111\endPfad
\Pfad(0,30),111111111111111111111111111111\endPfad
\Pfad(0,0),222222222222222222222222222222\endPfad
\Pfad(3,0),222222222222222222222222222222\endPfad
\Pfad(6,0),222222222222222222222222222222\endPfad
\Pfad(9,0),222222222222222222222222222222\endPfad
\Pfad(12,0),222222222222222222222222222222\endPfad
\Pfad(15,0),222222222222222222222222222222\endPfad
\Pfad(18,0),222222222222222222222222222222\endPfad
\Pfad(21,0),222222222222222222222222222222\endPfad
\Pfad(24,0),222222222222222222222222222222\endPfad
\Pfad(27,0),222222222222222222222222222222\endPfad
\Pfad(30,0),222222222222222222222222222222\endPfad
\PfadDicke{2pt}
%
\Label\ro{\text {\Huge X}}(13,28)
\Label\ro{\text {\Huge X}}(10,25)
\Label\ro{\text {\Huge X}}(25,22)
\Label\ro{\text {\Huge X}}(4,19)
\Label\ro{\text {\Huge X}}(1,16)
\Label\ro{\text {\Huge X}}(16,13)
\Label\ro{\text {\Huge X}}(22,10)
\Label\ro{\text {\Huge X}}(19,7)
\Label\ro{\text {\Huge X}}(7,4)
\Label\ro{\text {\Huge X}}(28,1)
%
\Label\u{\small\emptyset}(0,0)
\Label\u{\small\emptyset}(3,0)
\Label\u{\small\emptyset}(6,0)
\Label\u{\small\emptyset}(9,0)
\Label\u{\small\emptyset}(12,0)
\Label\u{\small\emptyset}(15,0)
\Label\u{\small\emptyset}(18,0)
\Label\u{\small\emptyset}(21,0)
\Label\u{\small\emptyset}(24,0)
\Label\u{\small\emptyset}(27,0)
\Label\ru{\small\emptyset}(30,0)
\Label\r{\small\emptyset}(30,3)
\Label\r{\small\emptyset}(30,6)
\Label\r{\small\emptyset}(30,9)
\Label\r{\small\emptyset}(30,12)
\Label\r{\small\emptyset}(30,15)
\Label\r{\small\emptyset}(30,18)
\Label\r{\small\emptyset}(30,21)
\Label\r{\small\emptyset}(30,24)
\Label\r{\small\emptyset}(30,27)
\Label\r{\small\emptyset}(30,30)
\Label\lo{\small1}(27,3)
\Label\lo{\small1}(24,3)
\Label\lo{\small1}(21,3)
\Label\lo{\small1}(18,3)
\Label\lo{\small1}(15,3)
\Label\lo{\small1}(12,3)
\Label\lo{\small1}(9,3)
\Label\lo{\small1}(6,3)
\Label\lo{\small1}(3,3)
\Label\lo{\small1}(0,3)
\Label\lo{\small1}(27,6)
\Label\lo{\small1}(24,6)
\Label\lo{\small1}(21,6)
\Label\lo{\small1}(18,6)
\Label\lo{\small1}(15,6)
\Label\lo{\small1}(12,6)
\Label\lo{\small1}(9,6)
\Label\lo{\small2}(6,6)
\Label\lo{\small2}(3,6)
\Label\lo{\small2}(0,6)
\Label\lo{\small1}(27,9)
\Label\lo{\small1}(24,9)
\Label\lo{\small1}(21,9)
\Label\lo{\small2}(18,9)
\Label\lo{\small2}(15,9)
\Label\lo{\small2}(12,9)
\Label\lo{\small2}(9,9)
\Label\lo{\small21\ }(6,9)
\Label\lo{\small21\ }(3,9)
\Label\lo{\small21\ }(0,9)
\Label\lo{\small1}(27,12)
\Label\lo{\small1}(24,12)
\Label\lo{\small2}(21,12)
\Label\lo{\small21\ }(18,12)
\Label\lo{\small21\ }(15,12)
\Label\lo{\small21\ }(12,12)
\Label\lo{\small21\ }(9,12)
\Label\lo{\small211\ \ }(6,12)
\Label\lo{\small211\ \ }(3,12)
\Label\lo{\small211\ \ }(0,12)
\Label\lo{\small1}(27,15)
\Label\lo{\small1}(24,15)
\Label\lo{\small2}(21,15)
\Label\lo{\small21\ }(18,15)
\Label\lo{\small31\ }(15,15)
\Label\lo{\small31\ }(12,15)
\Label\lo{\small31\ }(9,15)
\Label\lo{\small311\ \ }(6,15)
\Label\lo{\small311\ \ }(3,15)
\Label\lo{\small311\ \ }(0,15)
\Label\lo{\small1}(27,18)
\Label\lo{\small1}(24,18)
\Label\lo{\small2}(21,18)
\Label\lo{\small21}(18,18)
\Label\lo{\small31}(15,18)
\Label\lo{\small31}(12,18)
\Label\lo{\small31\ }(9,18)
\Label\lo{\small311\ \ }(6,18)
\Label\lo{\small311\ \ }(3,18)
\Label\lo{\small411\ \ }(0,18)
\Label\lo{\small1}(27,21)
\Label\lo{\small1}(24,21)
\Label\lo{\small2\ }(21,21)
\Label\lo{\small21\ \ }(18,21)
\Label\lo{\small31\ \ }(15,21)
\Label\lo{\small31\ \ }(12,21)
\Label\lo{\small31\ \ }(9,21)
\Label\lo{\small311\ \ }(6,21)
\Label\lo{\small411\ \,\ \ }(3,21)
\Label\lo{\small421\ \ }(0,21)
\Label\lo{\small1}(27,24)
\Label\lo{\small2}(24,24)
\Label\lo{\small21\ }(21,24)
\Label\lo{\small211\ \ }(18,24)
\Label\lo{\small311\ \ }(15,24)
\Label\lo{\small311\ \ }(12,24)
\Label\lo{\small311\ \ }(9,24)
\Label\lo{\small3111\ \,\ \ }(6,24)
\Label\lo{\small4111\ \,\ \ }(3,24)
\Label\lo{\small4211\ \,\ \ }(0,24)
\Label\lo{\small1}(27,27)
\Label\lo{\small2}(24,27)
\Label\lo{\small21\ }(21,27)
\Label\lo{\small211\ \ }(18,27)
\Label\lo{\small311\ \ }(15,27)
\Label\lo{\small311\ \ }(12,27)
\Label\lo{\small411\ \ }(9,27)
\Label\lo{\small4111\ \,\ \ }(6,27)
\Label\lo{\small4211\ \,\ \ }(3,27)
\Label\lo{\small4221\ \,\ \ }(0,27)
\Label\lo{\small1}(27,30)
\Label\lo{\small2}(24,30)
\Label\lo{\small21\ }(21,30)
\Label\lo{\small211\ \ }(18,30)
\Label\lo{\small311\ \ }(15,30)
\Label\lo{\small411\ \ }(12,30)
\Label\lo{\small421\ \ }(9,30)
\Label\lo{\small4211\ \,\ \ }(6,30)
\Label\lo{\small4221\ \,\ \ }(3,30)
\Label\lo{\small4222\ \,\ \ }(0,30)
\hskip12.3cm
$$
\caption{Example for the bijection between $\SYT_{10}(4)$ and $\NNest_{10}(2+1/2)$.}
\label{fig:3}
\end{figure}

By Greene's theorem again, the length of the longest NE-chain of
crosses is bounded above by~$2h$. Also here, we may restrict our
attention to symmetric NE-chains. The previous observation implies
that the size of a nesting in the matching (involution) corresponding to the
crosses can be at most~$h$. Moreover, a cross on the right-down
diagonal of the square must not be part of a (symmetric) NE-chain of
length $2h+1$. Translated into the corresponding property for the
matching defined by the set of crosses, this means that
the matching does not contain an $(h+1/2)$-nesting. In other words,
the matching is an element of $\NNest_n(h+1/2)$.

This completes the proof of the identity.
\end{proof}

Next we prove the identities in Theorem~\ref{thm:NN-VT}.

\begin{proof}[Proof of \eqref{eq:NN-VT-odd}]
Let $M$ be a matching in $\NNest_n(h+1)$. Obviously, in the square
growth diagrams that we used in the proofs of~\eqref{eq:SYT-NN-odd}
and~\eqref{eq:SYT-NN-even} (cf.\ Figures~\ref{fig:1} and~\ref{fig:3}),
only one half of the information is 
essential, the rest is redundant. Indeed, here we content ourselves
with just one half of the square, namely the triangular region below
the right-down diagonal. We put the matching $M$ as crosses into the
triangle in the same fashion as in the proofs of~\eqref{eq:SYT-NN-odd}
and~\eqref{eq:SYT-NN-even}.
Also here, we illustrate all the steps of the construction
by a running example, namely the matching in~\eqref{eq:match1}, which
is an element of $\NNest_{11}(3)=\NNest_{11}(2+1)$.
Figure~\ref{fig:2} shows that matching filled in the triangular region
forming half of an $11\times 11$ rectangle. (Obviously this is half of
the cell diagram of Figure~\ref{fig:1}, excluding the fixed points
along the right-down diagonal.) The labelings should be ignored at
this point.

\begin{figure}
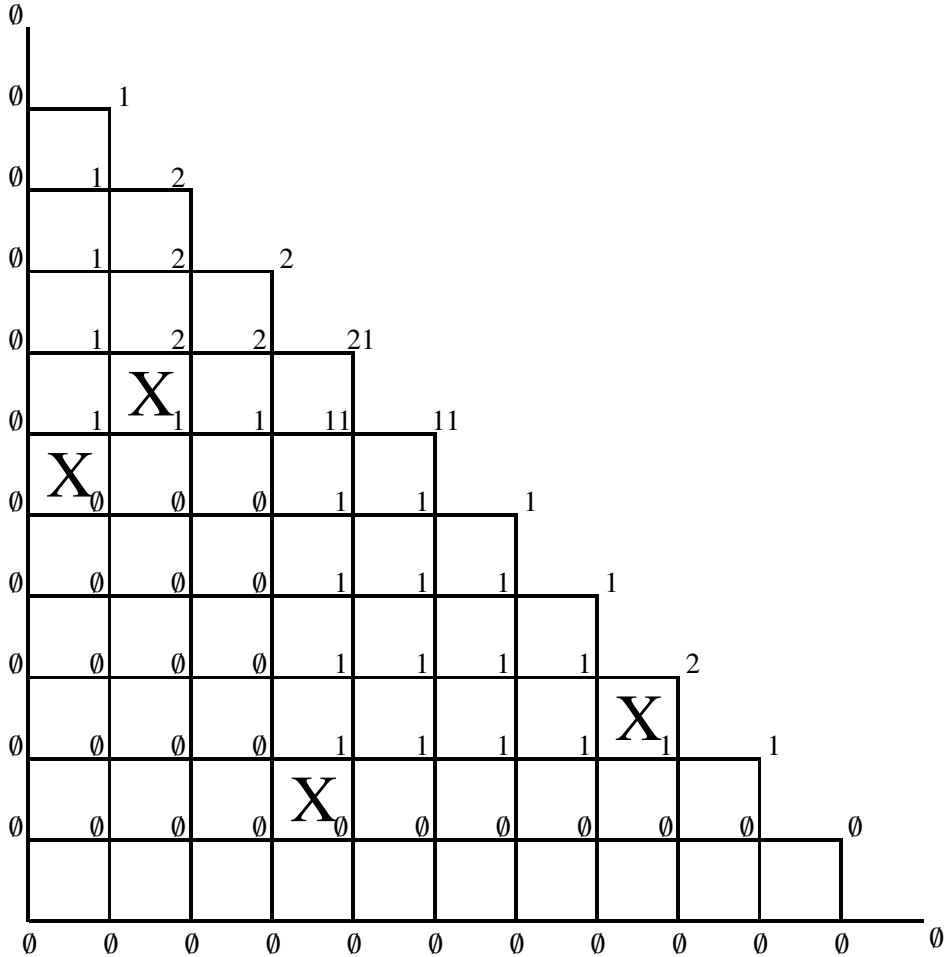

$$
\Einheit.36cm
\Pfad(-3,0),111111111111111111111111111111111\endPfad
\Pfad(-3,3),111111111111111111111111111111\endPfad
\Pfad(-3,6),111111111111111111111111111\endPfad
\Pfad(-3,9),111111111111111111111111\endPfad
\Pfad(-3,12),111111111111111111111\endPfad
\Pfad(-3,15),111111111111111111\endPfad
\Pfad(-3,18),111111111111111\endPfad
\Pfad(-3,21),111111111111\endPfad
\Pfad(-3,24),111111111\endPfad
\Pfad(-3,27),111111\endPfad
\Pfad(-3,30),111\endPfad
\Pfad(-3,0),222222222222222222222222222222222\endPfad
\Pfad(0,0),222222222222222222222222222222\endPfad
\Pfad(3,0),222222222222222222222222222\endPfad
\Pfad(6,0),222222222222222222222222\endPfad
\Pfad(9,0),222222222222222222222\endPfad
\Pfad(12,0),222222222222222222\endPfad
\Pfad(15,0),222222222222222\endPfad
\Pfad(18,0),222222222222\endPfad
\Pfad(21,0),222222222\endPfad
\Pfad(24,0),222222\endPfad
\Pfad(27,0),222\endPfad
%
\PfadDicke{2pt}
%
\Label\ro{\text {\Huge X}}(1,19)
\Label\ro{\text {\Huge X}}(-2,16)
\Label\ro{\text {\Huge X}}(19,7)
\Label\ro{\text {\Huge X}}(7,4)
%
\Label\u{\small\emptyset}(-3,0)
\Label\u{\small\emptyset}(0,0)
\Label\u{\small\emptyset}(3,0)
\Label\u{\small\emptyset}(6,0)
\Label\u{\small\emptyset}(9,0)
\Label\u{\small\emptyset}(12,0)
\Label\u{\small\emptyset}(15,0)
\Label\u{\small\emptyset}(18,0)
\Label\u{\small\emptyset}(21,0)
\Label\u{\small\emptyset}(24,0)
\Label\u{\small\emptyset}(27,0)
\Label\ru{\small\emptyset}(30,0)
%
%
\Label\ro{\small\emptyset}(27,3)
\Label\lo{\small\emptyset}(24,3)
\Label\lo{\small\emptyset}(21,3)
\Label\lo{\small\emptyset}(18,3)
\Label\lo{\small\emptyset}(15,3)
\Label\lo{\small\emptyset}(12,3)
\Label\lo{\small\emptyset}(9,3)
\Label\lo{\small\emptyset}(6,3)
\Label\lo{\small\emptyset}(3,3)
\Label\lo{\small\emptyset}(0,3)
\Label\lo{\small\emptyset}(-3,3)
\Label\ro{\small1}(24,6)
\Label\lo{\small1}(21,6)
\Label\lo{\small1}(18,6)
\Label\lo{\small1}(15,6)
\Label\lo{\small1}(12,6)
\Label\lo{\small1}(9,6)
\Label\lo{\small\emptyset}(6,6)
\Label\lo{\small\emptyset}(3,6)
\Label\lo{\small\emptyset}(0,6)
\Label\lo{\small\emptyset}(-3,6)
\Label\ro{\small2}(21,9)
\Label\lo{\small1}(18,9)
\Label\lo{\small1}(15,9)
\Label\lo{\small1}(12,9)
\Label\lo{\small1}(9,9)
\Label\lo{\small\emptyset}(6,9)
\Label\lo{\small\emptyset}(3,9)
\Label\lo{\small\emptyset}(0,9)
\Label\lo{\small\emptyset}(-3,9)
\Label\ro{\small1}(18,12)
\Label\lo{\small1}(15,12)
\Label\lo{\small1}(12,12)
\Label\lo{\small1}(9,12)
\Label\lo{\small\emptyset}(6,12)
\Label\lo{\small\emptyset}(3,12)
\Label\lo{\small\emptyset}(0,12)
\Label\lo{\small\emptyset}(-3,12)
%
\Label\ro{\small1}(15,15)
\Label\lo{\small1}(12,15)
\Label\lo{\small1}(9,15)
\Label\lo{\small\emptyset}(6,15)
\Label\lo{\small\emptyset}(3,15)
\Label\lo{\small\emptyset}(0,15)
\Label\lo{\small\emptyset}(-3,15)
\Label\ro{\small11\ }(12,18)
\Label\lo{\small11\ }(9,18)
\Label\lo{\small1}(6,18)
\Label\lo{\small1}(3,18)
\Label\lo{\small1}(0,18)
\Label\lo{\small\emptyset}(-3,18)
\Label\ro{\small21\ \ }(9,21)
\Label\lo{\small2}(6,21)
\Label\lo{\small2}(3,21)
\Label\lo{\small1}(0,21)
\Label\lo{\small\emptyset}(-3,21)
\Label\ro{\small2}(6,24)
\Label\lo{\small2}(3,24)
\Label\lo{\small1}(0,24)
\Label\lo{\small\emptyset}(-3,24)
%
\Label\lo{\small2}(3,27)
\Label\lo{\small1}(0,27)
\Label\lo{\small\emptyset}(-3,27)
\Label\ro{\small1}(0,30)
\Label\lo{\small\emptyset}(-3,30)
\Label\lo{\small\emptyset}(-3,33)
\hskip10.8cm
$$
\caption{Example for the bijection between $\NNest_{11}(3)$ and $\VT_{11}(2)$.}
\label{fig:2}
\end{figure}

Now we place empty diagrams along the bottom side and along the left
side of the triangular cell arrangement. Subsequently we apply the (forward) growth
diagram algorithm in direction top/right. Figure~\ref{fig:2} shows the
result in our running example.

We read the sequence $(\la^0,\la^1,\dots,\la^{2n})$
of diagrams that we obtain along the diagonal of
the triangular region. In our example, we obtain
$$
\Big(
\emptyset,\ \emptyset,\ 1,\ 1,\ 2,\ 2,\ 2,\ 2,\ 21,\ 11,\ 11,\ 1,\ 1,\
1,\ 1,\ 1,\ 2,\ 1,\ 1,\ \emptyset,\ \emptyset,\ \emptyset,\ \emptyset
\Big).
$$
However, again, this sequence contains a lot of redundant information.
Namely, since we filled a matching into the triangular cell arrangement,
there are exactly three possibilities for the subsequences of the form
$\la^{2i},\la^{2i+1},\la^{2i+2}$, for  $i=0,1,\dots,n-1$:

\begin{itemize} 
\item $\la^{2i}=\la^{2i+1}\subsetneqq\la^{2i+2}$;
\item $\la^{2i}\supsetneqq\la^{2i+1}=\la^{2i+2}$;
\item $\la^{2i}=\la^{2i+1}=\la^{2i+2}$.
\end{itemize}

Hence it suffices to keep all the even-indexed partitions
$(\la^0,\la^2,\dots,\la^{2n})$; the
odd-indexed ones can be reconstructed from them. In our running example,
this leads to the reduced sequence
$$
\Big(
\emptyset,\ 1,\ 2,\ 2,\ 21,\ 11,\ 1,\ 1,\ 2,\ 1,\ \emptyset,\ \emptyset
\Big).
$$

Since we started with a matching without an $(h+1)$-nesting, the
longest NE-chain of crosses is bounded above by~$h$.
By Greene's theorem, this implies that the first rows of the
diagrams~$\la^{2i}$ are also bounded above by~$h$.

In the final step we conjugate all partitions $\la^{2i}$ in the
(reduced) sequence and obtain
$\Big((\la^0)',(\la^2)',\break \dots,(\la^{2n})'\Big)$.
Clearly, this produces a vacillating tableau in
$\VT_n(h)$. In the case of our example, we obtain
$$
\Big(
\emptyset,\ 1,\ 11,\ 11,\ 21,\ 2,\ 1,\ 1,\ 11,\ 1,\ \emptyset,\ \emptyset
\Big).
$$
This completes the proof.
\end{proof}

\begin{proof}[Proof of \eqref{eq:NN-VT-even}]
Let $M\in\NNest(h+1/2)$.
We proceed as in the proof of~\eqref{eq:NN-VT-odd}, with one
exception: here, when we put our matching~$M$ in the form of crosses
into the triangular cell arrangement, we keep the crosses
corresponding to fixed points of~$M$ along the right-down diagonal.
As running example we choose the matching in~\eqref{eq:match2}. It is
a matching in $\NNest_{10}(2+1/2)$. The corresponding arrangement of
crosses is shown in Figure~\ref{fig:4}. For better recognition, the fixed
points are placed into dotted cells.

\begin{figure}
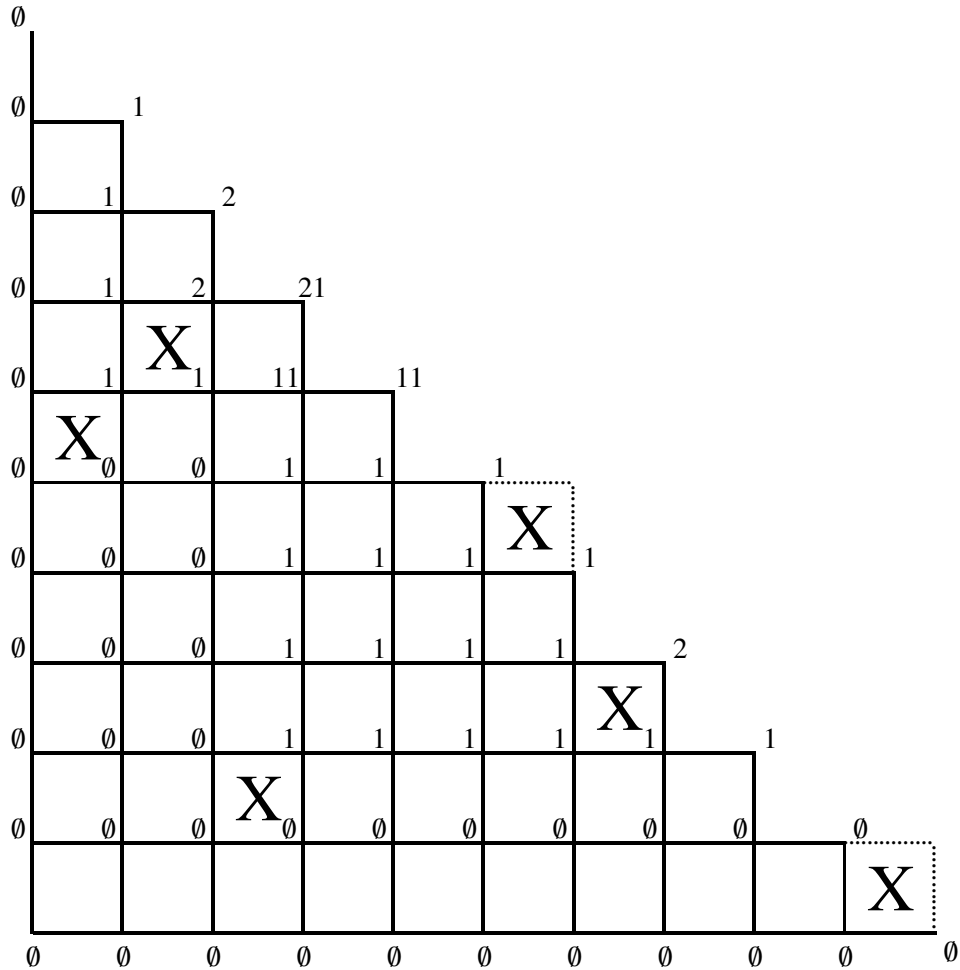

$$
\Einheit.4cm
\Pfad(0,0),111111111111111111111111111111\endPfad
\Pfad(0,3),111111111111111111111111111\endPfad
\Pfad(0,6),111111111111111111111111\endPfad
\Pfad(0,9),111111111111111111111\endPfad
\Pfad(0,12),111111111111111111\endPfad
\Pfad(0,15),111111111111111\endPfad
\Pfad(0,18),111111111111\endPfad
\Pfad(0,21),111111111\endPfad
\Pfad(0,24),111111\endPfad
\Pfad(0,27),111\endPfad
\Pfad(0,0),222222222222222222222222222222\endPfad
\Pfad(3,0),222222222222222222222222222\endPfad
\Pfad(6,0),222222222222222222222222\endPfad
\Pfad(9,0),222222222222222222222\endPfad
\Pfad(12,0),222222222222222222\endPfad
\Pfad(15,0),222222222222222\endPfad
\Pfad(18,0),222222222222\endPfad
\Pfad(21,0),222222222\endPfad
\Pfad(24,0),222222\endPfad
\Pfad(27,0),222\endPfad
%
\SPfad(30,0),222555\endSPfad
\SPfad(18,12),222555\endSPfad
\PfadDicke{2pt}
%
\Label\ro{\text {\Huge X}}(4,19)
\Label\ro{\text {\Huge X}}(1,16)
\Label\ro{\text {\Huge X}}(16,13)
\Label\ro{\text {\Huge X}}(19,7)
\Label\ro{\text {\Huge X}}(7,4)
\Label\ro{\text {\Huge X}}(28,1)
\Label\u{\small\emptyset}(0,0)
\Label\u{\small\emptyset}(3,0)
\Label\u{\small\emptyset}(6,0)
\Label\u{\small\emptyset}(9,0)
\Label\u{\small\emptyset}(12,0)
\Label\u{\small\emptyset}(15,0)
\Label\u{\small\emptyset}(18,0)
\Label\u{\small\emptyset}(21,0)
\Label\u{\small\emptyset}(24,0)
\Label\u{\small\emptyset}(27,0)
\Label\ru{\small\emptyset}(30,0)
%
%
\Label\ro{\small\emptyset}(27,3)
\Label\lo{\small\emptyset}(24,3)
\Label\lo{\small\emptyset}(21,3)
\Label\lo{\small\emptyset}(18,3)
\Label\lo{\small\emptyset}(15,3)
\Label\lo{\small\emptyset}(12,3)
\Label\lo{\small\emptyset}(9,3)
\Label\lo{\small\emptyset}(6,3)
\Label\lo{\small\emptyset}(3,3)
\Label\lo{\small\emptyset}(0,3)
\Label\ro{\small1}(24,6)
\Label\lo{\small1}(21,6)
\Label\lo{\small1}(18,6)
\Label\lo{\small1}(15,6)
\Label\lo{\small1}(12,6)
\Label\lo{\small1}(9,6)
\Label\lo{\small\emptyset}(6,6)
\Label\lo{\small\emptyset}(3,6)
\Label\lo{\small\emptyset}(0,6)
\Label\ro{\small2}(21,9)
\Label\lo{\small1}(18,9)
\Label\lo{\small1}(15,9)
\Label\lo{\small1}(12,9)
\Label\lo{\small1}(9,9)
\Label\lo{\small\emptyset}(6,9)
\Label\lo{\small\emptyset}(3,9)
\Label\lo{\small\emptyset}(0,9)
\Label\ro{\small1}(18,12)
\Label\lo{\small1}(15,12)
\Label\lo{\small1}(12,12)
\Label\lo{\small1}(9,12)
\Label\lo{\small\emptyset}(6,12)
\Label\lo{\small\emptyset}(3,12)
\Label\lo{\small\emptyset}(0,12)
%
\Label\ro{\small1}(15,15)
\Label\lo{\small1}(12,15)
\Label\lo{\small1}(9,15)
\Label\lo{\small\emptyset}(6,15)
\Label\lo{\small\emptyset}(3,15)
\Label\lo{\small\emptyset}(0,15)
\Label\ro{\small11}(12,18)
\Label\lo{\small11\ }(9,18)
\Label\lo{\small1}(6,18)
\Label\lo{\small1}(3,18)
\Label\lo{\small\emptyset}(0,18)
\Label\ro{\small21\ \ }(9,21)
\Label\lo{\small2}(6,21)
\Label\lo{\small1}(3,21)
\Label\lo{\small\emptyset}(0,21)
\Label\ro{\small2}(6,24)
\Label\lo{\small1}(3,24)
\Label\lo{\small\emptyset}(0,24)
\Label\ro{\small1}(3,27)
\Label\lo{\small\emptyset}(0,27)
\Label\lo{\small\emptyset}(0,30)
\hskip12.3cm
$$
\caption{Example for the bijection between $\NNest_{10}(2+1/2)$ and $\VT_{10}(2^*)$.}
\label{fig:4}
\end{figure}

The condition of $M$ not having an $(h+1)$-nesting and not having an
$(h+1/2)$-nesting translates into the
condition that all NE-chains of crosses in the configuration
have length at most~$h$, may they contain a cross
on the right-down diagonal (that is, a fixed point) or not.

Now, as in the proof of \eqref{eq:NN-VT-odd}, we place empty diagrams
along the bottom side and along the left 
side of the triangular cell arrangement. Subsequently we apply the (forward) growth
diagram algorithm in direction top/right; see Figure~\ref{fig:4}.
Again, we read every second
diagram along the right-down diagonal, say
$(\la^0,\la^2,\dots,\la^{2n})$.
In our running example in Figure~\ref{fig:4}, we read
$$
\Big(
\emptyset,\ 1,\ 2,\ 21,\ 11,\ 1,\ 1,\ 2,\ 1,\ \emptyset,\ \emptyset
\Big).
$$

Since the
longest NE-chain of crosses is bounded above by~$h$,
Greene's theorem implies that the first rows of the
diagrams~$\la^{2i}$ are bounded above by~$h$.
Moreover, we have $\la^{2i}=\la^{2i+2}$ for some~$i$ if and only if
at this place we find a cross corresponding to a fixed point of the
matching~$M$. Consequently, again by Greene's theorem, in this case
the length of the first row of $\la^{2i}=\la^{2i+2}$ is in fact at
most~$h-1$. 

As before, the final step consists in conjugating all partitions of
the sequence $(\la^0,\la^2,\dots,\la^{2n})$.
By the above observations it should be obvious that in this manner we
obtain an element of $\VT_n(h^*)$. In our example, we obtain
$$
\Big(
\emptyset,\ 1,\ 11,\ 21,\ 2,\ 1,\ 1,\ 11,\ 1,\ \emptyset,\ \emptyset
\Big),
$$
which is indeed an element of $\VT_{10}(2+1/2)$.
This finishes the  proof.
\end{proof}

Finally, for the sake of completeness, we recall the growth diagram
bijection from~\cite[last paragraph of Sec.~3]{KratCE} that proves the
symmetry relation in Theorem~\ref{thm:NC-NN}.  

\begin{proof}[Proof of \eqref{eq:NC-NN}]
We are going to describe a bijection between matchings $M_1$ and $M_2$
that has the property that, if $M_1$ has a $p$-crossing and a
$q$-nesting, then $M_2$ has a $q$-crossing and a $p$-nesting, for any
positive integers or half-integers~$p$ and~$q$. It is easy to see that
such a bijection implies the relation~\eqref{eq:NC-NN}.

In brief, the bijection works by taking the matching~$M_1$, putting it
in form of an arrangement of crosses into a triangular cell arrangement
of the appropriate size as in the proof of~\eqref{eq:NN-VT-even}, then
applying the forward growth diagram algorithm as in that proof,
subsequently conjugating all partitions along the right-down diagonal
and forgetting all other partition labels and the crosses that do not
correspond to fixed points, and by 
finally applying the inverse (backward) growth diagram algorithm
(keeping the crosses that corresponded to fixed points). The crosses
that one obtains define the image matching~$M_2$.
Greene's theorem implies the asserted properties concerning crossings
and nestings.

\begin{figure}
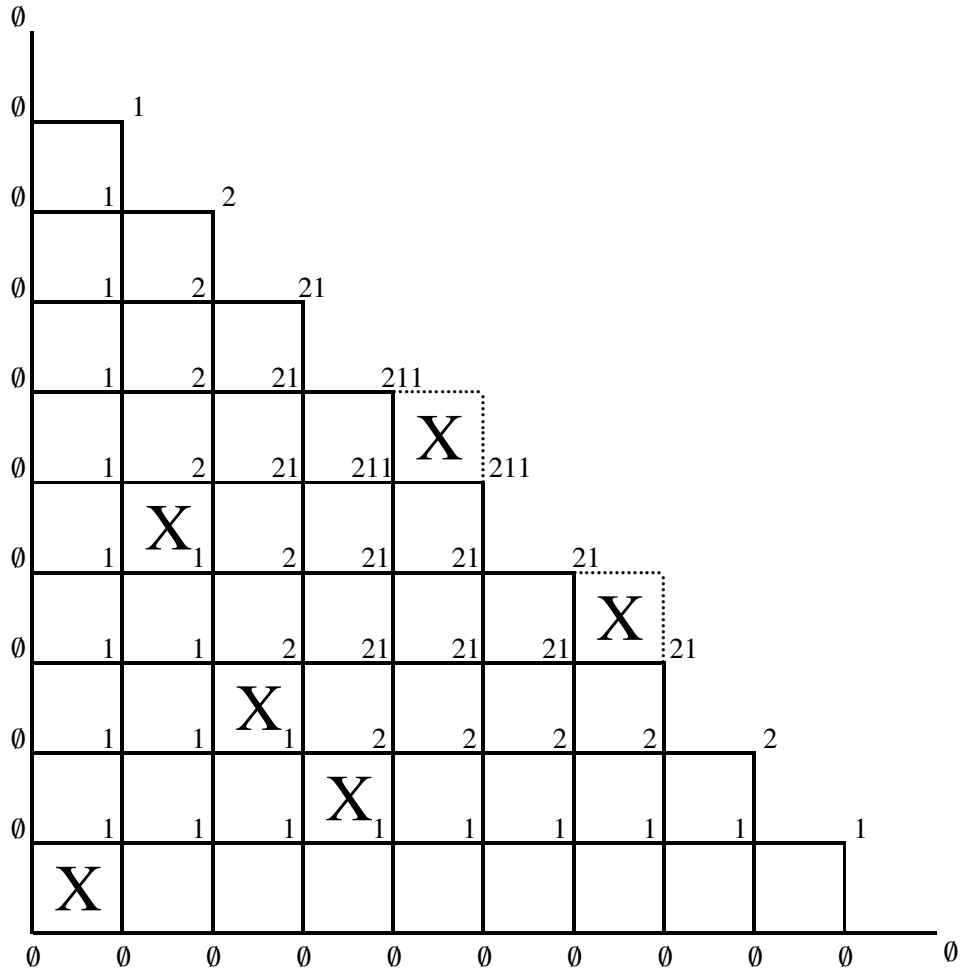

$$
\Einheit.4cm
\Pfad(0,0),111111111111111111111111111111\endPfad
\Pfad(0,3),111111111111111111111111111\endPfad
\Pfad(0,6),111111111111111111111111\endPfad
\Pfad(0,9),111111111111111111111\endPfad
\Pfad(0,12),111111111111111111\endPfad
\Pfad(0,15),111111111111111\endPfad
\Pfad(0,18),111111111111\endPfad
\Pfad(0,21),111111111\endPfad
\Pfad(0,24),111111\endPfad
\Pfad(0,27),111\endPfad
\Pfad(0,0),222222222222222222222222222222\endPfad
\Pfad(3,0),222222222222222222222222222\endPfad
\Pfad(6,0),222222222222222222222222\endPfad
\Pfad(9,0),222222222222222222222\endPfad
\Pfad(12,0),222222222222222222\endPfad
\Pfad(15,0),222222222222222\endPfad
\Pfad(18,0),222222222222\endPfad
\Pfad(21,0),222222222\endPfad
\Pfad(24,0),222222\endPfad
\Pfad(27,0),222\endPfad
\SPfad(21,9),222555\endSPfad
\SPfad(15,15),222555\endSPfad
\PfadDicke{2pt}
\Label\ro{\text {\Huge X}}(1,1)
\Label\ro{\text {\Huge X}}(4,13)
\Label\ro{\text {\Huge X}}(7,7)
\Label\ro{\text {\Huge X}}(10,4)
\Label\ro{\text {\Huge X}}(13,16)
\Label\ro{\text {\Huge X}}(19,10)
\Label\u{\small\emptyset}(0,0)
\Label\u{\small\emptyset}(3,0)
\Label\u{\small\emptyset}(6,0)
\Label\u{\small\emptyset}(9,0)
\Label\u{\small\emptyset}(12,0)
\Label\u{\small\emptyset}(15,0)
\Label\u{\small\emptyset}(18,0)
\Label\u{\small\emptyset}(21,0)
\Label\u{\small\emptyset}(24,0)
\Label\u{\small\emptyset}(27,0)
\Label\ru{\small\emptyset}(30,0)
\Label\ro{\small1}(27,3)
\Label\lo{\small1}(24,3)
\Label\lo{\small1}(21,3)
\Label\lo{\small1}(18,3)
\Label\lo{\small1}(15,3)
\Label\lo{\small1}(12,3)
\Label\lo{\small1}(9,3)
\Label\lo{\small1}(6,3)
\Label\lo{\small1}(3,3)
\Label\lo{\small\emptyset}(0,3)
\Label\ro{\small2}(24,6)
\Label\lo{\small2}(21,6)
\Label\lo{\small2}(18,6)
\Label\lo{\small2}(15,6)
\Label\lo{\small2}(12,6)
\Label\lo{\small1}(9,6)
\Label\lo{\small1}(6,6)
\Label\lo{\small1}(3,6)
\Label\lo{\small\emptyset}(0,6)
\Label\ro{\ \small21}(21,9)
\Label\lo{\small21\ }(18,9)
\Label\lo{\small21\ }(15,9)
\Label\lo{\small21\ }(12,9)
\Label\lo{\small2}(9,9)
\Label\lo{\small1}(6,9)
\Label\lo{\small1}(3,9)
\Label\lo{\small\emptyset}(0,9)
\Label\ro{\small21\ }(18,12)
\Label\lo{\small21\ }(15,12)
\Label\lo{\small21\ }(12,12)
\Label\lo{\small2}(9,12)
\Label\lo{\small1}(6,12)
\Label\lo{\small1}(3,12)
\Label\lo{\small\emptyset}(0,12)
\Label\ro{\ \ \ \small211}(15,15)
\Label\lo{\small211\ \ }(12,15)
\Label\lo{\small21\ }(9,15)
\Label\lo{\small2}(6,15)
\Label\lo{\small1}(3,15)
\Label\lo{\small\emptyset}(0,15)
\Label\ro{\small211\ \ }(12,18)
\Label\lo{\small21\ }(9,18)
\Label\lo{\small2}(6,18)
\Label\lo{\small1}(3,18)
\Label\lo{\small\emptyset}(0,18)
\Label\ro{\small21\ \ }(9,21)
\Label\lo{\small2}(6,21)
\Label\lo{\small1}(3,21)
\Label\lo{\small\emptyset}(0,21)
\Label\ro{\small2}(6,24)
\Label\lo{\small1}(3,24)
\Label\lo{\small\emptyset}(0,24)
\Label\ro{\small1}(3,27)
\Label\lo{\small\emptyset}(0,27)
\Label\lo{\small\emptyset}(0,30)
\hskip12.3cm
$$
\caption{Example for the bijection between $\NCNN_{10}(r,s)$ and
  $\NCNN_{10}(s,r)$ --- forward algorithm.}
\label{fig:5}
\end{figure}

We illustrate this construction by considering the matching
$$
\Big\{
(1,10),\ (2,6),\ (3,8),\ (4,9)
\Big\}
$$
with fixed points 5 and 7. It has a 3-crossing --- namely
$\big\{(2,6),\ (3,8),\ (4,9)\big\}$ ---, a $(3+1/2)$-crossing together
with the fixed point~5, several 2-nestings --- namely e.g.\ 
$\big\{(1,10),\ (3,8)\big\}$ ---, and a $(2+1/2)$-nesting together
with the fixed point~7. The corresponding arrangement of crosses
together with the forward growth diagram construction is shown in
Figure~\ref{fig:5}.

\begin{figure}
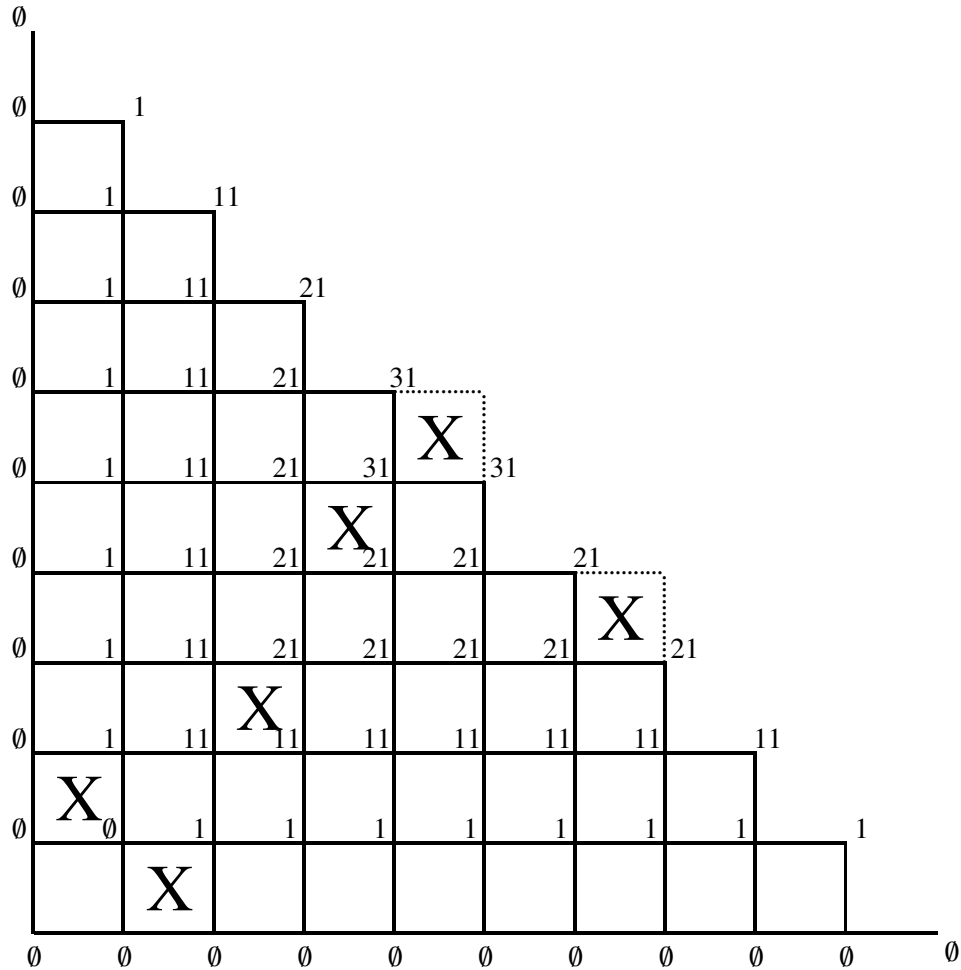

$$
\Einheit.4cm
\Pfad(0,0),111111111111111111111111111111\endPfad
\Pfad(0,3),111111111111111111111111111\endPfad
\Pfad(0,6),111111111111111111111111\endPfad
\Pfad(0,9),111111111111111111111\endPfad
\Pfad(0,12),111111111111111111\endPfad
\Pfad(0,15),111111111111111\endPfad
\Pfad(0,18),111111111111\endPfad
\Pfad(0,21),111111111\endPfad
\Pfad(0,24),111111\endPfad
\Pfad(0,27),111\endPfad
\Pfad(0,0),222222222222222222222222222222\endPfad
\Pfad(3,0),222222222222222222222222222\endPfad
\Pfad(6,0),222222222222222222222222\endPfad
\Pfad(9,0),222222222222222222222\endPfad
\Pfad(12,0),222222222222222222\endPfad
\Pfad(15,0),222222222222222\endPfad
\Pfad(18,0),222222222222\endPfad
\Pfad(21,0),222222222\endPfad
\Pfad(24,0),222222\endPfad
\Pfad(27,0),222\endPfad
%
\SPfad(21,9),222555\endSPfad
\SPfad(15,15),222555\endSPfad
\PfadDicke{2pt}
\Label\ro{\text {\Huge X}}(1,4)
\Label\ro{\text {\Huge X}}(4,1)
\Label\ro{\text {\Huge X}}(7,7)
\Label\ro{\text {\Huge X}}(10,13)
\Label\ro{\text {\Huge X}}(13,16)
\Label\ro{\text {\Huge X}}(19,10)
\Label\u{\small\emptyset}(0,0)
\Label\u{\small\emptyset}(3,0)
\Label\u{\small\emptyset}(6,0)
\Label\u{\small\emptyset}(9,0)
\Label\u{\small\emptyset}(12,0)
\Label\u{\small\emptyset}(15,0)
\Label\u{\small\emptyset}(18,0)
\Label\u{\small\emptyset}(21,0)
\Label\u{\small\emptyset}(24,0)
\Label\u{\small\emptyset}(27,0)
\Label\ru{\small\emptyset}(30,0)
\Label\ro{\small1}(27,3)
\Label\lo{\small1}(24,3)
\Label\lo{\small1}(21,3)
\Label\lo{\small1}(18,3)
\Label\lo{\small1}(15,3)
\Label\lo{\small1}(12,3)
\Label\lo{\small1}(9,3)
\Label\lo{\small1}(6,3)
\Label\lo{\small\emptyset}(3,3)
\Label\lo{\small\emptyset}(0,3)
\Label\ro{\small11\ }(24,6)
\Label\lo{\small11\ }(21,6)
\Label\lo{\small11\ }(18,6)
\Label\lo{\small11\ }(15,6)
\Label\lo{\small11\ }(12,6)
\Label\lo{\small11\ }(9,6)
\Label\lo{\small11\ }(6,6)
\Label\lo{\small1}(3,6)
\Label\lo{\small\emptyset}(0,6)
\Label\ro{\ \small21}(21,9)
\Label\lo{\small21\ }(18,9)
\Label\lo{\small21\ }(15,9)
\Label\lo{\small21\ }(12,9)
\Label\lo{\small21\ }(9,9)
\Label\lo{\small11\ }(6,9)
\Label\lo{\small1}(3,9)
\Label\lo{\small\emptyset}(0,9)
\Label\ro{\small21\ }(18,12)
\Label\lo{\small21\ }(15,12)
\Label\lo{\small21\ }(12,12)
\Label\lo{\small21\ }(9,12)
\Label\lo{\small11\ }(6,12)
\Label\lo{\small1}(3,12)
\Label\lo{\small\emptyset}(0,12)
\Label\ro{\ \small31}(15,15)
\Label\lo{\small31\ }(12,15)
\Label\lo{\small21\ }(9,15)
\Label\lo{\small11\ }(6,15)
\Label\lo{\small1}(3,15)
\Label\lo{\small\emptyset}(0,15)
\Label\ro{\small31\ \ }(12,18)
\Label\lo{\small21\ }(9,18)
\Label\lo{\small11\ }(6,18)
\Label\lo{\small1}(3,18)
\Label\lo{\small\emptyset}(0,18)
\Label\ro{\small21\ \ }(9,21)
\Label\lo{\small11\ }(6,21)
\Label\lo{\small1}(3,21)
\Label\lo{\small\emptyset}(0,21)
\Label\ro{\small11\ }(6,24)
\Label\lo{\small1}(3,24)
\Label\lo{\small\emptyset}(0,24)
\Label\ro{\small1}(3,27)
\Label\lo{\small\emptyset}(0,27)
\Label\lo{\small\emptyset}(0,30)
\hskip12.3cm
$$
\caption{Example for the bijection between $\NCNN_{10}(r,s)$ and
  $\NCNN_{10}(s,r)$ --- backward algorithm.}
\label{fig:5b}
\end{figure}

Figure~\ref{fig:5b} shows the growth diagram that results from
conjugating the partitions along the right-down diagonal and
subsequently applying the inverse growth diagram algorithm.
The resulting arrangement of crosses corresponds to the matching
$$
\Big\{
(1,9),\ (2,10),\ (3,8),\ (4,6)
\Big\}
$$
with the (same) fixed points 5 and 7.
Indeed, this matching has a 3-nesting --- namely
$\big\{(2,10),\ (3,8),\ (4,6)\big\}$ ---, a $(3+1/2)$-nesting together
with the fixed point~5, a 2-crossing --- namely 
$\big\{(1,9),\ (2,10)\big\}$ ---, and a $(2+1/2)$-crossing together
with the fixed point~7.

This completes the proof.
\end{proof}

\bibliographystyle{abbrv}

\end{document}